\documentclass[11pt]{amsart}
\usepackage{amssymb,amsmath,amsthm}
\usepackage{amsfonts}
\usepackage{enumerate}
\usepackage{dsfont}
\usepackage{cite}
\usepackage[colorlinks, citecolor=red]{hyperref}
\usepackage{mathrsfs}
\usepackage{epsfig}
\usepackage{lscape}
\usepackage{subfigure}
\usepackage{epstopdf}
\usepackage{caption}
\usepackage{algorithm}
\usepackage{algorithmic}
\usepackage{multirow}
\usepackage{geometry}
\usepackage{paralist}
\usepackage{enumerate}
\usepackage{url}
\usepackage{graphicx,cite}
\usepackage{longtable}
\usepackage{tikz}

\newtheorem{definition}{Definition}[section]
\newtheorem{corollary}[definition]{Corollary}
\newtheorem{prop}[definition]{Proposition}
\newtheorem{theorem}[definition]{Theorem}
\newtheorem{lemma}[definition]{Lemma}

\newtheorem{remark}[definition]{Remark}

\newtheorem{assumption}{Assumption}[section]
\date{}

\begin{document}
\baselineskip 18pt
\bibliographystyle{plain} 

\title[ASHBM for linear systems]{On adaptive stochastic heavy ball momentum for solving linear systems}

\author{Yun Zeng}
\address{School of Mathematical Sciences, Beihang University, Beijing, 100191, China. }
\email{zengyun@buaa.edu.cn}

\author{Deren Han}
\address{LMIB of the Ministry of Education, School of Mathematical Sciences, Beihang University, Beijing, 100191, China. }
\email{handr@buaa.edu.cn}

\author{Yansheng Su}
\address{School of Mathematical Sciences, Beihang University, Beijing, 100191, China. }
\email{suyansheng@buaa.edu.cn}

\author{Jiaxin Xie}
\address{LMIB of the Ministry of Education, School of Mathematical Sciences, Beihang University, Beijing, 100191, China. }
\email{xiejx@buaa.edu.cn}

\begin{abstract}
The stochastic heavy ball momentum (SHBM) method has gained considerable popularity as a scalable approach for solving large-scale optimization problems. However, one limitation of this method is its reliance on prior knowledge of certain problem parameters, such as singular values of a matrix. In this paper,  we propose an adaptive variant of the SHBM method for solving stochastic problems that are reformulated from linear systems  using user-defined distributions.  Our adaptive SHBM (ASHBM) method utilizes iterative information to update the parameters, addressing an open problem in the literature regarding the adaptive learning of momentum parameters. We prove that our method converges linearly in expectation, with a better convergence bound compared to  the basic  method. Notably, we demonstrate that the deterministic version of our ASHBM algorithm can be reformulated as a variant of the conjugate gradient (CG) method, inheriting many of its appealing properties, such as finite-time convergence. Consequently, the ASHBM method can be further generalized to develop a brand-new framework of the stochastic CG (SCG) method for solving linear systems. Our theoretical results are supported by numerical experiments.
\end{abstract}

\maketitle

\let\thefootnote\relax\footnotetext{Key words: 
	Linear systems, stochastic methods, heavy ball momentum, adaptive strategy, conjugate gradient}

\let\thefootnote\relax\footnotetext{Mathematics subject classification (2020): 65F10, 65F20, 90C25, 15A06, 68W20}


\section{Introduction}
The problem of solving linear systems is a fundamental problem that has a wide range of applications in numerous contexts, including signal processing \cite{Byr04}, optimal control \cite{Pat17}, machine learning \cite{Cha08}, and partial differential equations \cite{Ols14}.
With the advent of the age of big data,  it becomes increasingly desirable to design efficient algorithms for solving linear systems of unprecedented sizes. This has led to a surge of interest in randomized iterative methods,  such as the randomized Kaczmarz (RK) method \cite{Str09,Kac37} and the stochastic gradient descent (SGD) method \cite{robbins1951stochastic},  due to their potential for solving large-scale problems.
Indeed, the RK method can be viewed as a variant of the SGD method \cite{Han2022-xh,Needell2016-mh,needell2014stochasticMP,ma2017stochastic,zeng2022randomized}.

In recent years, the momentum acceleration technique has been recognized as an effective approach to improving the performance of optimization methods. One example is Polyak's heavy ball momentum method \cite{polyak1964some,ghadimi2015global}, and a fruitful line of research has been dedicated to extending this acceleration technique to enhance the performance of the SGD method \cite{loizou2020momentum,barre2020complexity,sebbouh2021almost,han2022pseudoinverse}. However, the resulting stochastic heavy ball momentum  (SHBM) method has a disadvantage that it requires  the prior information of some problem parameters, such as the singular values of the coefficient matrix \cite{loizou2020momentum,han2022pseudoinverse}. Therefore, an adaptive variant of the SHBM method that learns the parameters adaptively would be especially beneficial.  In this paper, we investigate the adaptive SHBM (ASHBM) method  for solving linear systems, which updates the parameters using readily available information.

We consider the consistent linear system
\begin{equation}
	\label{LS}
	Ax=b, \ A\in\mathbb{R}^{m\times n}, \ b\in\mathbb{R}^m,
\end{equation}
and  its reformulation into the following stochastic optimization problem
\begin{equation}
	\label{SOP1}
	\min\limits_{x \in \mathbb{R}^n} f(x):=\mathbb{E} \left[f_{S}(x)\right],
\end{equation}
where the expectation is taken  over random matrices $S$ drawn from a probability space $(\Omega, \mathcal{F}, P)$.
The function $f_{S}(x)$ is a stochastic convex quadratic function of a least-squares type, defined as
$$f_{S}(x):=\frac{1}{2} \left\|S^\top(Ax-b)\right\|^2_2, $$
where  $\top$ denotes the transpose of either a vector or a matrix. Naturally, the  stochastic problem \eqref{SOP1} will be  designed to be \emph{exact},  i.e. the set of minimizers of the problem \eqref{SOP1} is identical to the set of solutions of the linear system \eqref{LS}.

In this paper, we will consider  the  SHBM method to solve the optimization problem \eqref{SOP1}. At each iteration,  the SHBM method draws a random matrix $S_k \in \Omega$ and  employs the following update formula
$$x^{k+1}=x^k-\alpha_k\nabla f_{S_k}(x^k)+\beta_k(x^k-x^{k-1}),$$
where $\nabla$ is the gradient operator, $\alpha_k$ is the step-size, and $\beta_k$ is the momentum parameter. Note that $\nabla f_{S_k}(x^k)$ can be regarded as a randomized estimate of the full gradient $\nabla f(x)$. When $\beta_k=0$, the SHBM method reduces to the SGD method.
A typical case is the RK method \cite{Han2022-xh,Needell2016-mh,needell2014stochasticMP,ma2017stochastic,zeng2022randomized}. For any $ i\in\{1,2,\cdots,m\} $, let $ e_i $ denote the $ i $-th unit vector, $ A_{i, :} $ denote the $ i $-th row of $ A $, and $ b_i $ denote the $ i $-th entry of $ b $. Given any positive \textit{weights} $ \{w_i\}_{i=1}^m $, considering $ \Omega=\left\{\frac{e_i}{\sqrt{w_i}}\right\}_{i=1}^m$ with $ \frac{e_i}{\sqrt{w_i}} $ being sampled with probability $  \frac{w_i}{\sum_{i=1}^{m} w_i}$ at the $ k $-th iteration, then we obtain the following iteration scheme 
\[
x^{k+1}=x^k-\alpha_k\frac{A_{i_k, :}x^k-b_{i_k}}{w_{i_k} }A_{i_k, :}^\top,
\]	
which is a weighted variant of the RK method. See Remark \ref{remark-xie-0429-1}, Remark \ref{remark-3.5}, and Section \ref{NS-Section61} for further discussion on the methods that SHBM can recover.

\subsection{Our contributions}

In this paper, we present a generic algorithmic framework for solving the  stochastic optimization problem \eqref{SOP1}.  The main contributions of this work are as follows.

\begin{itemize}
	\item[1.] We develop a novel scheme for reformulating any consistent linear system into a stochastic problem and establish  conditions under which  the reformulation is exact. Particularly, instead of relying on a fixed probability space $(\Omega, \mathcal{F}, P)$, we utilize a specific class of probability spaces $\{(\Omega_k, \mathcal{F}_k, P_k)\}_{k\geq0}$ to generate the randomized matrix $S_k$ at each iteration. By manipulating the probability spaces, our  algorithmic framework can recover a wide range of popular algorithms, including  the RK method and its variants. Furthermore, it also enables us to design more versatile hybrid algorithms with improved performance, accelerated convergence, and better scalability.
	
	\item[2.] The SHBM method  has attracted much attention in recent years  due to its ability to improve the convergence of the SGD method. However, the preferable parameters $\alpha_k$ and $\beta_k$ for the SHBM method may rely on certain problem parameters that are generally inaccessible. For instance, the choices of $\alpha_k$ and $\beta_k$  for the SHBM method for solving the linear system $Ax=b$ require knowledge of the largest singular value as well as the smallest nonzero singular value of matrix $A$ \cite{loizou2020momentum,polyak1964some,ghadimi2015global,bollapragada2022fast}. Hence, it is an open problem whether one can  design a theoretically supported adaptive SHBM method \cite{barre2020complexity,bollapragada2022fast}.  In this paper, we give a positive answer to the above problem for a class of stochastic problems \eqref{SOP1}, developing the adaptive  SHBM (ASHBM) method.
	We provide an intuitive geometric interpretation that our method can be regarded as an orthogonal projection method. Along with a requirement for the sampling matrices, this interpretation exhibits how we adaptively learn the parameters $\alpha_k$ and $\beta_k$ in a simple and effective manner.
	
	\item[3.] Besides, although the SHBM method may be effective in practice, existing results show that the convergence factor of the expected error of the SHBM method for solving linear systems is inferior to that of the SGD method \cite{loizou2020momentum,han2022pseudoinverse}.  Therefore, it is an open problem whether one can prove a better factor for SHBM compared to SGD \cite[Section 4.1]{loizou2020momentum}. We also give a positive answer to this problem with our adaptive parameters $\alpha_k$ and $\beta_k$. Indeed, we  demonstrate that our approach not only simplifies the analysis but also endows the proposed method with an improved convergence factor.
	
	\item[4.] The deterministic version of our method is coincidentally equivalent to a variant of the conjugate gradient (CG) method,  which equips it with many of the attractive features of the CG method. This equivalence also inspires a novel perspective that ASHBM  can be regarded as a stochastic CG (SCG) method. To the best of our knowledge, this is the first time that such SCG procedure has been investigated. The efficiency and rapid convergence of the CG method in solving large-scale sparse systems may further explain the effectiveness of our method. The conclusions  are validated by the results of our numerical experiments.
\end{itemize}

\subsection{Related work}
\label{REL}

\subsubsection{Kaczmarz method}
The Kaczmarz method \cite{Kac37}, also known as the algebraic reconstruction technique (ART) \cite{herman1993algebraic,gordon1970algebraic}, is a classical iterative algorithm used to solve the large-scale linear system of equations \eqref{LS}. The method alternates between choosing a row of the system and updating the solution based on the projection onto the hyperplane defined by that row.
There is empirical evidence that selecting of the rows in a random order rather than a deterministic order can often accelerate the convergence of the Kaczmarz method \cite{herman1993algebraic,natterer2001mathematics,feichtinger1992new}. In the seminal paper \cite{Str09}, Strohmer and Vershynin studied the randomized  Kaczmarz (RK) method and proved that if linear system \eqref{LS} is consistent, then RK converges linearly  in expectation. Subsequently, there is a large amount of work on the development of the Kaczmarz-type methods including accelerated randomized Kaczmarz methods \cite{liu2016accelerated,han2022pseudoinverse,loizou2020momentum}, block Kaczmarz methods \cite{Nec19,needell2013two,needell2014paved,moorman2021randomized,Gow15}, greedy randomized Kaczmarz methods \cite{Bai18Gre,Gow19}, randomized sparse Kaczmarz methods \cite{schopfer2019linear,chen2021regularized}, etc. The connection between our method and the RK method is discussed in Remark \ref{remark-xie-0429-1}.

Moreover, we note that the RK method has also been adapted to handle inconsistent linear systems.  Needell \cite{needell2010randomized,needell2014stochasticMP} showed that RK applied to \emph{inconsistent} linear systems converges only to within a radius (\emph{convergence horizon}) of the least-squares solution. In a fruitful line of research, the extensions of the Kaczmarz method have been proposed for solving  inconsistent systems, which include randomized extended Kaczmarz (REK) method \cite{Zou12,Du19}, its block or deterministic variants \cite{wu2021semiconvergence,DS2021,Du20Ran,wu2022two,wu2022extended,popa1998extensions,popa1999characterization,bai2019partially}, greedy randomized augmented Kaczmarz (GRAK) method \cite{bai2021greedy}, RK with adaptive step-sizes (RKAS) \cite{zeng2022randomized}, randomized extended Gauss-Seidel (REGS) method \cite{Du19,ma2015convergence},  etc.

\subsubsection{ Polyak step-size}

Recall that the SGD method for solving the finite-sum problem $\min\frac{1}{m}\sum_{i=1}^mf_i(x)$ utilizes the update $x^{k+1}=x^k-\alpha_k\nabla f_{i_k}(x^k)$, where $\alpha_k$ is the step-size and $i_k$ is selected randomly. In recent years, there has been a surge of interest in the development and analysis of SGD. Despite these efforts, identifying adaptive step-sizes that do not require hyperparameter tuning remains a challenging problem. In this context, the stochastic Polyak step-size
\begin{equation}\label{polyak-stepsize}
	\alpha_k=\frac{f_{i_k}(x^k)-\hat{f}_{i_k}}{c\cdot\left\|\nabla f_{i_k}(x^k)\right\|^2_2}
\end{equation}
was proposed in \cite{loizou2021stochastic}, where $c>0$ is a fixed constant and $\hat{f}_i=\inf f_i$.
For convex lower bounded functions $f_i$, if there  exists $\hat{x}\in\mathbb{R}^n$ such that $f_i(\hat{x})=\hat{f}_i$ for all $i=1,\ldots,m$, i.e. the \emph{interpolation} condition holds, it has been proven that the iterates of this method converge. Although this assumption seems restrictive, it can be satisfied under certain circumstances, e.g. the stochastic optimization problem \eqref{SOP1}. Actually, the stochastic Polyak step-size will be employed in the basic method studied in this paper, see Remark \ref{remark-xie-0429}.

In a recent work \cite{barre2020complexity}, the authors wondered whether Polyak step-sizes could be combined with adaptive momentum to achieve accelerated first-order methods. For convex quadratic minimization, Goujaud et al. \cite{goujaud2022quadratic} introduced an adaptive variant of the heavy ball momentum (HBM) method, providing a positive answer to this question. They showed that their method is equivalent to a variant of the classical conjugate gradient (CG) method. In this paper, we affirmatively answer the problem for a class of stochastic problems \eqref{SOP1}. Furthermore, we introduce  a novel framework of the stochastic CG (SCG) method for solving linear systems.

\subsubsection{Heavy ball momentum method}

The HBM method introduced in $1964$ by Polyak  \cite{polyak1964some} is a modification of the classical gradient descent (GD) method. Originally, for twice continuously differentiable, strongly convex, and Lipschitz continuous gradient functions, a local accelerated linear convergence rate of the HBM method was established with appropriate parameters \cite{polyak1964some}. Only recently, the authors of \cite{ghadimi2015global} proved a global convergence of the HBM method for smooth and convex functions. Inspired by the success of the HBM method, several recent studies have extended this acceleration technique to enhance the SGD method, resulting in the development of the stochastic HBM (SHBM) method \cite{loizou2020momentum,barre2020complexity,sebbouh2021almost,han2022pseudoinverse,P2017Stochastic,
	loizou2021revisiting,morshed2020stochastic}.

Our result is closely related to the work of Loizou and Richt{\'a}rik \cite{loizou2020momentum}, where the SHBM method  for solving specialized convex quadratic problems was studied. They integrated several momentum variants of algorithms, such as the randomized Kaczmarz method and the randomized coordinate descent method,  into one framework.
They proved \cite[Theorem $1$]{loizou2020momentum} the linear convergence of the expected error  of SHBM. However, the convergence factor is based on a specific time-invariant momentum parameter that is related to the singular values of the matrix $A$. Moreover, the convergence factor is weaker than that of the SGD method. These limitations hinder the potential of the SHBM method and Loizou and Richt{\'a}rik \cite[Section 4.1]{loizou2020momentum} proposed an open problem: is it possible to prove a convergence factor for SHBM that is  better than that of the SGD  method? Bollapragada et al. \cite{bollapragada2022fast} recently investigated the accelerated convergence of the SHBM method for quadratic problems using standard momentum step-sizes. They demonstrated that with a sufficiently large batch size, the SHBM method maintains the fast linear rate of the HBM method. Besides, they also raised the question of how to adaptively learn the parameters $\alpha_k$ and $\beta_k$ to achieve state-of-the-art performance. This paper positively answers  the aforementioned problems.

Finally, we note that in \cite{saab2022adaptive}, the authors proposed an adaptive (deterministic) HBM (AHBM) method that employs the absolute differences between current and previous model parameters, as well as their gradients. The authors showed that for objective functions that are both smooth and strongly convex, AHBM guarantees a global linear convergence rate. Furthermore, they also extended their approach to the SHBM method and established its almost sure convergence, but without  offering convergence bound. Nevertheless, the approach proposed in this paper, which can be viewed as an orthogonal projection method, is distinct from that of  \cite{saab2022adaptive}. In addition, our method converges linearly in expectation.

\subsection{Organization}

The remainder of the paper is organized as follows. After introducing some notations and preliminaries in Section $2$, we present and analyze the basic method in Section $3$. In Section 4, we propose the ASHBM method and show its improved linear convergence rate. The relationship between the ASHBM method and the CG method is established in Section 5.  In Section 6, we perform some numerical experiments to show the effectiveness of the proposed method. Finally, we conclude the paper in Section 7.

\section{Notation and Preliminaries}

\subsection{Notations}
For any matrix $A \in \mathbb{R}^{m \times n}$, we use $A_{i,:}$, $A^\top$, $A^\dag$, $\|A\|_F$, $\text{Range}(A)$,  and $\text{Null}(A)$ to denote the $i$-th row, the transpose, the Moore-Penrose pseudoinverse, the Frobenius norm, the column space, and the null space of $A$, respectively. We use $\sigma_{\max}(A)$ and $\sigma_{\min}(A)$ to denote the largest and the smallest nonzero singular value of $A$, respectively, and we use $\lambda_{\max}(A^\top A)$ to denote  the largest eigenvalues of $A^\top A$.
Given $\mathcal{J} \subseteq [m]:=\{1,\ldots,m\}$, the cardinality of the set $\mathcal{J} $ is denoted by $|\mathcal{J} |$ and the complementary set of $\mathcal{J} $ is denoted by $\mathcal{J} ^c$, i.e. $\mathcal{J} ^c=[m]\setminus \mathcal{J}$.
We use $A_{\mathcal{J} ,:}$ and $A_{:,\mathcal{J} }$ to denote the row and column submatrix indexed by $\mathcal{J} $, respectively.
We use $ \operatorname{diag}(d_i,i\in \mathcal{J}) $ to represent the diagonal matrix consists of elements from the set $ \{ d_i \mid i \in \mathcal{J} \} $.
For any vector $b \in \mathbb{R}^m$, we use $b_i$ and $\|b\|_2$ to denote the $i$-th  entry and the Euclidean norm of $b$, respectively. The identity matrix is denoted by $I$. We use $\mathbb{N}$ to denote the set of natural numbers. For any random variables $\xi_1$ and $\xi_2$, we use $\mathbb{E}[\xi_1]$ and $\mathbb{E}[\xi_1\lvert \xi_2]$ to denote the expectation of $\xi_1$ and the conditional expectation of $\xi_1$ given $\xi_2$. The indicator function $\mathbb{I}_{\mathcal{C}}(\cdot)$ of the set $\mathcal{C}$ is defined as
\begin{equation}
	\nonumber
	\mathbb{I}_{\mathcal{C}} (x) = 
	\left\{\begin{array}{ll}
		1, \;\; \text{if} \; x \in \mathcal{C};
		\\
		0, \;\; \text{otherwise}.
	\end{array}
	\right.
\end{equation}

\subsection{Exactness of the reformulation}

In this section,  we investigate the condition such that the set of minimizers of the problem \eqref{SOP1} is identical to the set of solutions of the linear system \eqref{LS}. In order to proceed, we shall propose a basic assumption on the probability space $(\Omega, \mathcal{F}, P)$ used in this paper.

\begin{assumption}\label{assump1}
	Let $S$ be a random variable on the probability space $(\Omega, \mathcal{F}, P)$, we assume that the matrix $\mathbb{E}\left[SS^\top\right]$ has finite entries.
\end{assumption}

Note that if the assumption holds, then $\mathbb{E}\left[SS^\top\right]$ is symmetric and positive semidefinite. We shall enforce this assumption throughout the paper and therefore will not refer to it henceforth.
For the equivalence of \eqref{SOP1} and \eqref{LS}, we have the following result.

\begin{lemma}
	\label{iif-lemma}
	The set of minimizers of the stochastic reformulation \eqref{SOP1} is identical to the set of solutions of the linear system \eqref{LS} if and only if $\mbox{Null}\left(A^\top\mathbb{E}\left[SS^\top\right]A\right)=\mbox{Null}(A)$.
\end{lemma}
\begin{proof}
	Pick any $x^*$ such that $Ax^*=b$. Then we know that $x$ is a minimizer of \eqref{SOP1} if and only if
	$$
	\nabla f(x)=0\Leftrightarrow A^\top\mathbb{E}\left[SS^\top\right](Ax-b)=0\Leftrightarrow A^\top\mathbb{E}\left[SS^\top\right]A(x-x^*)=0,
	$$
	i.e. $x\in x^*+\text{Null}\left(A^\top\mathbb{E}\left[SS^\top\right]A\right)$. On the other hand, we know that $\tilde{x}$ is a solution of the linear system if and only if
	$
	A\tilde{x}=b \Leftrightarrow A(\tilde{x}-x^*)=0,
	$
	i.e. $\tilde{x}\in x^*+\mbox{Null}(A)$.  This implies \eqref{SOP1} and \eqref{LS} have  the same set of solutions if and only if $\mbox{Null}\left(A^\top\mathbb{E}\left[SS^\top\right]A\right)=\mbox{Null}(A)$.
\end{proof}

We now present a sufficient condition for the equivalence.
\begin{lemma}\label{lemma-suf}
	If $\mathbb{E}\left[SS^\top\right]$ is a positive definite matrix, then  \eqref{SOP1} and \eqref{LS} have  the same set of solutions.
\end{lemma}
\begin{proof}
	If $\mathbb{E}\left[SS^\top\right]$ is a positive definite matrix, then $\mbox{Null}\left(A^\top\mathbb{E}\left[SS^\top\right]A\right)=\mbox{Null}(A)$. By applying Lemma \ref{iif-lemma}, we can arrive at this lemma.
\end{proof}

Let $\left\{(\Omega_k, \mathcal{F}_k, P_k)\right\}_{k\geq0}$  be the class of probability spaces that are used in this paper, we make the following assumption on those probability spaces.

\begin{assumption}
	\label{Ass}
	Let $\left\{(\Omega_k, \mathcal{F}_k, P_k)\right\}_{k\geq0}$  be  probability spaces from which the sampling matrices are drawn. We assume that  for any $k\geq0$, $\mathop{\mathbb{E}}_{S\in\Omega_k} \left[S S^\top\right]$ is a positive definite matrix.
\end{assumption}

It follows from Lemma \ref{lemma-suf} that Assumption \ref{Ass} is a sufficient condition such that the stochastic reformulation is exact for any $k\geq0$.

\subsection{Some useful lemmas}

In this subsection, we will present some useful lemmas.

\begin{lemma}
	\label{lemma-non}
	Assume that the linear system $Ax=b$ is consistent. Then for any matrix $S \in \mathbb{R}^{m\times q}$ and any vector $\tilde{x} \in \mathbb{R}^{n}$, it holds that $A^\top SS^\top(A\tilde{x}-b) \neq 0$ if and only if $S^\top(A\tilde{x}-b) \neq 0$.
\end{lemma}
\begin{proof}
	Suppose that $Ax^*=b$, then we know that $S^\top(A\tilde{x}-b) =0$ if and only if
	$$
	(A\tilde{x}-b)^\top S S^\top(A\tilde{x}-b)=(\tilde{x}-x^*)^\top A^\top S S^\top A(\tilde{x}-x^*)=0,
	$$
	which is equivalent to $A^\top S S^\top A(\tilde{x}-x^*)=A^\top SS^\top(A\tilde{x}-b)=0$. This completes the proof of this lemma.
\end{proof}

\begin{lemma}
	\label{xie-empty}
	Let matrix $S$ be a random variable such that $
	\mathbb{E}\left[SS^\top\right]=D
	$ is positive definite.
	Then $S^\top(Ax-b)=0$ for all $S$ holds if and only if $Ax=b$.
\end{lemma}

\begin{proof}
	If $S^\top(Ax-b)=0$ for all $S$, then
	$
	0=\mathbb{E}\left[\|S^\top(Ax-b)\|^2_2\right]=(Ax-b)^\top D(Ax-b).
	$
	Since  $D$ is positive definite, we know that $Ax-b=0$. On the other hand, if $Ax=b$, then it holds that $S^\top(Ax-b)=0$ for all $S$.
\end{proof}

\begin{lemma}
	\label{positive}
	Let $S\in\mathbb{R}^{m\times q}$ be a real-valued random variable defined on a probability space $(\Omega,\mathcal{F},P)$. Suppose that
	$
	D=\mathbb{E}\left[SS^\top\right]
	$
	is a positive definite matrix.  Then
	$$
	\mathbb{E}\left[\frac{SS^\top}{\|S\|^2_2}\right]
	$$
	is well-defined and positive definite, here we define $\frac{0}{0}=0$.
\end{lemma}
\begin{proof}
	Since $D$ is positive definite, then for any $x\in\mathbb{R}^m$ with $\|x\|_2=1$, there exists a $\varepsilon_0>0$ such that
	$$
	x^\top Dx=\mathbb{E}\left[x^{\top}SS^{\top}x \right] \geq \varepsilon_0>0.
	$$
	In addition, due to $x^\top Dx$ is bounded and $x^\top SS^\top x$ is nonnegative, there exists a $c:=c_{\varepsilon_0} > 0$ such that
	\begin{equation}
		\label{psd0-}
		\begin{aligned}
			\mathbb{E}\left[x^{\top}SS^{\top}x \cdot \mathbb{I}_{\{S \mid \Vert S \Vert_2^2 \leq c\}}(S) \right] > \frac{\varepsilon_0}{2}.
		\end{aligned}
	\end{equation}
	Note that
	$$
	x^\top \int_{\Omega} \frac{SS^\top}{\|S\|^2_2} dP \cdot x= \int_{\Omega} \frac{ \Vert S^\top x \Vert_2^2}{\|S\|^2_2} dP \leq \int_{\Omega} \frac{ \Vert S\Vert_2^2 \Vert x\Vert_2^2}{\|S\|^2_2} dP=\int_{\Omega} \Vert x\Vert_2^2 dP=1,
	$$
	which implies $\int_{\Omega} \frac{SS^\top}{\|S\|^2_2} dP < \infty$, hence $\mathbb{E}\left[\frac{SS^\top}{\|S\|^2_2}\right]$ is well-defined. 
	Furthermore,
	$$
	x^\top\mathbb{E}\left[\frac{SS^\top}{\|S\|^2_2}\right]x \geq \mathbb{E}\left[\frac{x^\top SS^\top x}{\|S\|^2_2} \cdot \mathbb{I}_{\{S| \Vert S \Vert_2^2 \leq c\}}(S) \right] \geq\frac{1}{c} \mathbb{E}\left[x^{\top}SS^{\top}x \cdot \mathbb{I}_{\{S| \Vert S \Vert_2^2 \leq c\}}(S) \right] >\frac{\varepsilon_0}{2c} >0,
	$$
	where the third inequality follows from \eqref{psd0-}. Thus, we can get that $\mathbb{E}\left[\frac{SS^\top}{\|S\|^2_2}\right]$ is positive definite. This completes the proof of this lemma.
\end{proof}

\section{Basic method}
\label{basic-m}

In this section, we  focus on utilizing the stochastic gradient descent (SGD) method to solve the linear system \eqref{LS}. Instead of using a fixed probability space $(\Omega, \mathcal{F}, P)$, we will use a class of probability spaces $\{(\Omega_k, \mathcal{F}_k, P_k)\}_{k\geq0}$ for choosing sampling matrices. Specifically, at the $k$-th iteration, we consider the following stochastic reformulation of the linear system \eqref{LS}
$$\mathop{\min}\limits_{x \in \mathbb{R}^n} f^k(x):=\mathop{\mathbb{E}}\limits_{S\in\Omega_k}\left[ f_S(x)\right],$$
where $f_S(x):=\frac{1}{2} \left\|S^\top(Ax-b)\right\|^2_2$ with $S$ being a random variable in $(\Omega_k, \mathcal{F}_k, P_k)$. We can ensure the exactness of this reformulation due to Assumption \ref{Ass}. Starting from $x^k$, we employ only one step of the SGD method to solve the reformulated problem
$$
x^{k+1}=x^k-\alpha_k\nabla f_{S_k}(x^k)=x^k-\alpha_k A^\top S_k S_k^\top(Ax^k-b),
$$
where $S_k$ is drawn from the sample space $\Omega_k$ and $\alpha_k$ is the step-size. Particularly, we consider the following adaptive step-sizes
\begin{equation}
	\label{alp}
	\alpha_k=
	\left\{\begin{array}{ll}
		(2-\zeta_k)L_{\text{adap}}^{(k)}, & \text{if} \; S_k^\top (Ax^k-b)\neq 0;
		\\
		0, & \text{otherwise},
	\end{array}
	\right.
\end{equation}
where
\begin{equation}
	\label{Lk}
	L_{\text{adap}}^{(k)}=\frac{2 f_{S_k}(x^k)}{\|\nabla f_{S_k}(x^k)\|^2_2}=
	\frac{\|S_k^\top (Ax^k-b)\|_2^2}{\|A^\top S_k S_k^\top (Ax^k-b)\|_2^2}
\end{equation}
and $ \zeta_k \in (0,2) $ are relaxation parameters for adjusting the step-size.
From Lemma \ref{lemma-non}, we know that $L_{\text{adap}}^{(k)}$ is well-defined as $S_k^\top (Ax^k-b)\neq 0$ implies that $A^\top S_k S_k^\top (Ax^k$ $ -b) \neq 0$.
We emphasize that for the case $S_k^\top (Ax^k-b)= 0$, we set $\alpha_k$ to be zero since any constant value  can be chosen. This is because we now have $A^\top S_kS_k^\top(Ax^k-b)=0$, and thus, regardless of the choice of $\alpha_k$, it holds that $x^{k+1}=x^k$.
Now we are ready to state the basic method, which is formally described in Algorithm \ref{BSM}.

\begin{algorithm}[htpb]
	\caption{Basic  method}
	\label{BSM}
	\begin{algorithmic}
		\STATE{\textbf{Input:} $A \in \mathbb{R}^{m \times n}$, $b \in \mathbb{R}^m$, probability spaces $\{(\Omega_k, \mathcal{F}_k, P_k)\}_{k\geq0}$, $\{\zeta_k\}_{k \geq 0}$ with $\zeta_k \in (0, 2)$, $k=0$, and the initial point $x^0 \in \text{Range}(A^\top)$.}
		\begin{enumerate}
			\item[1:] Randomly select a sampling matrix $S_k\in \Omega_k$.
			\item[2:] If $S_k^\top (Ax^k-b)\neq 0$,
			
			\qquad Compute the parameter $L_{\text{adap}}^{(k)}$ in \eqref{Lk} and set $\alpha_k=(2-\zeta_k)L_{\text{adap}}^{(k)}$;
			
			\qquad Update $x^{k+1}=x^k-\alpha_k A^\top S_k S_k^\top (Ax^k-b)$.
			
			Otherwise,
			
			\qquad Set $x^{k+1}=x^k$.

			\item[3:] If the stopping rule is satisfied, stop and go to output. Otherwise, set $k=k+1$ and return to Step $1$.
		\end{enumerate}
		\STATE{\textbf{Output:} The approximate solution $x^k$.}
	\end{algorithmic}
\end{algorithm}

\begin{remark}\label{remark-xie-0429}
	It can be observed from \eqref{polyak-stepsize} that $L_{\text{adap}}^{(k)}$ in \eqref{Lk} is equivalent to the stochastic Polyak step-size. This implies that the step-size employed in Algorithm \ref{BSM} coincides with the stochastic Polyak step-size.
\end{remark}

\begin{remark}
	\label{remark-xie-0429-1}
	We consider the randomized average block Kaczmarz (RABK) method proposed by Necoara \cite{Nec19}.
	The RABK method leads to the following iteration:
	\begin{equation}\label{BKM}
		x^{k+1}=x^k-\alpha_k\bigg(\sum\limits_{i\in \mathcal{J} _k}\omega^{(k)}_i\frac{A_{i,:}x^k-b_i}{\|A_{i,:}\|^2_2}A_{i,:}^\top\bigg),
	\end{equation}
	where the weights $\omega^{(k)}_i\in[0,1]$ such that $\sum\limits_{i\in \mathcal{J} _k}\omega^{(k)}_i=1$,  $\mathcal{J} _k\subseteq[m]$, and $\alpha_k>0$ is the step-size.
	If $|\mathcal{J} _k|=1$, $\alpha_k=1$, and the index $i$ is selected with probability $\frac{\|A_{i,:}\|^2_2}{\|A\|^2_F}$, then RABK reduces to the classical RK method \cite{Str09}.
	Let $I_{:,\mathcal{J}_k}$
	denote a column concatenation of the columns of the $m\times m$ identity matrix $I$ indexed by $\mathcal{J} _k$, and the diagonal matrix  $D_{\mathcal{J} _k,:}=\text{diag}(\sqrt{\omega^{(k)}_i}/\|A_{i,:}\|_2,i\in \mathcal{J} _k)$.
	Then the iteration scheme \eqref{BKM} can alternatively be written as
	$$x^{k+1}=x^k-\alpha_k A^\top S_k S_k^\top (Ax^k-b),$$
	where $S_k = I_{:,\mathcal{J} _k}D_{\mathcal{J} _k,:}$, which can be viewed as a sampling matrix selected from  a certain probability space $(\Omega_k, \mathcal{F}_k, P_k)$. This means that the RABK method \eqref{BKM} can be viewed as a special case of our basic method.
\end{remark}

Furthermore, the versatility of our framework and the general convergence theorem (Theorem \ref{CRS-ASGD}) enable one to customize the probability spaces $\{(\Omega_k, \mathcal{F}_k, P_k)\}_{k\geq0}$ to other specific problems. For example, random sparse matrices or sparse Rademacher matrices may be suitable for a particular set of problems.

\subsection{Convergence analysis}

To establish the convergence of Algorithm \ref{BSM}, let us first introduce some notations. We define
\begin{equation}
	\label{matrix-H}
	H_k:=
	\left\{\begin{array}{ll}
		\mathop{\mathbb{E}}\limits_{S \in \Omega_k}[SS^\top], \quad\;\;\, \text{if} \; \Omega_k \; \text{is} \; \text{bounded};
		\\
		\\
		\mathop{\mathbb{E}}\limits_{S \in \Omega_k}\left[\frac{SS^\top}{\|S\|_2^2}\right], \;\;\, \ \text{otherwise},
	\end{array}
	\right.
\end{equation}
and
\begin{equation}
	\label{lambda-max}
	\lambda_{\max}^{(k)}:=
	\left\{\begin{array}{ll}
		\sup\limits_{S\in \Omega_k} \lambda_{\max}\left(A^\top SS^\top A\right), \;\;\,\ \text{if} \; \Omega_k \; \text{is} \; \text{bounded};
		\\
		\\
		\sup\limits_{S\in \Omega_k} \lambda_{\max}\left(\frac{A^\top SS^\top A}{\|S\|_2^2}\right), \quad\;\ \text{otherwise}.
	\end{array}
	\right.
\end{equation}
It follows from Assumption \ref{Ass} and Lemma \ref{positive} that $H_k$ in \eqref{matrix-H} is well-defined and  is a positive definite matrix. We define the set
\begin{equation}
	\label{xie-Qk}
	\mathcal{Q}_k:=\{S\in \Omega_k \mid S^\top (Ax^k-b)\neq 0\},
\end{equation}
where Algorithm \ref{BSM} indeed executes one step such that $x^{k+1}\neq x^k$ when $S_k \in \mathcal{Q}_k$. Additionally, we can observe that $\{\mathcal{Q}_k, \mathcal{Q}_k^c\}$ forms a partition of $\Omega_k$.

\begin{theorem}
	\label{CRS-ASGD}
	Suppose that the  linear system {\eqref{LS}} is consistent and the probability spaces $\{(\Omega_k, \mathcal{F}_k, P_k)\}_{k\geq 0}$ satisfy Assumption \ref{Ass}. Let $\{x^k\}_{k \geq 0}$ be the iteration sequence generated by Algorithm \ref{BSM}. Then
	$$\mathop{\mathbb{E}} \left[\| x^{k+1}-A^\dagger b \|_2^2 \  \big| \   x^k\right] \leq \left(1-\zeta_k(2-\zeta_k)\frac{\sigma_{\min}^{2} (H_k^{\frac{1}{2}}A)}{\lambda_{\max}^{(k)}}\right) \left\| x^k-A^\dagger b \right\|_2^2,
	$$
	where $H_k$ and $\lambda_{\max}^{(k)}$ are given by \eqref{matrix-H} and \eqref{lambda-max}, respectively.
	
\end{theorem}
\begin{proof}
	Let $\mathcal{Q}_k$ be defined as \eqref{xie-Qk} and if the sampling matrix $S_k \in \mathcal{Q}_k$, then we have
	\begin{equation}\label{descent-xie}
		\begin{aligned}
			\|x^{k+1}-A^{\dagger}b\|_2^2
			=&\|x^k-\alpha_k A^\top S_k S_k^\top (Ax^k-b)-A^{\dagger}b\|_2^2 \\
			=&\|x^k-A^{\dag}b\|_2^2-2\alpha_k \left\langle x^k-A^{\dagger}b, A^\top S_k S_k^\top (Ax^k-b) \right\rangle
			\\
			&+\alpha_k^2 \|A^\top S_k S_k^\top (Ax^k-b)\|_2^2 \\
			=&\|x^k-A^{\dag}b\|_2^2-2(2-\zeta_k)L_{\text{adap}}^{(k)} \|S_k^\top (Ax^k-b)\|_2^2
			\\
			&+(2-\zeta_k)^2L_{\text{adap}}^{(k)} \|S_k^\top (Ax^k-b)\|_2^2 \\
			=&\|x^k-A^{\dag}b\|_2^2-\zeta_k(2-\zeta_k) L_{\text{adap}}^{(k)} \|S_k^\top (Ax^k-b)\|_2^2,
		\end{aligned}
	\end{equation}
	where the third equality follows from the definition of $\alpha_k$ and $L_{\text{adap}}^{(k)}$ in \eqref{alp} and \eqref{Lk}, respectively. For convenience, we use $\mathbb{E}_{S_k \in \Omega_k}[\ \cdot \ ]$ to denote $\mathbb{E}_{S_k \in \Omega_k}[\ \cdot \mid x^k]$.
	Thus
	\begin{equation}
		\label{Exp}
		\begin{aligned}
			\mathop{\mathbb{E}}\limits_{S_k \in \Omega_k}\left[\|x^{k+1}-A^{\dag}b\|_2^2\right]
			=&\mathop{\mathbb{E}}\limits_{S_k \in \mathcal{Q}_k}\left[\|x^{k+1}-A^{\dag}b\|_2^2\right]+\mathop{\mathbb{E}}\limits_{S_k \in \mathcal{Q}_k^c}\left[\|x^{k+1}-A^{\dag}b\|_2^2\right] \\
			=&\mathop{\mathbb{E}}\limits_{S_k \in \mathcal{Q}_k}\left[\|x^k-A^{\dag}b\|_2^2-\zeta_k(2-\zeta_k) L_{\text{adap}}^{(k)} \|S_k^\top  (Ax^k-b)\|_2^2\right]
			\\
			&+\mathop{\mathbb{E}}\limits_{S_k \in \mathcal{Q}_k^c}[\|x^k-A^{\dag}b\|_2^2] \\
			=&\mathop{\mathbb{E}}\limits_{S_k \in \Omega_k}[\|x^k-A^{\dag}b\|_2^2]-\zeta_k(2-\zeta_k)\mathop{\mathbb{E}}\limits_{S_k \in \mathcal{Q}_k}\left[L_{\text{adap}}^{(k)} \|S_k^\top (Ax^k-b)\|_2^2 \right] \\
			=&\|x^k-A^{\dag}b\|_2^2-\zeta_k(2-\zeta_k)\mathop{\mathbb{E}}\limits_{S_k \in \mathcal{Q}_k}\left[L_{\text{adap}}^{(k)} \|S_k^\top (Ax^k-b)\|_2^2\right].
		\end{aligned}
	\end{equation}
	
	We consider the case where $\Omega_k$ is bounded. If $S_k \in \mathcal{Q}_k$, then we have
	$$L_{\text{adap}}^{(k)}=\frac{\|S_k^\top (Ax^k -b)\|_2^2}{\|A^\top S_k S_k^\top (Ax^k-b)\|_2^2}\geq \frac{1}{\lambda_{\max}(A^\top S_kS_k^\top A)}\geq \frac{1}{\lambda_{\max}^{(k)}}.$$
	Substituting it into \eqref{Exp},  we can get
	\begin{equation}
		\label{proof-xie1}
		\begin{aligned}
			\mathop{\mathbb{E}}\limits_{S_k \in \Omega_k}\left[\|x^{k+1}-A^{\dag}b\|_2^2\right]
			&\leq \|x^k-A^{\dag}b\|_2^2-\frac{\zeta_k(2-\zeta_k)}{\lambda_{\max}^{(k)}}\mathop{\mathbb{E}}\limits_{S_k \in \mathcal{Q}_k}\left[\|S_k^\top (Ax^k-b)\|_2^2\right]\\
			&= \|x^k-A^{\dag}b\|_2^2-\frac{\zeta_k(2-\zeta_k)}{\lambda_{\max}^{(k)}}\mathop{\mathbb{E}}\limits_{S_k \in \Omega_k}\left[\|S_k^\top (Ax^k-b)\|_2^2\right]\\
			&= \|x^k-A^{\dag}b\|_2^2-\frac{\zeta_k(2-\zeta_k)}{\lambda_{\max}^{(k)}} \left\|H_k^{\frac{1}{2}} A(x^k-A^{\dag}b)\right\|_2^2 \\
			&\leq \left(1-\zeta_k(2-\zeta_k)\frac{\sigma_{\min}^{2} (H_k^{\frac{1}{2}}A)}{\lambda_{\max}^{(k)}}\right) \|x^k-A^{\dag}b\|_2^2,
		\end{aligned}
	\end{equation}
	where the first equality follows from the fact that $\mathbb{E}_{S_k \in \mathcal{Q}^c_k}\left[\|S_k^\top (Ax^k-b)\|_2^2\right]=0$ as $S_k^\top (Ax^k-b)=0$ for $S_k \in \mathcal{Q}^c_k$, and the last inequality follows from $H_k=\mathbb{E}_{S \in \Omega_k}[SS^\top]$ is positive definite and $x^k-A^{\dag}b \in \text{Range}(A^\top)$.
	
	Next, we consider  the case where $\Omega_k$ is unbounded. If $S_k \in \mathcal{Q}_k$, we have
	$$L_{\text{adap}}^{(k)}=\frac{\|S_k^\top (Ax^k-b)\|_2^2}{\|A^\top S_k S_k^\top (Ax^k-b)\|_2^2}\geq \frac{1}{\lambda_{\max}\left(\frac{A^\top S_kS_k^\top A}{\|S_k\|_2^2}\right)} \frac{1}{\|S_k\|_2^2}\geq \frac{1}{\lambda_{\max}^{(k)} \|S_k\|_2^2}.$$
	Substituting it into \eqref{Exp} and using the similar arguments in \eqref{proof-xie1}, we can get
	$$
	\begin{aligned}
		\mathop{\mathbb{E}}\limits_{S_k \in \Omega_k}[\|x^{k+1}-A^{\dag}b\|_2^2]
		&\leq \|x^k-A^{\dag}b\|_2^2-\frac{\zeta_k(2-\zeta_k)}{\lambda_{\max}^{(k)}}\mathop{\mathbb{E}}\limits_{S_k \in \mathcal{Q}_k}\left[\frac{\|S_k^\top (Ax^k-b)\|_2^2}{\|S_k\|_2^2}\right] \\
		&= \|x^k-A^{\dag}b\|_2^2-\frac{\zeta_k(2-\zeta_k)}{\lambda_{\max}^{(k)}} \left\|H_k^{\frac{1}{2}} A(x^k-A^{\dag}b)\right\|_2^2 \\
		&\leq \left(1-\zeta_k(2-\zeta_k)\frac{\sigma_{\min}^{2} (H_k^{\frac{1}{2}}A)}{\lambda_{\max}^{(k)}}\right) \|x^k-A^{\dag}b\|_2^2.
	\end{aligned}
	$$
	This completes the proof of this theorem.
\end{proof}

Based on Theorem \ref{CRS-ASGD}, we have the following corollary. 
\begin{corollary}
	\label{full}
	Under the same conditions of Theorem \ref{CRS-ASGD} and set $\rho_k:=\zeta_k(2-\zeta_k)\frac{\sigma_{\min}^{2} (H_k^{\frac{1}{2}}A)}{\lambda_{\max}^{(k)}}$, the iteration sequence $\{x^k\}_{k \geq 0}$ generated by Algorithm \ref{BSM} satisfies
	\begin{equation}
		\label{full_expectation_basic}
		\begin{aligned}
			\mathop{\mathbb{E}} \left[\| x^k-A^\dagger b \|_2^2\right] \leq \prod \limits_{i=0}^{k-1} \left(1-\rho_i \right) \left\| x^0-A^\dagger b \right\|_2^2.
		\end{aligned}
	\end{equation}
\end{corollary}

Next, we compare the upper bound in \eqref{full_expectation_basic} with those of the RK method and the RABK method. 

\begin{remark}
	\label{remark-3.5}
	When $ \Omega_k = \left\{ \frac{e_i}{\Vert A_{i, :} \Vert_2} \right\}_{i=1}^m $ and  $S=\frac{e_i}{\Vert A_{i, :} \Vert_2}$ is selected with probability $\frac{\Vert A_{i, :} \Vert_2^2}{\Vert A \Vert_F^2}$, Algorithm \ref{BSM} with $\zeta_k = 1$ recovers the RK method
	$$
	x^{k+1}=x^k-\frac{A_{i_k, :}x^k-b_{i_k}}{\Vert A_{i_k, :} \Vert_2^2}A_{i_k, :}^{\top}.
	$$
	In this case, we have $H_k=\frac{1}{\Vert A \Vert_F^2}I$ and $\lambda_{\max}^{(k)}=1$. Thus, \eqref{full_expectation_basic} becomes
	$$
	\mathop{\mathbb{E}} \left[\| x^k-A^\dagger b \|_2^2 \right] \leq \left(1-\frac{\sigma_{\min}^{2} (A)}{\Vert A \Vert_F^2 }\right)^k \left\| x^0-A^\dagger b \right\|_2^2,
	$$
	which coincides with the convergence result of RK \cite{Str09}. Furthermore, given that the RK method can actually be viewed as a SGD method with stochastic Polyak step-sizes when applied to linear systems \cite[Section B.3]{loizou2021stochastic}, we have also examined the correlation between the basic method and the SGD method with stochastic Polyak step-sizes. This investigation indicates that, in the context of solving linear systems, the SGD method with stochastic Polyak step-sizes is a special case of the basic method, where the sample space $ \Omega_k $ is finite and remains constant through iterations and the relaxation parameters $ \zeta_k = 1$ for $ k\geq 0 $. In such cases, the result in Corollary \ref{full} reaffirms the convergence results for the SGD method with stochastic Polyak step-sizes in \cite[Corollary 3.2]{loizou2021stochastic}.
\end{remark}

\begin{remark}
	Corollary \ref{full} can derive a slightly tighter convergence bound for the RABK method with adaptive step-sizes compared to that in \cite{Nec19}. Assume that the subset $\mathcal{J}_k \subseteq [m]$ in \eqref{BKM} is characterized by a certain probability space $\{\bar{\Omega},\bar{\mathcal{F}},\bar{P}\}$  and the weights satisfy $0 < \omega_{\min} \leq \omega^{(k)}_i \leq \omega_{\max} < 1$ for all $k $ and $ i$.
	Necoara \cite[Theorem 4.2]{Nec19} showed that for RABK with the step-size $\alpha_{k}=(2-\zeta) L_{\text{adap}}^{(k)}$, it holds that
	$$
	\mathop{\mathbb{E}} \left[\| x^k-A^\dagger b \|_2^2 \right] \leq \left(1-\zeta(2-\zeta)\frac{\omega_{\min} \sigma_{\min} (W)}{\omega_{\max} \lambda_{\max}^{\text{block}}}\right)^k \left\| x^0-A^\dagger b \right\|_2^2,
	$$
	where
	$$
	\lambda_{\max}^{\text{block}}=\mathop{\max}\limits_{\mathcal{J} \in \bar{\Omega}} \lambda_{\max} \left( A_{\mathcal{J}, :}^\top \operatorname{diag}\left(\frac{1}{\Vert A_{i, :} \Vert_2^2}, i \in \mathcal{J} \right) A_{\mathcal{J}, :} \right), \,
	W=A^\top \operatorname{diag}\left(\frac{p_i}{\Vert A_{i, :} \Vert_2^2}, i \in [m]\right) A,
	$$
	and $p_i$ denotes the probability of the index $i$ being selected into $\mathcal{J}$.
	Meanwhile, the parameters in Corollary \ref{full} lead to 
	$$
	H_k=\mathop{\mathbb{E}}\limits_{\mathcal{J} \in \bar{\Omega}} \left[ \sum\limits_{i\in \mathcal{J}} \omega^{(k)}_i \frac{e_i e_i^{\top}}{\Vert A_{i, :} \Vert_2^2} \right] = \operatorname{diag}\left( \omega^{(k)}_i \frac{p_i}{\Vert A_{i, :} \Vert_2^2}, i \in [m]\right) \geq \omega_{\min} \operatorname{diag}\left( \frac{p_i}{\Vert A_{i, :} \Vert_2^2}, i \in [m]\right) ,
	$$
	and
	$$
	\lambda_{\max}^{(k)}=\mathop{\max}\limits_{\mathcal{J} \in \bar{\Omega}} \lambda_{\max} \left( A_{\mathcal{J}, :}^\top \operatorname{diag}\left(\frac{\omega^{(k)}_i}{\Vert A_{i, :} \Vert_2^2}, i \in \mathcal{J} \right) A_{\mathcal{J}, :} \right) \leq \omega_{\max} \lambda_{\max}^{\text{block}}.
	$$
	Letting $\zeta_k = \zeta$, then $$\rho_k=\zeta(2-\zeta)\frac{\sigma_{\min}^{2} (H_k^{\frac{1}{2}}A)}{\lambda_{\max}^{(k)}} \geq \zeta(2-\zeta)\frac{\omega_{\min} \sigma_{\min}(W)}{\omega_{\max} \lambda_{\max}^{\text{block}}},$$ 
	which indicates that
	$$
	\prod \limits_{i=0}^{k-1} \left(1-\rho_i \right) \leq \left(1-\zeta(2-\zeta)\frac{\omega_{\min} \sigma_{\min} (W)}{\omega_{\max} \lambda_{\max}^{\text{block}}}\right)^k,
	$$
	where the equality holds if $\omega^{(k)}_i = \omega$ is a constant.
	Therefore, Corollary \ref{full} leads to a better convergence factor for the RABK method.
\end{remark}

\subsection{Modified basic method}
An obvious weakness of Algorithm \ref{BSM} is that $x^{k+1}=x^k$ when the sampling matrix $S_k \in \mathcal{Q}_k^c$. Hence if we ensure that the sampling matrix $S_k\in \mathcal{Q}_k$, i.e. $S_k^\top(Ax^k-b)\neq 0$, then $\|x^{k+1}-A^{\dagger}b\|_2<\|x^k-A^{\dagger}b\|_2$ (see \eqref{descent-xie}), which could reduce the number of iterations of the basic method. For this reason, we construct the modified basic method described in Algorithm \ref{ibm-xie}.

\begin{algorithm}[htpb]
	\caption{Modified basic method}\label{ibm-xie}
	\begin{algorithmic}
		\STATE{\textbf{Input:} $A \in \mathbb{R}^{m \times n}$, $b \in \mathbb{R}^m$, probability spaces $\{(\Omega_k, \mathcal{F}_k, P_k)\}_{k\geq 0}$, $\{\zeta_k\}_{k \geq 0}$ with $\zeta_k \in (0, 2)$, $k=0$, and the initial point $x^0 \in \text{Range}(A^\top)$.}
		\begin{enumerate}
			\item[1:] Randomly select a sampling matrix $S_k\in \Omega_k$ until $S_k^\top (Ax^k-b)\neq 0$.
			\item[2:] Compute the parameter $L_{\text{adap}}^{(k)}$ in \eqref{Lk} and set $\alpha_k=(2-\zeta_k)L_{\text{adap}}^{(k)}$.
			\item[3:] Update $x^{k+1}=x^k-\alpha_k A^\top S_k S_k^\top (Ax^k-b)$.
			\item[4:] If the stopping rule is satisfied, stop and go to output. Otherwise, set $k=k+1$ and return to Step $1$.
		\end{enumerate}
		\STATE{\textbf{Output:} The approximate solution $x^k$.}
	\end{algorithmic}
\end{algorithm}

We note that Algorithm \ref{ibm-xie} is well-defined. This is because if the set $\mathcal{Q}_k$ is empty, then  Lemma \ref{xie-empty} implies that $x^k$ is already a solution of the linear system \eqref{LS}. In addition, Theorem \ref{CRS-ASGD} is also applicable to Algorithm \ref{ibm-xie}.

\section{Acceleration by adaptive heavy ball momentum}

In this section, we aim to enhance the basic  method by incorporating adaptive heavy ball momentum. Suppose that $x^0\in \text{Range}(A^\top)$ is an arbitrary initial point and let $x^1$ be generated by Algorithm \ref{BSM} with $\zeta_0=1$. At the $k$-th iteration $(k\geq1)$, the stochastic heavy ball momentum (SHBM) method updates with the following iterative strategy
$$
x^{k+1}=x^{k}-\alpha_{k} \nabla f_{S_k}(x^{k})+\beta_k(x^{k}-x^{k-1}),
$$
where $S_k$ is randomly chosen from $\Omega_k$, $\alpha_k$ is the step-size, and $\beta_k$ is the momentum parameter. We pay particular attention to the selection of the parameters $\alpha_k$ and $\beta_k$. Specifically, we expect to choose  $\alpha_k$ and $\beta_k$ that the error $\|x^{k+1}-A^\dagger b\|_2$ is minimized, i.e. the minimizers of  the following constrained optimization problem
\begin{equation}\label{opt-prob}
	\begin{aligned}
		\min\limits_{ \alpha,\beta\in\mathbb{R}}& \ \ \|x-A^\dagger b\|_2^2\\
		\text{subject to}& \ \ x=x^{k}-\alpha \nabla f_{S_k}(x^{k})+\beta(x^{k}-x^{k-1}).
	\end{aligned}
\end{equation}
If $\| \nabla f_{S_k}(x^k) \|_2^2 \| x^k-x^{k-1} \|_2^2 - \langle \nabla f_{S_k}(x^k), x^k-x^{k-1} \rangle^2\neq 0$, then the minimizers of \eqref{opt-prob}
are
\begin{equation}\label{ab-xie-0502}
	\left\{\begin{array}{ll}
		\alpha_k
		=
		\frac{\| x^k-x^{k-1} \|_2^2 \langle\nabla f_{S_k}(x^{k}),x^k-A^\dagger b\rangle - \langle \nabla f_{S_k}(x^k),x^k-x^{k-1} \rangle \langle x^k-x^{k-1}, x^k-A^\dagger b \rangle}{\| \nabla f_{S_k}(x^k) \|_2^2 \| x^k-x^{k-1} \|_2^2 - \langle \nabla f_{S_k}(x^k), x^k-x^{k-1} \rangle^2},
		\\
		\\
		\beta_k
		=
		\frac{\langle \nabla f_{S_k}(x^k),x^k-x^{k-1} \rangle \langle\nabla f_{S_k}(x^{k}),x^k-A^\dagger b\rangle - \| \nabla f_{S_k}(x^k) \|_2^2 \langle x^k-x^{k-1}, x^k-A^\dagger b \rangle}{\| \nabla f_{S_k}(x^k) \|_2^2 \| x^k-x^{k-1} \|_2^2 - \langle \nabla f_{S_k}(x^k), x^k-x^{k-1} \rangle^2},
	\end{array}
	\right.
\end{equation}
which seem to be intractable since $\langle\nabla f_{S_k}(x^{k}),x^k-A^\dagger b\rangle$ and $\langle x^{k}-x^{k-1}, x^k-A^\dagger b \rangle$ include an unknown vector $A^\dagger b$.  In fact, we can calculate the two terms directly to get rid of $A^\dagger b$.

On the one hand, noting that $ \nabla f_{S_k}(x^{k})=A^\top S_k S_k^\top (Ax^k-b)$ and $AA^\dagger b=b$, then we know that $$\langle\nabla f_{S_k}(x^{k}),x^k-A^\dagger b\rangle=\langle S_k^\top (Ax^k-b),S_k^\top A(x^k-A^\dagger b)\rangle=\|S_k^\top (Ax^k-b)\|_2^2.$$
On the other hand, we note that for $ k\geq2 $, $x^k$ is the orthogonal projection of $A^{\dagger} b$ onto the affine set
$$\Pi_{k-1}:=x^{k-1}+\mbox{Span}\{\nabla f_{S_{k-1}}(x^{k-1}), x^{k-1}-x^{k-2}\}$$
and hence
$$\langle x^{k}-x^{k-1},x^k-A^\dagger b  \rangle=0.$$
A geometric interpretation is presented in Figure \ref{GI1}. For $k=1$, by the definition of $x^1$, we also have that $\langle x^{1}-x^{0}, x^{1}-A^\dag b \rangle=0$. Therefore,   \eqref{ab-xie-0502} can be simplified to
\begin{equation}
	\label{Parameters}
	\left\{\begin{array}{ll}
		\alpha_k
		=
		\frac{\| x^k-x^{k-1} \|_2^2 \|S_k^\top (Ax^k-b)\|_2^2}{\| \nabla f_{S_k}(x^k) \|_2^2 \| x^k-x^{k-1} \|_2^2 - \langle \nabla f_{S_k}(x^k), x^k-x^{k-1} \rangle^2},
		\\
		\\
		\beta_k
		=
		\frac{\langle \nabla f_{S_k}(x^k),x^k-x^{k-1} \rangle \|S_k^\top (Ax^k-b)\|_2^2}{\| \nabla f_{S_k}(x^k) \|_2^2 \| x^k-x^{k-1} \|_2^2 - \langle \nabla f_{S_k}(x^k), x^k-x^{k-1} \rangle^2}.
	\end{array}
	\right.
\end{equation}
To sum up, for any $k\geq1$, if $\| \nabla f_{S_k}(x^k) \|_2^2 \| x^k-x^{k-1} \|_2^2 - \langle \nabla f_{S_k}(x^k), x^k-x^{k-1} \rangle^2\neq 0$, i.e. $\text{dim}(\Pi_k)=2$, can be guaranteed, then the minimizers of \eqref{opt-prob} can be computed by \eqref{Parameters}.

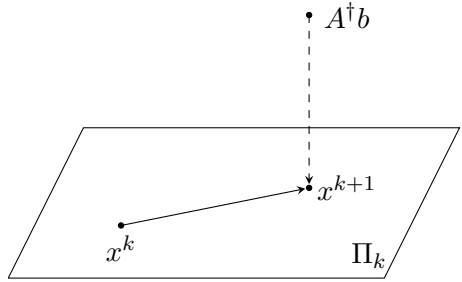
\begin{figure}[hptb]
	\centering
	\begin{tikzpicture}
		\draw (0,0)--(5,0)--(6,2)--(1,2)--(0,0);
		
		\filldraw (1.5,0.7) circle [radius=1pt]
		(4,3.5) circle [radius=1pt]
		(4,1.2) circle [radius=1pt];
		\draw (1.5,0.4) node {$x^k$};
		\draw (4.5,3.5) node {$A^\dagger b$};
		\draw (4.5,1.2) node {$x^{k+1}$};
		\draw (4.8,0.3) node {$\Pi_k$};
		\draw [dashed,-stealth] (4,3.5) -- (4,1.25);
		\draw [-stealth] (1.5,0.7) -- (3.95,1.19);
	\end{tikzpicture}
	\caption{A geometric interpretation of our design. The next iterate $x^{k+1}$  arises such that $x^{k+1}$ is the orthogonal projection of $A^{\dagger} b$ onto the affine set $\Pi_k=x^k+\mbox{Span}\{\nabla f_{S_k}(x^k), x^k-x^{k-1}\}$.}	
	\label{GI1}
\end{figure}

The following proposition indicates that if the sampling matrices $S_k$ are chosen such that $S_k^\top (Ax^k-b)\neq 0$ for $k\geq 0$, then $\text{dim}(\Pi_\ell)=2(\ell\geq1)$.
\begin{prop}
	\label{prop}
	Let $x^0 \in \text{Range}(A^\top)$ be an initial point and suppose that $x^1$ is generated by Algorithm \ref{BSM} with $\zeta_0=1$. Let  $\{x^k\}_{k\geq2}$ be the sequence obtained by solving the optimization problem \eqref{opt-prob}. We have the following results:
	\begin{itemize}
		\item[$(\romannumeral1)$]  For any fixed $\ell$, if  $S_\ell^\top (Ax^\ell-b)\neq 0$, then  $x^{\ell+1}\neq x^\ell$;
		\item[(\romannumeral2)] If $S_\ell^\top (Ax^\ell-b)\neq 0$ and $S_{\ell-1}^\top (Ax^{\ell-1}-b)\neq 0 (\ell\geq1)$, then  $\text{dim}(\Pi_\ell)=2$.
	\end{itemize}
\end{prop}
\begin{proof}
	$(\romannumeral1)$ Consider the case where $\ell=0$. If $S_0^\top (Ax^0-b)\neq 0$, it follows from Lemma \ref{lemma-non} that $\nabla f_{S_0}(x^0)=A^\top S_0S_0^\top (Ax^0-b)\neq 0$ . Thus, we can directly get $x^1 \neq x^0$. For the case where $\ell\geq1$, by the assumption we know that $\nabla f_{S_\ell}(x^\ell)\neq 0$ which implies $\text{dim}(\Pi_\ell)\geq1$.
	
	If $\text{dim}(\Pi_\ell)=1$, i.e. $\nabla f_{S_\ell}(x^\ell)$ is parallel to $x^\ell-x^{\ell-1}$, we have $\Pi_\ell=x^\ell+\mbox{Span}\{\nabla f_{S_\ell}(x^\ell)\}$. Thus, $x^{\ell+1}=x^\ell-L_{\text{adap}}^\ell \nabla f_{S_\ell}(x^\ell)$ is not equal to $x^\ell$.
	
	For the case where $\text{dim}(\Pi_\ell)=2$, we can also derive $x^{\ell+1}\neq x^\ell$. Otherwise, we will have $\alpha_{\ell}=0$ and $\beta_{\ell}=0$. It follows from $\text{dim}(\Pi_\ell)=2$, we have $\|x^\ell-x^{\ell-1}\|_2\neq0$, $\| S^\top(Ax^\ell-b)\|_2\neq 0$, and $\| \nabla f_{S_\ell}(x^\ell) \|_2^2 \| x^\ell-x^{\ell-1} \|_2^2 - \langle \nabla f_{S_\ell}(x^\ell), x^\ell-x^{\ell-1} \rangle^2\neq 0$. Then from \eqref{Parameters}, we know that $\alpha_\ell\neq0$, which contradicts to the assumption that $\alpha_\ell=0$.
	
	$(\romannumeral2)$  According to $(\romannumeral1)$,  we can get
	\begin{equation}
		\label{xie-0423-2}
		x^{\ell}-x^{\ell-1}\neq 0\ \text{and} \ x^{\ell+1}-x^\ell\neq 0.
	\end{equation}
	Moreover, for $\ell\geq2$, we have
	\begin{equation}\label{xie-0423-1}
		\begin{aligned}
			\langle x^{\ell+1}-x^\ell, x^\ell-x^{\ell-1} \rangle
			&=\langle (x^{\ell+1}-A^\dagger b)-(x^\ell-A^\dagger b), x^\ell-x^{\ell-1} \rangle \\
			&=\langle x^{\ell+1}-A^\dagger b, x^\ell-x^{\ell-1} \rangle-\langle x^\ell-A^\dagger b, x^\ell-x^{\ell-1} \rangle \\
			&=0.
		\end{aligned}
	\end{equation}
	The last equality follows from the fact that the next iterate $x^{i+1}$  arises as the orthogonal projection of $A^{\dagger} b$ onto the affine set $\Pi_i=x^i+\mbox{Span}\{\nabla f_{S_i}(x^i), x^i-x^{i-1}\}$, where $i\in\{\ell-1,\ell\}$ (see Figure \ref{GI1}). Besides, one can verify that \eqref{xie-0423-1} also holds for $\ell=1$.
	
	Since $\nabla f_{S_{\ell}}(x^\ell)\neq0$, we know that $\text{dim}(\Pi_\ell)\geq1$. If $\text{dim}(\Pi_\ell)=1$, i.e. $\nabla f_{S_\ell}(x^\ell)$ is parallel to $x^\ell-x^{\ell-1}$. Then we have $x^{\ell+1}-x^\ell=-\alpha_{\ell}\nabla f_{S_\ell}(x^\ell)+\beta_\ell(x^\ell-x^{\ell-1})$ which
	is also parallel to $x^\ell-x^{\ell-1}$. Thus, by \eqref{xie-0423-1}, we know that either $x^{\ell+1}-x^\ell=0$ or $x^\ell-x^{\ell-1}=0$, which contradicts \eqref{xie-0423-2}. Hence, we have $\text{dim}(\Pi_\ell)=2$. This completes the proof of this proposition.
\end{proof}
Now, we have already constructed the ASHBM method  described in Algorithm \ref{ASHBM}.

\begin{algorithm}[htpb]
	\caption{ Adaptive stochastic heavy ball momentum (ASHBM)}
	\label{ASHBM}
	\begin{algorithmic}
		\STATE{\textbf{Input:} $A \in \mathbb{R}^{m \times n}$, $b \in \mathbb{R}^m$, probability spaces $\{(\Omega_k, \mathcal{F}_k, P_k)\}_{k\geq 0}$, $k=1$, and the initial point $x^0 \in \text{Range}(A^\top)$.}
		\begin{enumerate}
			\item[1:] Update $x^1$ by the modified basic method (Algorithm \ref{ibm-xie}) with $\zeta_0=1$.
			\item[2:] Randomly select a sampling matrix $S_k \in \Omega_k$ until $S_k^\top (Ax^k-b)\neq 0$.
			\item[3:] Compute the parameters $\alpha_k$ and $\beta_k$ in \eqref{Parameters}.
			\item[4:]
			Update $x^{k+1}=x^k-\alpha_k A^\top S_k S_k^\top (Ax^k-b)+\beta_k (x^k-x^{k-1})$.
			
			\item[5:] If the stopping rule is satisfied, stop and go to output. Otherwise, set $k=k+1$ and return to Step $2$.
		\end{enumerate}
		\STATE{\textbf{Output:} The approximate solution $x^k$.}
	\end{algorithmic}
\end{algorithm}

\subsection{Convergence analysis}

Let us first introduce some auxiliary variables. We denote
$$
\tilde{x}^{k+1}:=x^k-L_{\text{adap}}^{(k)} A^\top S_k S_k^\top(Ax^k-b),
$$
where $L_{\text{adap}}^{(k)}$ is given by \eqref{Lk}. Let $\mathcal{Q}_k$ be defined as \eqref{xie-Qk}. Starting from $x^k$ and choosing $S_k\in \mathcal{Q}_k$, we know that $\tilde{x}^{k+1}$ can be obtained by taking one step of Algorithm \ref{ibm-xie} with $\zeta_k=1$.
Set $$u_k:=\langle \nabla f_{S_k}(x^k),x^k-x^{k-1} \rangle \nabla f_{S_k}(x^k)-\|\nabla f_{S_k}(x^k)\|^2_2(x^k-x^{k-1}).$$
We use $\theta_k$ to denote the angle between $\tilde{x}^{k+1}-A^\dagger b$ and $u_k$. Let
\begin{equation}
	\label{def-gamma}
	\gamma_k:=\inf\limits_{S_k \in \mathcal{Q}_k} \left\{ \cos^2 \theta_k \right\}.
\end{equation}

We have the following convergence result for Algorithm \ref{ASHBM}.
\begin{theorem}
	\label{CRS-AmSGD}
	Suppose that the  linear system \eqref{LS} is consistent and the probability spaces $\{(\Omega_k, \mathcal{F}_k, P_k)\}_{k\geq 0}$ satisfy Assumption \ref{Ass}. Let $\{x^k\}_{k \geq0}$ be the iteration sequence generated by Algorithm \ref{ASHBM}.
	Then
	$$
	\mathbb{E}\left[\|x^{k+1}-A^\dagger b\|_2^2 \ \big| \  x^k\right] \leq ( 1-\gamma_k) \left( 1-\frac{\sigma_{\min}^{2} (H_k^{\frac{1}{2}}A)}{\lambda_{\max}^{(k)}} \right) \|x^k-A^\dagger b\|_2^2,
	$$
	where $\gamma_k$, $H_k$, and $\lambda_{\max}^{(k)}$ are given by \eqref{def-gamma}, \eqref{matrix-H}, and \eqref{lambda-max}, respectively.
\end{theorem}
\begin{proof}
	Since $x^{k+1}-A^\dagger b$ is orthogonal to $\Pi_k$ and $x^{k+1},\tilde{x}^{k+1}\in \Pi_k$, the Pythagorean Theorem implies that
	$$
	\|x^{k+1}-A^\dagger b\|_2^2=\|\tilde{x}^{k+1}-A^\dagger b\|_2^2-\|x^{k+1}-\tilde{x}^{k+1}\|_2^2.
	$$
	Since $\|u_k\|^2_2=\| \nabla f_{S_k}(x^k)\|^2_2\left(\| \nabla f_{S_k}(x^k) \|_2^2 \| x^k-x^{k-1} \|_2^2 - \langle \nabla f_{S_k}(x^k), x^k-x^{k-1} \rangle^2\right)$ and $\text{dim}(\Pi_k)=2$, we know that $u_k\neq0$. Let
	$l_k:=\frac{\langle \tilde{x}^{k+1}-A^\dag b,u_k \rangle}{\| u_k \|_2^2}$, then it holds that
	$$
	x^{k+1}=\tilde{x}^{k+1}-l_k u_k.
	$$
	Hence
	$$
	\|x^{k+1}-\tilde{x}^{k+1}\|_2^2
	=
	l_k^2 \|u_k\|_2^2
	=
	\frac{\langle \tilde{x}^{k+1}-A^\dagger b, u_k \rangle^2}{\|u_k\|_2^2}
	=\cos^2\theta_k\| \tilde{x}^{k+1}-A^\dagger b\|^2_2,
	$$
	where $\theta_k$ denotes the angle between $\tilde{x}^{k+1}-A^\dagger b$ and $u_k$.
	Thus
	$$
	\begin{aligned}
		\mathbb{E}\left[\|x^{k+1}-A^\dagger b\|_2^2 \big|  x^k\right]&=\mathbb{E}\left[\|\tilde{x}^{k+1}-A^\dagger b\|_2^2 \big|  x^k\right]-\mathbb{E}\left[\|x^{k+1}-\tilde{x}^{k+1}\|_2^2\big|  x^k\right]
		\\
		&=\mathbb{E}\left[(1-\cos^2\theta_k)\|\tilde{x}^{k+1}-A^\dagger b\|_2^2 \big|  x^k\right]
		\\
		&\leq
		\left( 1-\gamma_k \right) \mathbb{E}\left[\|\tilde{x}^{k+1}-A^\dagger b\|_2^2 \big|  x^k\right]  \\
		&\leq
		\left( 1-\gamma_k \right) \left( 1-\frac{\sigma_{\min}^{2} (H_k^{\frac{1}{2}}A)}{\lambda_{\max}^{(k)}} \right) \|x^k-A^\dag b\|_2^2,
	\end{aligned}
	$$
	where the last inequality follows from the definition of  $\tilde{x}^{k+1}$ which can be obtained by one step of  Algorithm \ref{ibm-xie} with $\zeta_k=1$ and Theorem \ref{CRS-ASGD} is  applicable to Algorithm \ref{ibm-xie}.
\end{proof}

\begin{remark}
	Upon comparison of Theorem \ref{CRS-ASGD} and Theorem \ref{CRS-AmSGD}, it can be observed that the ASHBM method exhibits convergence bound that is at least as that of the basic method.
	Particularly, for certain probability spaces $\{(\Omega_k, \mathcal{F}_k, P_k)\}_{k\geq 0}$, the parameter $\gamma_k$ in Theorem \ref{CRS-AmSGD} can  be strictly greater than $0$ (i.e. $\gamma_k>0$) as long as $Ax^k\neq  b$ (see Remark \ref{remak-xie-0427}).  Hence, our result positively addresses the open problem proposed by Loizou and Richt{\'a}rik for quadratic objectives \cite[Section 4.1]{loizou2020momentum}.
\end{remark}

\section{Connection to conjugate gradient-type methods}

In this section, we will introduce an equivalent form of the ASHBM method. Using this form, we show that the \emph{conjugate gradient normal equation error} (CGNE) method \cite[Section 11.3.9]{golub2013matrix}, a variant of the conjugate gradient method, is actually a  special case of ASHBM. Inspired by this observation, we obtain a novel framework of the stochastic conjugate gradient (SCG) methods for solving linear systems, from which a range of SCG methods can be derived via setting different probability spaces $ \{ \Omega_k, \mathcal{F}_k, P_k \}_{k\geq 0} $.

\subsection{An equivalent form of ASHBM}

In this subsection, we will derive an equivalent expression for ASHBM and then further study some properties of ASHBM. We have the following result.

\begin{theorem}\label{SCGNE}
	Suppose that $x^0\in \text{Range}(A^\top)$ is the  initial point in Algorithm \ref{ASHBM}  and set $p_0=-\nabla f_{S_{0}}(x^{0})=-A^\top S_0S_0^\top (Ax^0-b)$. Then for any $k\geq0$,  Algorithm \ref{ASHBM} can be  equivalently rewritten as
	\begin{equation}
		\label{EF}
		\left\{\begin{array}{ll}
			\delta_k =\|S_k^\top (Ax^k-b)\|_2^2 / \|p_k\|_2^2,
			\\[1.7mm]
			x^{k+1}=x^k+\delta_k p_k,
			\\[1.7mm]
			\nabla f_{S_{k+1}}(x^{k+1})= A^\top S_{k+1}S_{k+1}^\top (Ax^{k+1}-b),
			\\[1.7mm]
			\eta_k=\langle \nabla f_{S_{k+1}}(x^{k+1}), p_k \rangle/ \|p_k\|_2^2,
			\\[1.7mm]
			p_{k+1}= - \nabla f_{S_{k+1}}(x^{k+1})+\eta_k p_k.
		\end{array}
		\right.
	\end{equation}
\end{theorem}

\begin{proof}
	Set $\alpha_0:=\|S_0^\top (Ax^0-b)\|_2^2 / \|\nabla f_{S_{0}}(x^{0})\|_2^2$, $\beta_0:=0$, and $x^{-1}=x^0$.
	Then according to the expression of $x^{k+1}$ in Algorithm \ref{ASHBM}, we have
	$$
	x^{k+1}-x^k=-\alpha_k \nabla f_{S_k}(x^k)+\beta_k(x^k-x^{k-1}).
	$$
	For $k\geq0$, we set $\tilde{p}_k:=x^{k+1}-x^k$. Then we have
	$$\tilde{p}_{k+1}=-\alpha_{k+1} \nabla f_{S_{k+1}}(x^{k+1})+\beta_{k+1}  \tilde{p}_{k}.$$
	From \eqref{Parameters} and Proposition \ref{prop}, we know that $\alpha_{k+1}\neq 0$ and
	$$
	\frac{\beta_{k+1}}{\alpha_{k+1}}=\frac{\langle \nabla f_{S_{k+1}}(x^{k+1}), x^{k+1}-x^{k} \rangle}{\|x^{k+1}-x^{k}\|_2^2}= \frac{\langle \nabla f_{S_{k+1}}(x^{k+1}), \tilde{p}_{k} \rangle}{\|\tilde{p}_{k}\|_2^2}.
	$$
	Hence we have
	$$
	\begin{aligned}
		\tilde{p}_{k+1}&=\alpha_{k+1}\left( -\nabla f_{S_{k+1}}(x^{k+1})+\frac{\beta_{k+1}}{\alpha_{k+1}}  \tilde{p}_{k}\right)\\
		&=\alpha_{k+1}\left( -\nabla f_{S_{k+1}}(x^{k+1})+ \frac{\langle \nabla f_{S_{k+1}}(x^{k+1}), \tilde{p}_{k} \rangle}{\|\tilde{p}_{k}\|_2^2} \tilde{p}_{k} \right).
	\end{aligned}$$
	We state that $p_k$ in \eqref{EF} satisfies $p_k=\frac{\tilde{p}_k}{\alpha_k}$. Indeed, let $\bar{p}_k=\frac{\tilde{p}_k}{\alpha_k}$, then  $\{\bar{p}_k\}_{k\geq0}$ has the following recursive relationship
	$$
	\bar{p}_{k+1}=-\nabla f_{S_{k+1}}(x^{k+1})+\frac{\langle \nabla f_{S_{k+1}}(x^{k+1}), \bar{p}_{k} \rangle}{\|\bar{p}_{k}\|_2^2} \bar{p}_{k}.
	$$
	Note that $\bar{p}_0=\frac{\tilde{p}_0}{\alpha_0}=\frac{x^1-x^0}{\alpha_0}=-\nabla f_{S_{0}}(x^{0})=p_0$. Hence for $i\geq1$, we also have $\bar{p}_i=p_i$. So we can obtain that $p_k=\frac{\tilde{p}_k}{\alpha_k}$, i.e. $$x^{k+1}=x^k+\alpha_k p_k.$$
	Next, we show that $\alpha_k =\delta_k$. If $k=0$, it is obvious that $\alpha_0=\delta_0$. We consider the case where $k\geq1$.
	Since $\alpha_kp_k=x^{k+1}-x^k\in \Pi_k$ and $x^{k+1}-A^\dagger b$ is orthogonal to $\Pi_k$, we know that $\langle x^{k+1}-A^{\dagger}b, \alpha_k  p_k\rangle=0$, which implies
	$$
	0=\langle x^{k+1}-x^k+x^k-A^{\dagger}b, \alpha_k  p_k\rangle=\alpha^2_k\|p_k\|^2_2+\alpha_k \langle x^k-A^{\dag}b, p_k\rangle.
	$$
	Hence $\alpha_k=\frac{\langle A^{\dagger}b-x^k, p_k \rangle}{\|p_k\|_2^2}$. Since $p_{k-1}$ is parallel to $x^{k}-x^{k-1}$, we have $\langle A^{\dagger}b-x^k, p_{k-1} \rangle=0$, which leads to
	$$
	\langle A^{\dag}b-x^k, p_k \rangle=\langle A^{\dagger}b-x^k, -\nabla f_{S_k}(x^k) \rangle=\|S_k^\top (Ax^k-b)\|_2^2.
	$$
	Therefore, we have
	$$
	\alpha_k=\frac{\langle A^{\dagger}b-x^k, p_k \rangle}{\|p_k\|_2^2}=\frac{\|S_k^\top (Ax^k-b)\|_2^2}{\|p_k\|_2^2}=\delta_k
	$$
	as desired. This completes the proof of this theorem.
\end{proof}

Based on the proof above, we know that for any $k\geq 0$, $\langle x^{k+1}-A^{\dag}b, p_k\rangle=0$.
In addition, we have the following result for iteration \eqref{EF}.

\begin{prop}
	\label{prop-0426}
	Suppose that $\{x^k\}_{k\geq0}$ and $\{p_k\}_{k \geq 0}$ are the sequences generated by \eqref{EF}. Let $r^k=Ax^k-b$. Then we have
	\begin{enumerate}
		\item[(\romannumeral1)] $\langle p_k ,p_{k+1}\rangle=0$;
		\item[(\romannumeral2)] $(r^{k+1})^\top S_k S_k^\top r^k=0$.
	\end{enumerate}
\end{prop}
\begin{proof}
	$(\romannumeral1)$ By the  format of $p_{k+1}$ in \eqref{EF}, we can get
	$$
	\langle p_k, p_{k+1} \rangle= \langle p_k, - \nabla f_{S_{k+1}}(x^{k+1})+\eta_k p_k \rangle= -\langle p_k, \nabla f_{S_{k+1}}(x^{k+1})\rangle+\eta_k \|p_k\|_2^2=0.
	$$
	
	$(\romannumeral2)$ Since $r^{k+1}=Ax^{k+1}-b=Ax^k+\delta_k A p_k-b=r^k+\delta_k A p_k$, we have
	$$\begin{aligned}
		(r^{k+1})^\top S_k S_k^\top r^k
		&=\langle r^k+\delta_k A p_k, S_k S_k^\top r^k \rangle =\|S_k^\top r^k\|_2^2-\delta_k \langle p_k, -\nabla f_{S_k}(x^k) \rangle  \\
		&=\|S_k^\top r^k\|_2^2-\delta_k \langle p_k, p_k-\eta_{k-1} p_{k-1} \rangle  =\|S_k^\top r^k\|_2^2-\delta_k \|p_k\|_2^2
		\\
		&=0
	\end{aligned}
	$$
	as desired.
\end{proof}

\begin{remark}\label{remak-xie-0427}
	Based on Proposition \ref{prop-0426}, we can show that the parameter $\gamma_k$ in Theorem \ref{CRS-AmSGD} can be strictly larger than $0$ in some cases. For example, if the sample space $\Omega_k=\{I\}$ for all $k$, then $L_{\text{adap}}^{(k)}=\frac{\|r^k\|^2_2}{\|A^\top r^k\|^2_2}$ and
	$$
	\begin{aligned}
		u_k&=\langle A^\top r^k,x^k-x^{k-1}\rangle A^\top r^k-\|A^\top r^k\|^2_2(x^k-x^{k-1})\\
		&=\langle  r^k,r^k-r^{k-1}\rangle A^\top r^k-\|A^\top r^k\|^2_2(x^k-x^{k-1})\\
		&=\|r^k\|^2_2A^\top r^k-\|A^\top r^k\|^2_2(x^k-x^{k-1}),
	\end{aligned}
	$$
	where the last equality follows from Proposition \ref{prop-0426}.
	We have
	$$
	\begin{aligned}
		\langle \tilde{x}^{k+1}-A^\dagger b, u^k \rangle
		=&\left\langle x^k-A^\dagger b-\frac{\|r^k\|^2_2}{\|A^\top r^k\|^2_2} A^\top r^k,\|r^k\|^2_2A^\top r^k-\|A^\top r^k\|^2_2(x^k-x^{k-1})\right\rangle  \\
		=
		&- \|A^\top r^k\|^2_2\langle x^k-A^\dagger b,x^k-x^{k-1}\rangle
		+\|r^k\|^2_2\langle A^\top r^k,x^k-x^{k-1}\rangle\\
		=& \|r^k\|_2^4,
	\end{aligned}
	$$
	where the last equality follows from $\langle x^k-A^\dagger b,x^k-x^{k-1}\rangle=0$. Hence, we now have $\gamma_k=\cos^2 \theta_k>0$ as long as $Ax^k\neq  b$.
	We note  that a similar result can be obtained if we assume that the sample space $\Omega_k=\{S\}$ for all $k$, where $S$ is a fixed matrix such that $S^\top S$ is positive definite.
\end{remark}

\subsection{Connection to  the CG method}

Conjugate gradient (CG) method is one of the most well-known adaptive algorithms.
Originally introduced by Hestenes and Stiefel in 1952 \cite{hestenes1952methods} for solving linear systems, this method has gained popularity due to its strong theoretical guarantees, such as finite-time convergence. Ever since, numerous variants of CG method have been developed for a range of nonlinear problems, and we refer to \cite{hager2006survey}  for a nice survey  of them.
Consider the following equivalent problem
\begin{equation}
	\label{cg-xie-0427}
	AA^\top y= b, x=A^\top y
\end{equation}
of $Ax=b$. Starting with an arbitrary point $x^0\in\mathbb{R}^n$, $r^0=Ax^0-b$, and $p_0=-A^\top r^0$, the CG method for solving \eqref{cg-xie-0427} results in the following  \emph{conjugate gradient normal equation error} (CGNE) method \cite[Section 11.3.9]{golub2013matrix}
\begin{equation}
	\label{CG-method}
	\left\{
	\begin{array}{ll}
		\mu_k=\frac{\|r^k\|^2_2}{\|p_k\|^2_2},\\
		x^{k+1}=x^k+\mu_k p_k,  \\
		r^{k+1}=r^k+\mu_k Ap_k,  \\
		\tau_k=\frac{\|r^{k+1}\|^2_2}{\|r^k\|^2_2},
		\\
		p_{k+1}=-A^\top r^{k+1}+\tau_k p_k.
	\end{array}
	\right.
\end{equation}
To investigate the connection between CGNE and ASHBM, we present  another expression for the parameter $\eta_k$ in \eqref{EF}. Specifically, we have
$$
\begin{aligned}
	\eta_k&=\frac{\langle \nabla f_{S_{k+1}}(x^{k+1}), p_k \rangle}{\|p_k\|_2^2}=\frac{\langle A^\top S_{k+1}S^\top_{k+1}r^{k+1}, p_k \rangle}{\|p_k\|_2^2}=\frac{\langle S^\top_{k+1}r^{k+1}, S_{k+1}^\top A p_k \rangle}{\|p_k\|_2^2}\\
	&=\frac{\langle S^\top_{k+1}r^{k+1}, S_{k+1}^\top(r^{k+1}-r^k) \rangle}{\delta_k\|p_k\|_2^2}=\frac{\|S_{k+1}^\top r^{k+1}\|^2_2-\langle S^\top_{k+1}r^{k+1}, S_{k+1}^\top r^k \rangle}{\|S^\top_kr^k\|_2^2}.
\end{aligned}
$$
Therefore, if we assume that the sample spaces $\Omega_k=\{I\}$ for all $k$,
Proposition \ref{prop-0426} implies that now the parameter  $\eta_k = \frac{\|r^{k+1}\|^2_2}{\|r^k\|^2_2}$. It is evident that \eqref{EF} and \eqref{CG-method} are now equivalent, indicating that CGNE is a special case of ASHBM.
Moreover, if we assume that the sample spaces $\Omega_k$ are a singleton set $\{S\}$ for all $k$, where $S$ is a fixed matrix satisfying $S^\top S$ is positive definite, then ASHBM can be interpreted as a preconditioned CGNE method, where the preconditioner is the matrix $S$.

The above discussion indicates that the ASHBM approach can be leveraged to establish a novel stochastic conjugate gradient (SCG) method.
The derived SCG method is presented in Algorithm \ref{SCG-xie}.
To the best of our knowledge, although there have already been work  on the SCG methods in \cite{Jin2019-dr,Schraudolph2004-zh,Yang2022-ie}, such a SCG procedure has not been investigated.

\begin{algorithm}[htpb]
	\caption{Stochastic conjugate gradient(SCG)}
	\label{SCG-xie}
	\begin{algorithmic}
		\STATE{\textbf{Input:} $A \in \mathbb{R}^{m \times n}$, $b \in \mathbb{R}^m$, probability spaces $\{(\Omega_k, \mathcal{F}_k, P_k)\}_{k\geq 0}$, $k=0$ and the initial point $x^0 \in \text{Range}(A^\top)$.}
		\begin{enumerate}
			\item[1:] Randomly select a sampling matrix $S_0 \in \Omega_0$ until $S_0^\top (Ax^0-b)\neq 0$.
			\item[2:] Set $p_0=-A^\top S_0S_0^\top (Ax^0-b)$.
			\item[3:] Set $\delta_k =\|S_k^\top (Ax^k-b)\|_2^2 / \|p_k\|_2^2$.
			\item[4:] Update $x^{k+1}=x^k+\delta_k p_k$.
			\item[5:] Randomly select a sampling matrix $S_{k+1} \in \Omega_{k+1}$ until $S_{k+1}^\top (Ax^{k+1}-b)\neq 0$.
			\item[6:] Compute
			$$
			\begin{aligned}
				\eta_k&=\frac{\|S_{k+1}^\top (Ax^{k+1}-b)\|^2_2-\langle S^\top_{k+1}(Ax^{k+1}-b), S_{k+1}^\top (Ax^{k}-b) \rangle}{\|S^\top_k(Ax^{k}-b)\|_2^2},\\
				p_{k+1}&= - A^\top S_{k+1}S_{k+1}^\top (Ax^{k+1}-b)+\eta_k p_k.
			\end{aligned}
			$$
			\item[7:] If the stopping rule is satisfied, stop and go to output. Otherwise, set $k=k+1$ and return to Step $3$.
		\end{enumerate}
		\STATE{\textbf{Output:} The approximate solution $x^k$.}
	\end{algorithmic}
\end{algorithm}

\section{Numerical experiments}

In this section, we implement the ASHBM method (Algorithm \ref{ASHBM}) and the modified basic method (Algorithm \ref{ibm-xie}).
We also compare our algorithms with the momentum variant of the basic method with time-invariant parameters proposed in \cite{loizou2020momentum,han2022pseudoinverse}
and the   {\sc Matlab} functions \texttt{pinv} and \texttt{lsqminnorm}.

All the methods are implemented in  {\sc Matlab} R2022a for Windows $11$ on a desktop PC with Intel(R) Core(TM) i7-1360P CPU @ 2.20GHz  and 32 GB memory. The code to reproduce our results can be found at \href{https://github.com/xiejx-math/ASHBM-codes.git}{https://github.com/xiejx-math/ASHBM-codes.git}.

\subsection{ Numerical setup}
\label{NS-Section61}

We examine the following two commonly used  probability spaces \cite{Nec19,necoara2022stochastic,Du20Ran}.

{\bf Uniform sampling:} We consider the uniform sampling of $p$ unique indices that make up the set $\mathcal{J}$, i.e. $\mathcal{J}\subset [m]$ and $|\mathcal{J}| = p$ for all samplings, with $p$ being the block size. Then, we set $S = \sqrt{m/p}I_{:,\mathcal{J}}/\|A\|_F$. It can be observed that the total number of possible choices  is  $m\choose p$ and $P(\mathcal{J})=1/{m\choose p}$ for any $\mathcal{J}$. In this case, Algorithm \ref{ibm-xie} and ASHBM (Algorithm \ref{ASHBM}) yield the \emph{randomized block Kaczmarz} method with \emph{uniform rule} (RBKU) and the RBKU with \emph{adaptive momentum} (AmRBKU), respectively. The term  $\sigma_{\min}^{2} (H_k^{\frac{1}{2}}A)/\lambda_{\max}^{(k)}$ in Theorems \ref{CRS-ASGD} and  \ref{CRS-AmSGD} here becomes
\begin{equation}\label{upperboundU}
	\frac{\sigma_{\min}^{2} (H_k^{\frac{1}{2}}A)}{\lambda_{\max}^{(k)}}=\frac{p}{m}\cdot\frac{\sigma_{\min}^2(A)}{\max\limits_{\mathcal{J}\subset[m],|\mathcal{J}|=p}\|A_{\mathcal{J},:}\|^2_2}.
\end{equation}

{\bf Partition sampling:} Consider the following partition of $[m]$
$$	
\begin{aligned}
	\mathcal{I}_i&=\left\{\varpi(k): k=(i-1)p+1,(i-1)p+2,\ldots,ip\right\}, i=1, 2, \ldots, t-1,
	\\
	\mathcal{I}_t&=\left\{\varpi(k): k=(t-1)p+1,(t-1)p+2,\ldots,m\right\}, |\mathcal{I}_t|\leq p,
\end{aligned}$$
where $\varpi$ is a uniform random permutation on $[m]$ and $p$ is the block size. 
We select an index $ i \in [t] $ with a probability of $\|A_{\mathcal{I}_i,:}\|^2_F/\|A\|^2_F$, and then set $S=I_{:,\mathcal{I}_i}/\|A_{\mathcal{I}_i,:}\|_F$. 
In this case, Algorithm \ref{ibm-xie} and ASHBM (Algorithm \ref{ASHBM}) yield the randomized average block Kaczmarz (RABK) method \cite{Nec19,Du20Ran} and  the RABK with \emph{adaptive momentum} (AmRABK), respectively.
The term  $\sigma_{\min}^{2} (H_k^{\frac{1}{2}}A)/\lambda_{\max}^{(k)}$ in Theorems \ref{CRS-ASGD} and  \ref{CRS-AmSGD} here writes as
\begin{equation}\label{upperboundP}
	\frac{\sigma_{\min}^{2} (H_k^{\frac{1}{2}}A)}{\lambda_{\max}^{(k)}}=\frac{\sigma_{\min}^2(A)}{\|A\|^2_F\cdot\max_{i\in[t]} \frac{\|A_{\mathcal{I}_i,:}\|^2_2}{\|A_{\mathcal{I}_i,:}\|^2_F}}.
\end{equation}

We consider two types of coefficient matrices. One is randomly generated Gaussian matrices by using the {\sc Matlab} function {\tt randn}. Specifically, for given $m, n, r$, and $\kappa>1$, we construct a dense matrix $A$ by $A=U D V^\top$, where $U \in \mathbb{R}^{m \times r}, D \in \mathbb{R}^{r \times r}$, and $V \in \mathbb{R}^{n \times r}$. Using {\sc Matlab}  notation, these matrices are generated by {\tt [U,$\sim$]=qr(randn(m,r),0)}, {\tt [V,$\sim$]=qr(randn(n,r),0)}, and {\tt D=diag(1+($\kappa$-1).*rand(r,1))}. So the condition number and the rank of $A$ are upper bounded by $\kappa$ and $r$, respectively. The other is  real-world data which are available via SuiteSparse Matrix Collection \cite{Kol19} and LIBSVM \cite{chang2011libsvm}, where only the coefficient matrices $A$ are employed in the experiments.

In our implementations, to ensure the consistency of the linear system, we first generate the solution by $x^*={\tt randn(n,1)}$ and then set $b=Ax^*$. All computations are initialized with $x^0=0$. 
The computations are terminated once the relative solution error (RSE), defined as
$\text{RSE}=\|x^k-A^\dagger b\|^2_2/\|x^0-A^\dagger b\|^2_2$, is less than a specific error tolerance or the number of iterations exceeds a certain limit. 
In practice, we consider $\|S_k^\top (Ax^k-b)\|_2$ as zero when it is less than  \texttt{eps}.
In our test, we set $\zeta_k=1$ for both RABK and RBKU (see Algorithm \ref{ibm-xie}).
For each experiment, we run $50$ independent trials.

\subsection{Choice of the block size $p$} 

The optimal block size selection for the RABK method has been analyzed by Necoara, assuming that the matrix $A$ is normalized, i.e., $\|A_{i,:}\|_2=1$ for all $i\in[m]$. Necoara estimated the term $\max_{i\in[t]} \frac{\|A_{\mathcal{I}_i,:}\|^2_2}{\|A_{\mathcal{I}_i,:}\|^2_F}$ in \eqref{upperboundP} and demonstrated that the optimal block size is approximately $p^* \backsimeq m/\|A\|^2_2$, see \cite[Section 4.1]{necoara2022stochastic} and \cite[Section 4.3]{Nec19}. 

In this subsection, we investigate the impact of the block size $p$ on the convergence of RBKU, AmRBKU, RABK, and AmRABK for general matrices through numerical experiments, particularly for randomly generated Gaussian matrices.
We present three different scenarios: well-conditioned, relatively ill-conditioned, and rank-deficient coefficient matrices. 
The performance of the algorithms is measured in both the computing time (CPU) and 
the number of full iterations $(k\cdot\frac{p}{m})$, which makes the number of operations for one pass through the rows of $A$ are the same for all the algorithms. 
The results are displayed in Figure \ref{figureR1}. 
The bold line illustrates the median value derived from $50$ trials. The lightly shaded area signifies the  range from the minimum to the maximum values, while the darker shaded one indicates the data lying between the $25$-th and $75$-th quantiles.  

\begin{figure}[tbhp]
	\centering
	\begin{tabular}{cc}
		\includegraphics[width=0.31\linewidth]{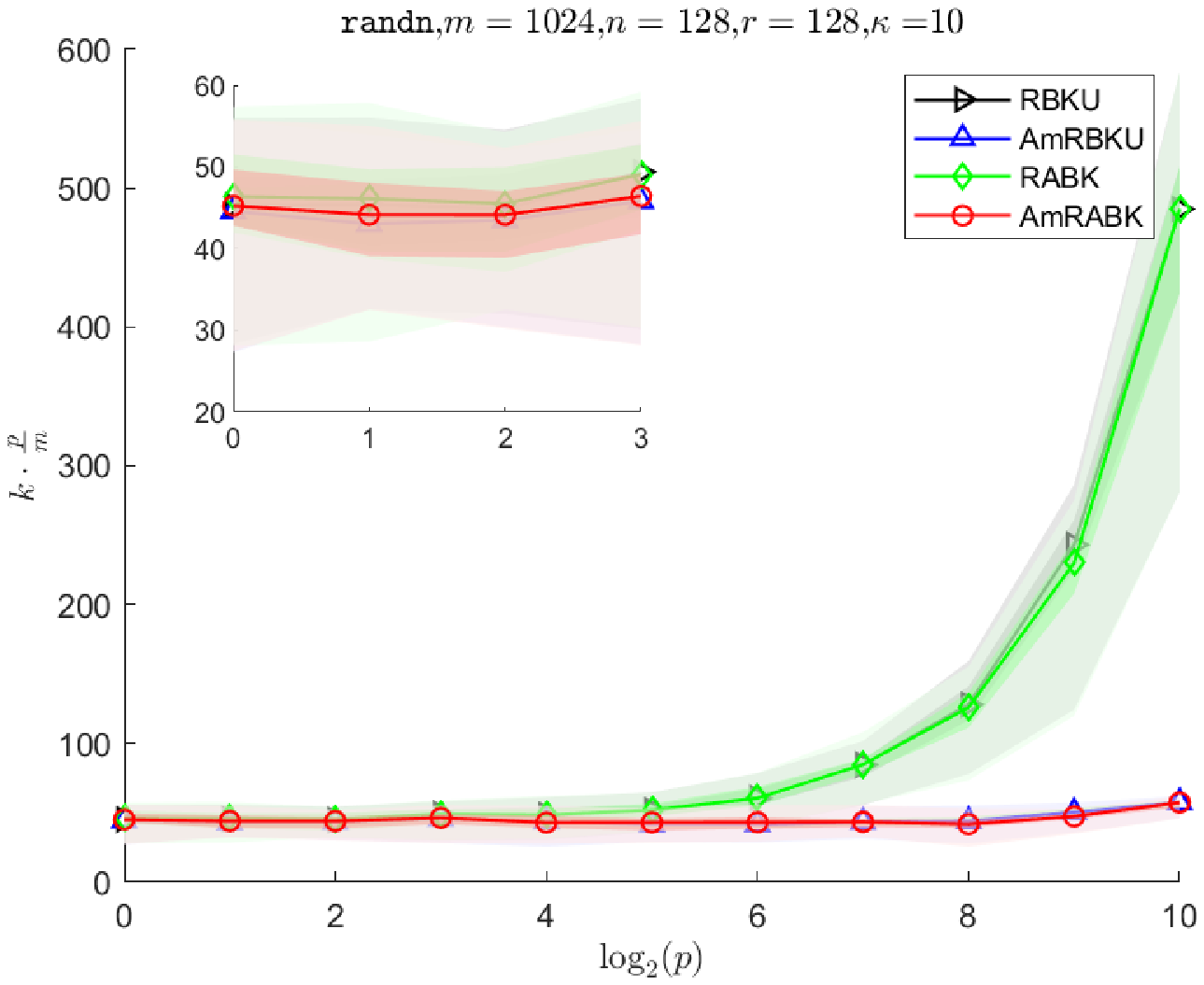}
		\includegraphics[width=0.31\linewidth]{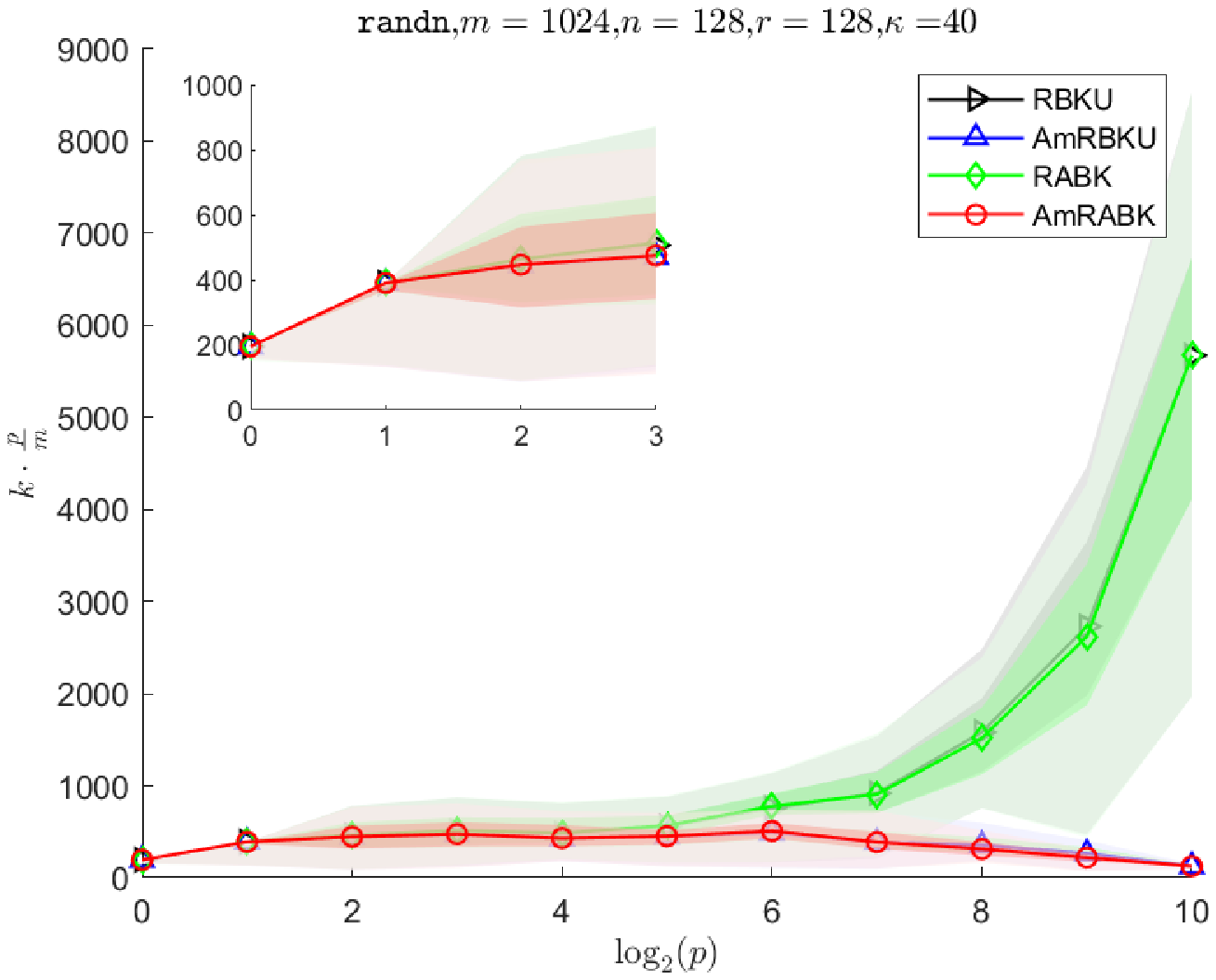}
		\includegraphics[width=0.31\linewidth]{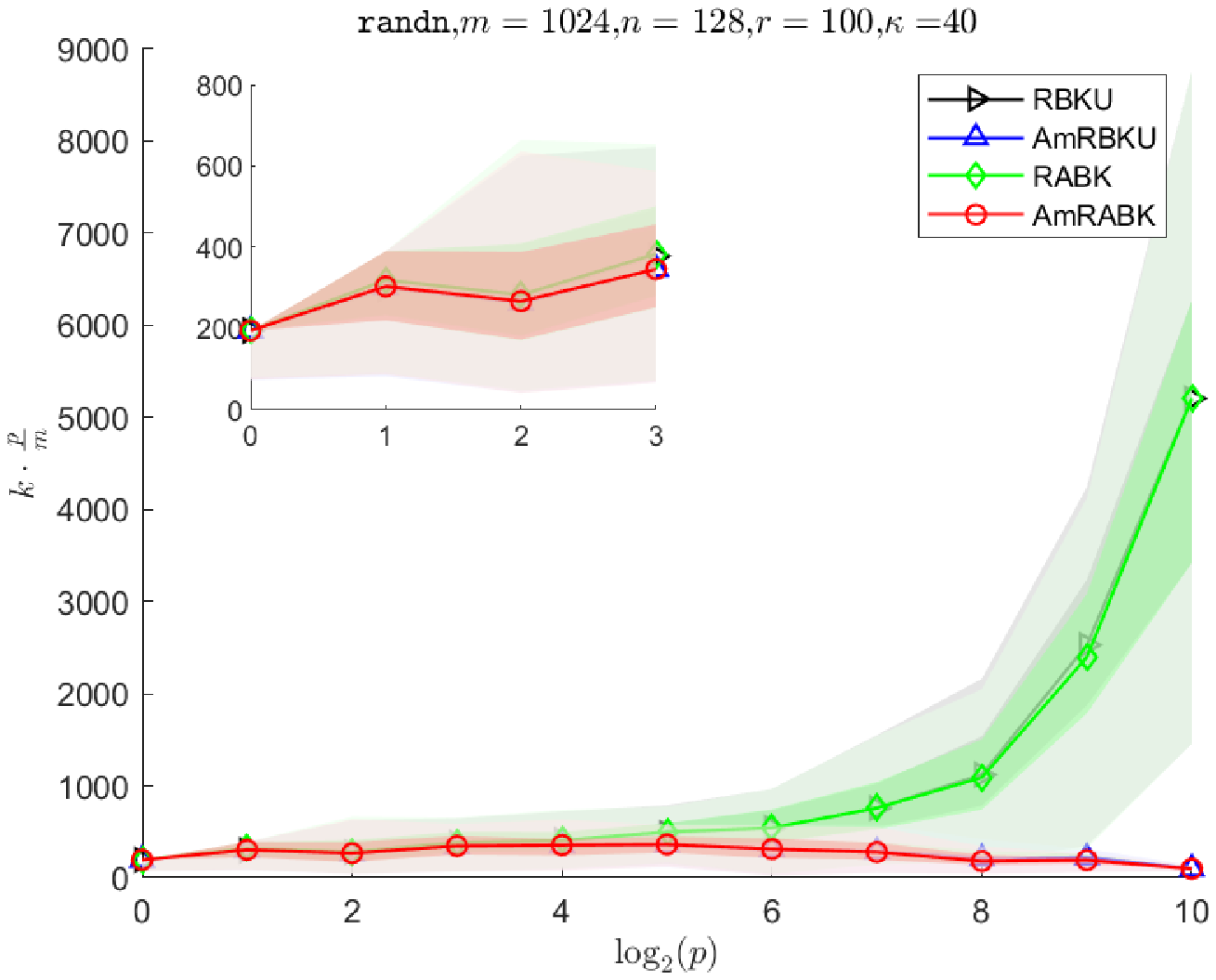}\\
		\includegraphics[width=0.31\linewidth]{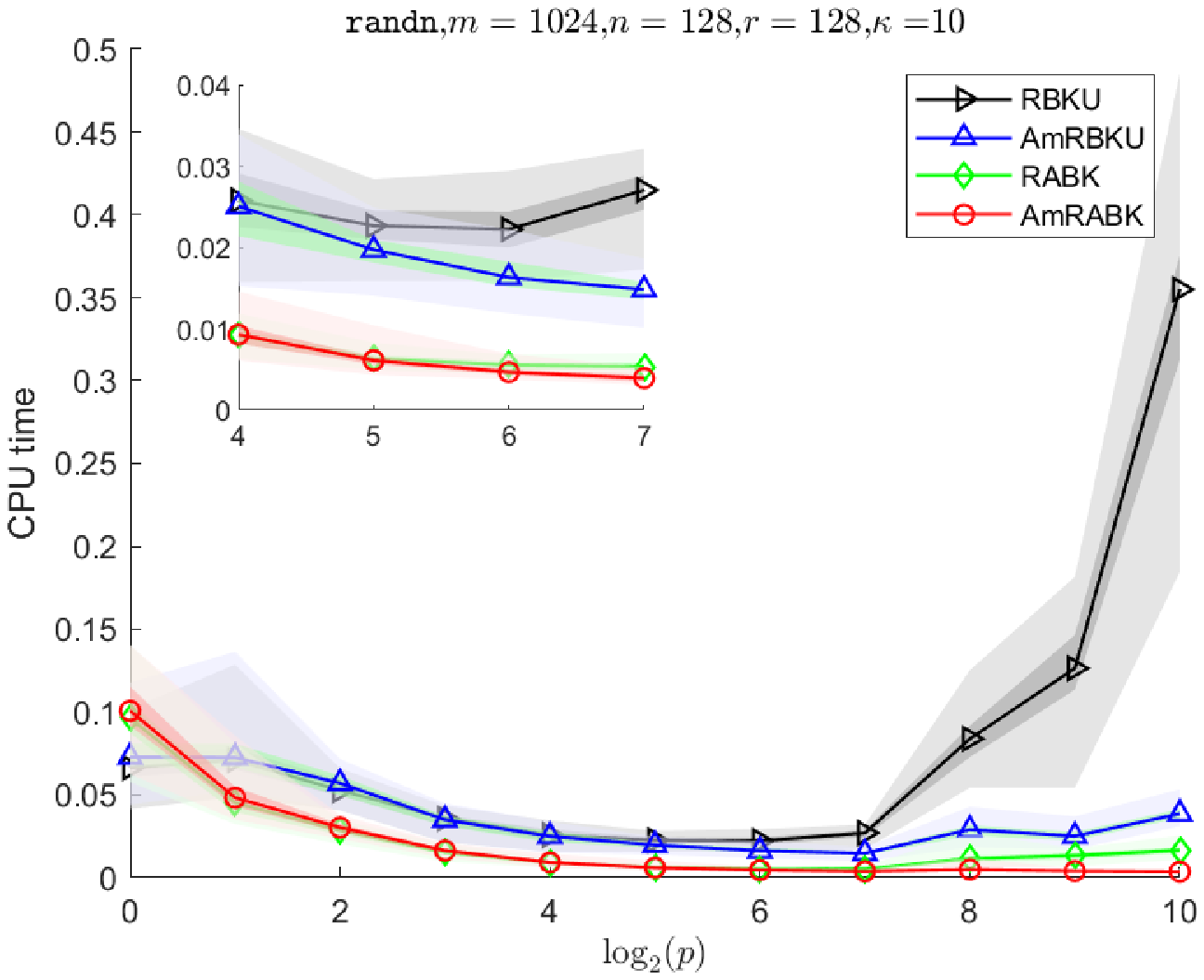}
		\includegraphics[width=0.31\linewidth]{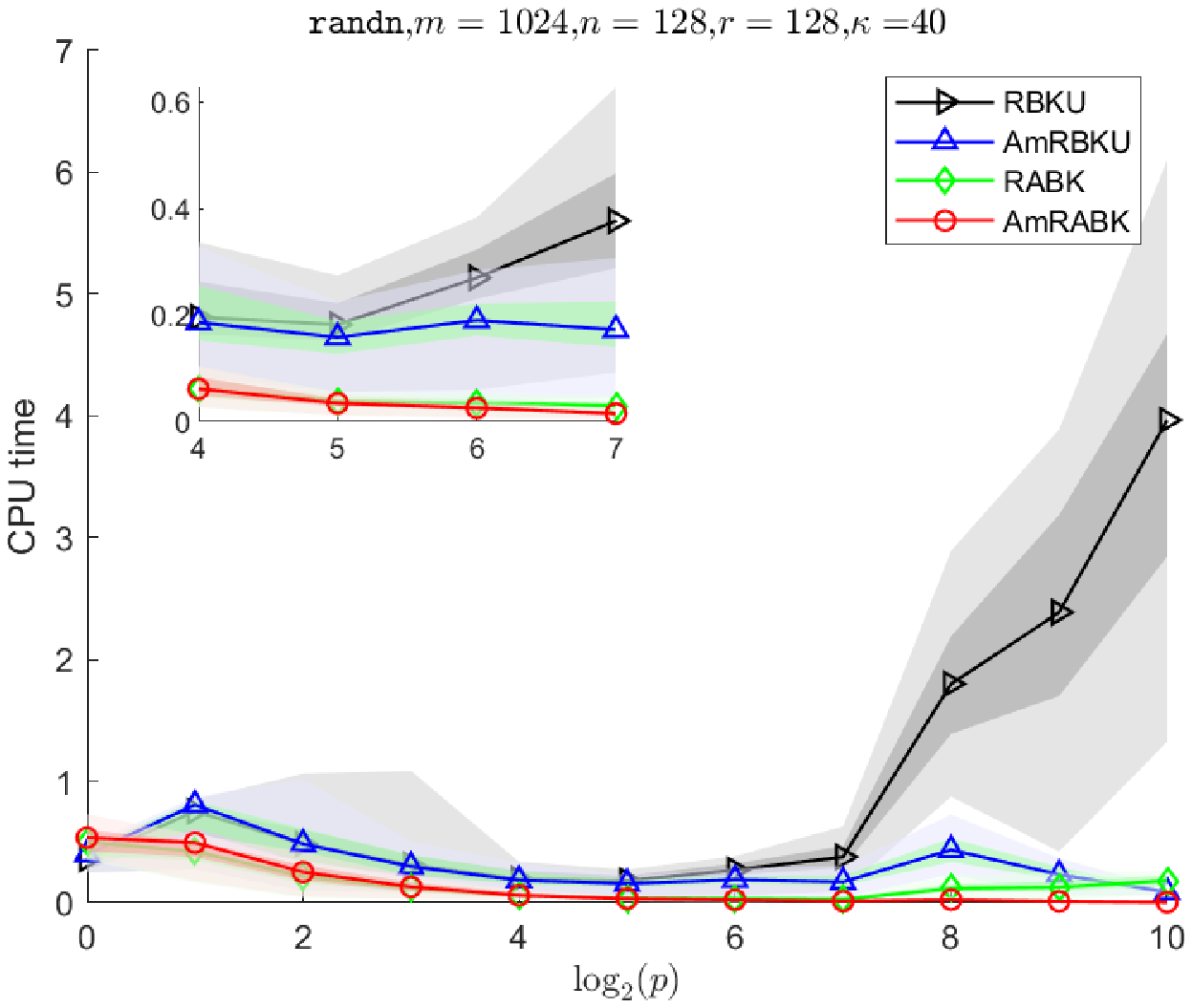}
		\includegraphics[width=0.31\linewidth]{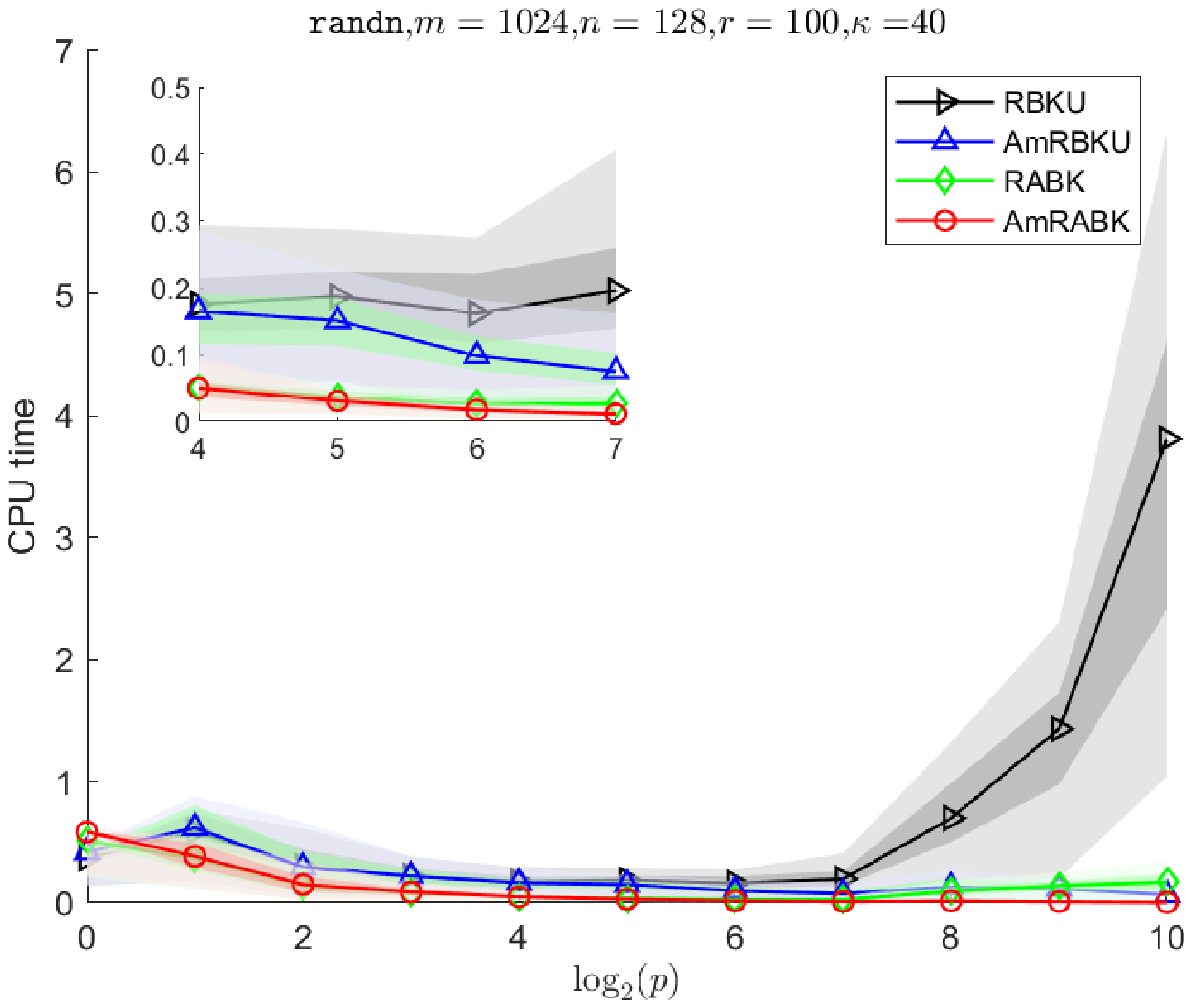}
	\end{tabular}
	\caption{
		Figures depict the evolution of the number of full iterations (top) and the CPU time (bottom) with respect to the block size $p$.  The title of each plot indicates the values of $m,n,r$, and $\kappa$. All computations are terminated once $\text{RSE}<10^{-12}$. }
	\label{figureR1}
\end{figure}

It can be seen from Figure \ref{figureR1} that for any fixed value of $p$, both RBKU and RABK demonstrate a similar number of full iterations, which is also the case for AmRBKU and  AmRABK. However,  AmRBKU and AmRABK generally require fewer full iterations than RBKU  and  RABK, especially when the block size is large.
When regarding CPU time, AmRABK requires less computation than AmRBKU  if $p\geq2$, which is consistent with the trend between RABK and RBKU.
This can be attributed to the increasing time consumption of row extraction from matrix $A$ for both RBKU and AmRBKU, as the value of $p$ increases. However,  RABK and AmRABK can overcome this issue by storing the submatrices of $A$ determined by the partition at the beginning, eliminating the need for subsequent row extraction from matrix $A$. 

Figure \ref{figureR1} also presents that the values of $p$ significantly affect the number of full iterations of both  RBKU  and  RABK.  
Conversely, the effect of $p$ on the convergence of AmRBKU and AmRABK is relatively negligible.
To provide a more clear visualization, we
depict the actual convergence factor of  RABK  and  AmRABK  in Figure \ref{figureR2}. Here, the actual convergence factor is defined as
$$\rho = \left(\frac{\|x^K-A^\dagger b\|^2_2}{\|x^0-A^\dagger b\|^2_2}\right)^{1/K},$$
where $K$ is the number of iterations when the algorithm terminates.
In Figure \ref{figureR2}, we also include the plot of the upper bound
\begin{equation}
	\label{upper-bound}
	1-\frac{\sigma_{\min}^2(A)}{\|A\|^2_F\cdot\max_{i\in[t]} \frac{\|A_{\mathcal{I}_i,:}\|^2_2}{\|A_{\mathcal{I}_i,:}\|^2_F}}
\end{equation}
derived from  \eqref{upperboundP} for  both  RABK  and AmRABK.  Note that the actual convergence factor of  RBKU and AmRBKU is not plotted due to the computational impracticality of obtaining the upper bound derived from \eqref{upperboundU}.
As seen in Figure \ref{figureR2}, both  RABK and AmRABK display smaller convergence factors than that indicated by the upper bound \eqref{upper-bound}. In addition, the actual convergence factor of AmRABK  is less than that of  RABK, and as $p$ increases, it decreases more significantly. 

\begin{figure}[tbhp]
	\centering
	\begin{tabular}{cc}
		\includegraphics[width=0.31\linewidth]{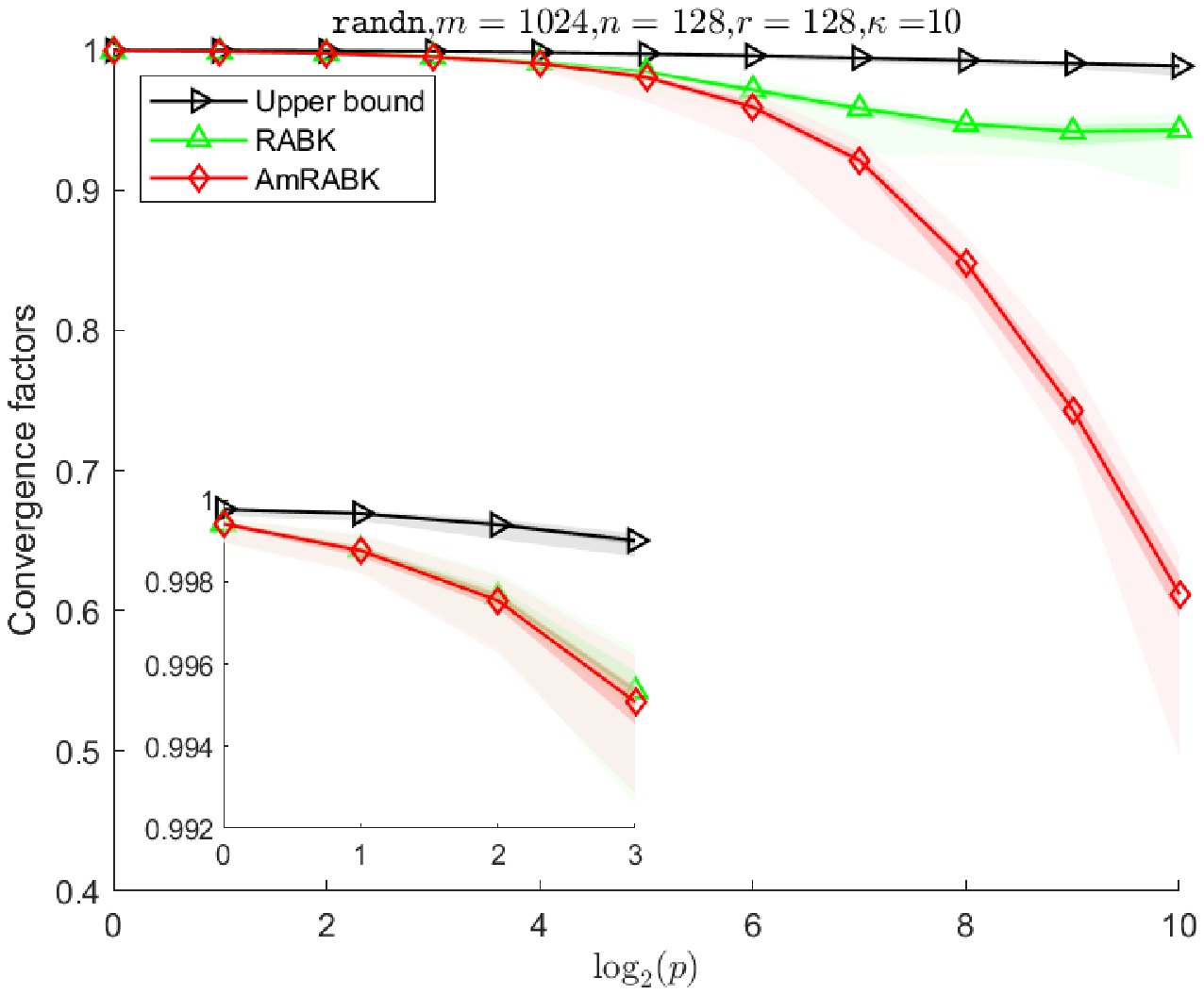}
		\includegraphics[width=0.31\linewidth]{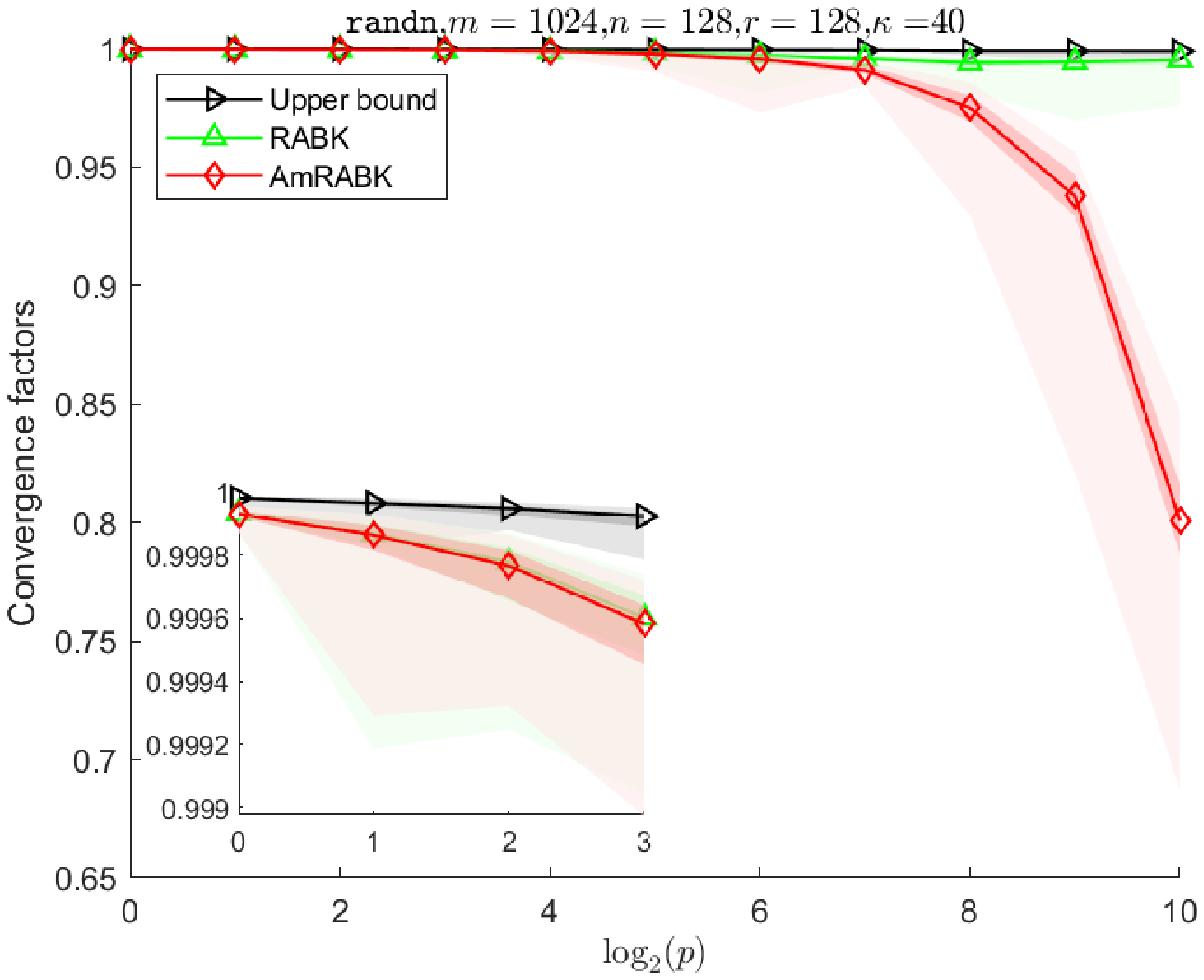}
		\includegraphics[width=0.31\linewidth]{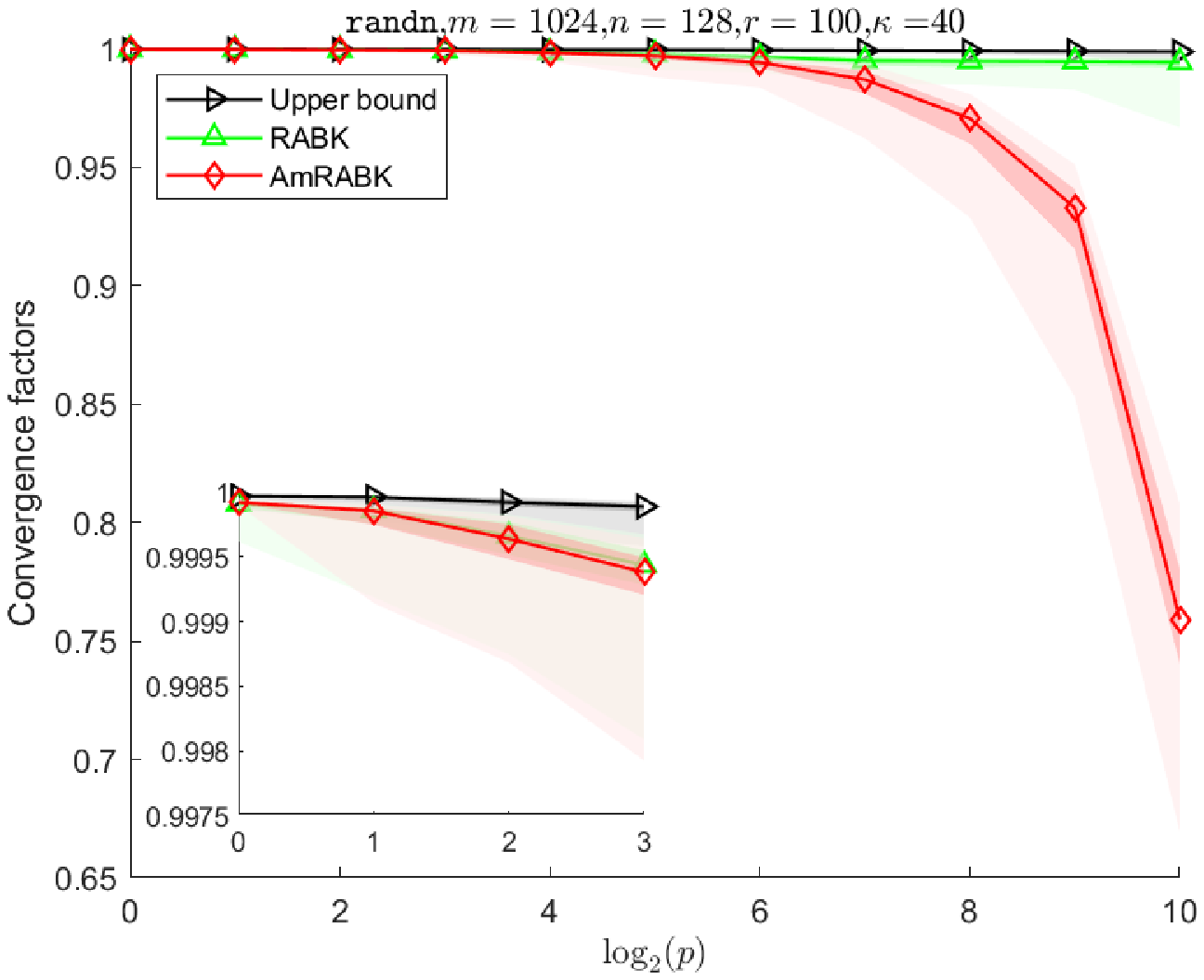}
	\end{tabular}
	\caption{Figures depict the evolution of the convergence factor with respect to the block size $p$.  The title of each plot indicates the values of $m,n,r$, and $\kappa$. All computations are terminated once $\text{RSE}<10^{-12}$.}
	\label{figureR2}
\end{figure}

To evaluate the performance improvement via the actual convergence factor, in Figure \ref{figureR3}, we present the the evolution of the RSE over the number of iterations for RABK and AmRABK, as well as the worst-case convergence bound derived from \eqref{upper-bound}.
Evidently, both RABK and AmRABK converge significantly faster than the worst-case convergence suggested by the upper bound.

Finally, we note that during our test, it has been consistently observed that the method derived from partition sampling surpasses the method derived from uniform sampling in terms of CPU time. Consequently, our subsequent tests will solely focus on the method derived from partition sampling.

\begin{figure}[tbhp]
	\centering
	\begin{tabular}{cc}
		\includegraphics[width=0.31\linewidth]{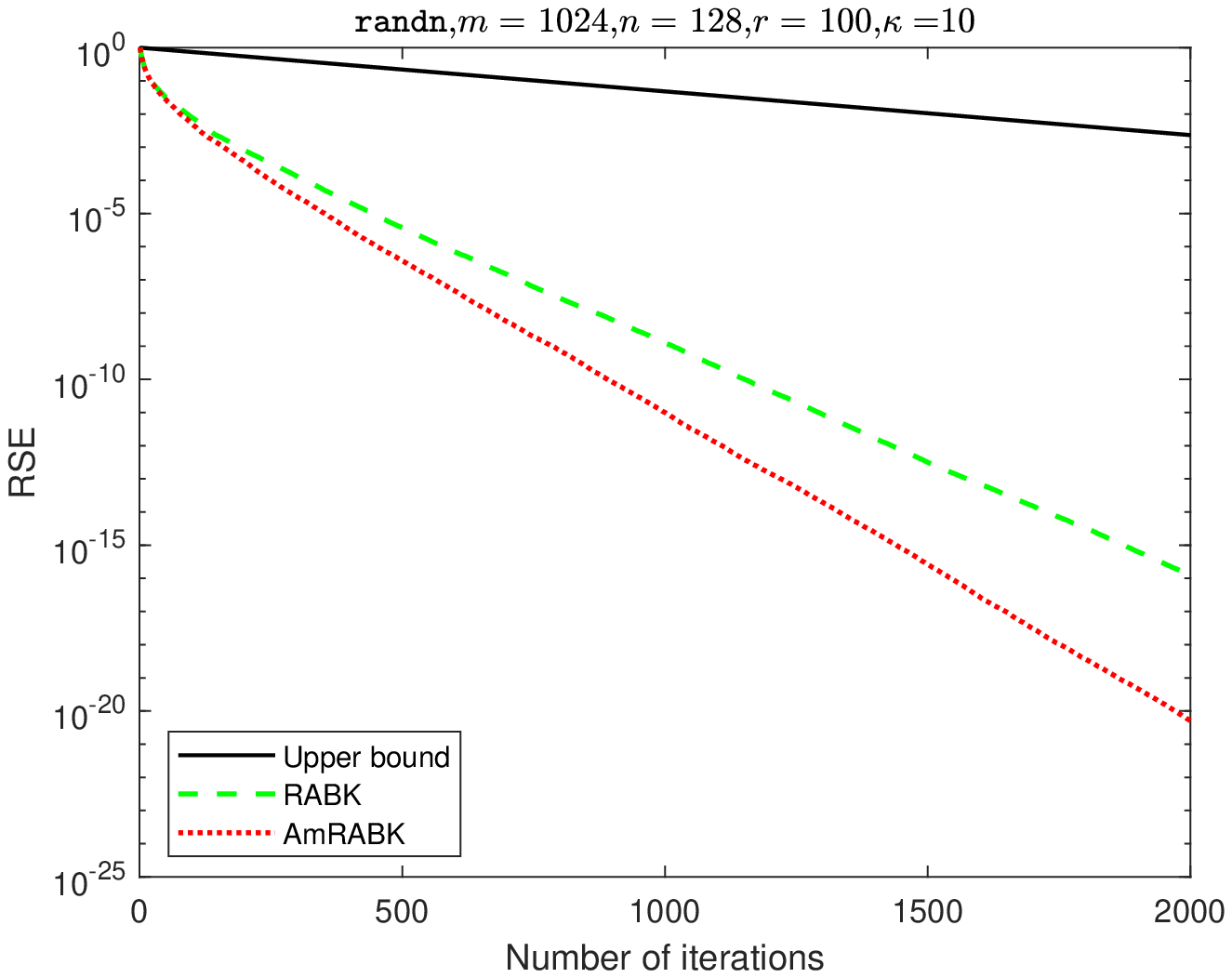}
		\includegraphics[width=0.31\linewidth]{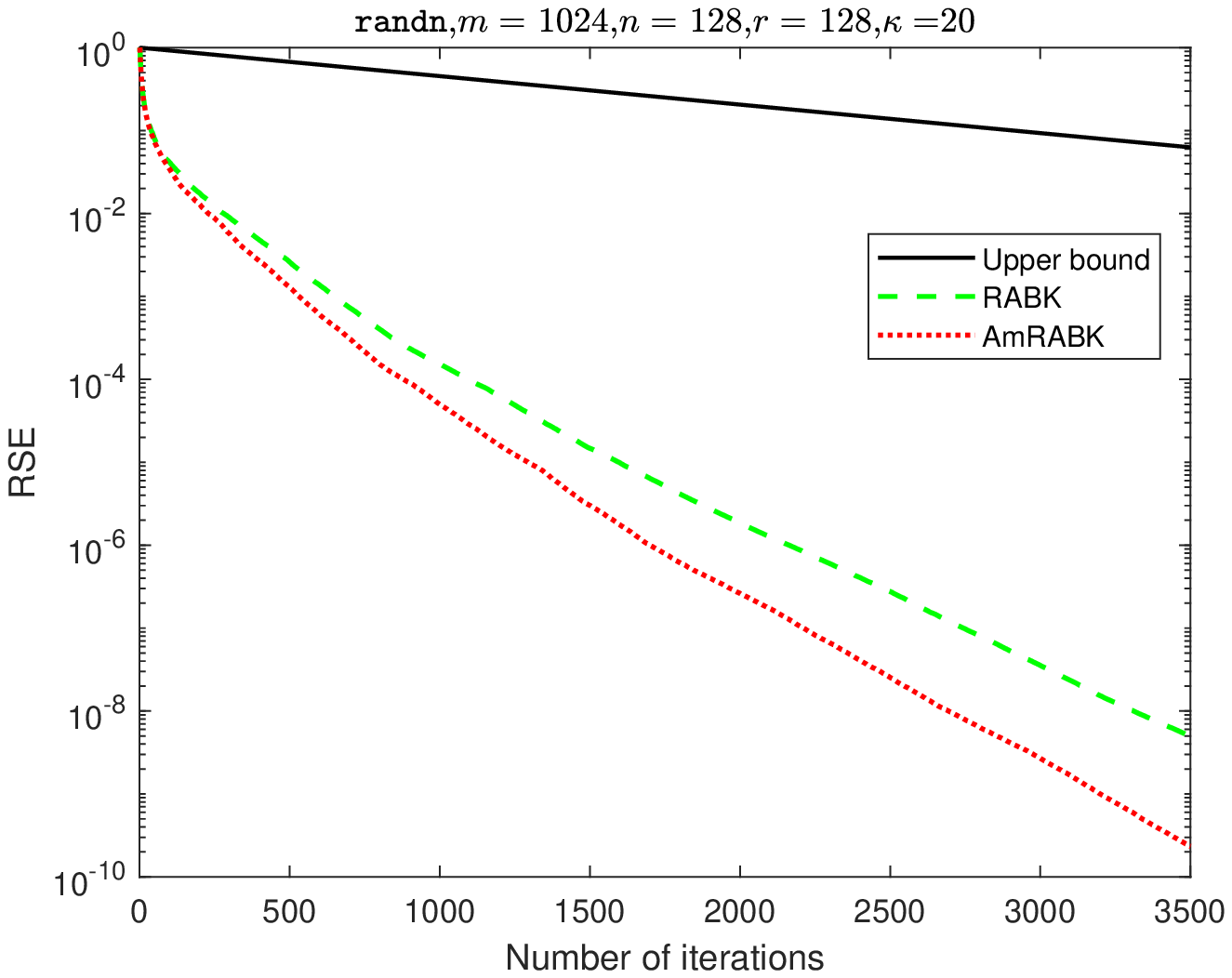}
		\includegraphics[width=0.31\linewidth]{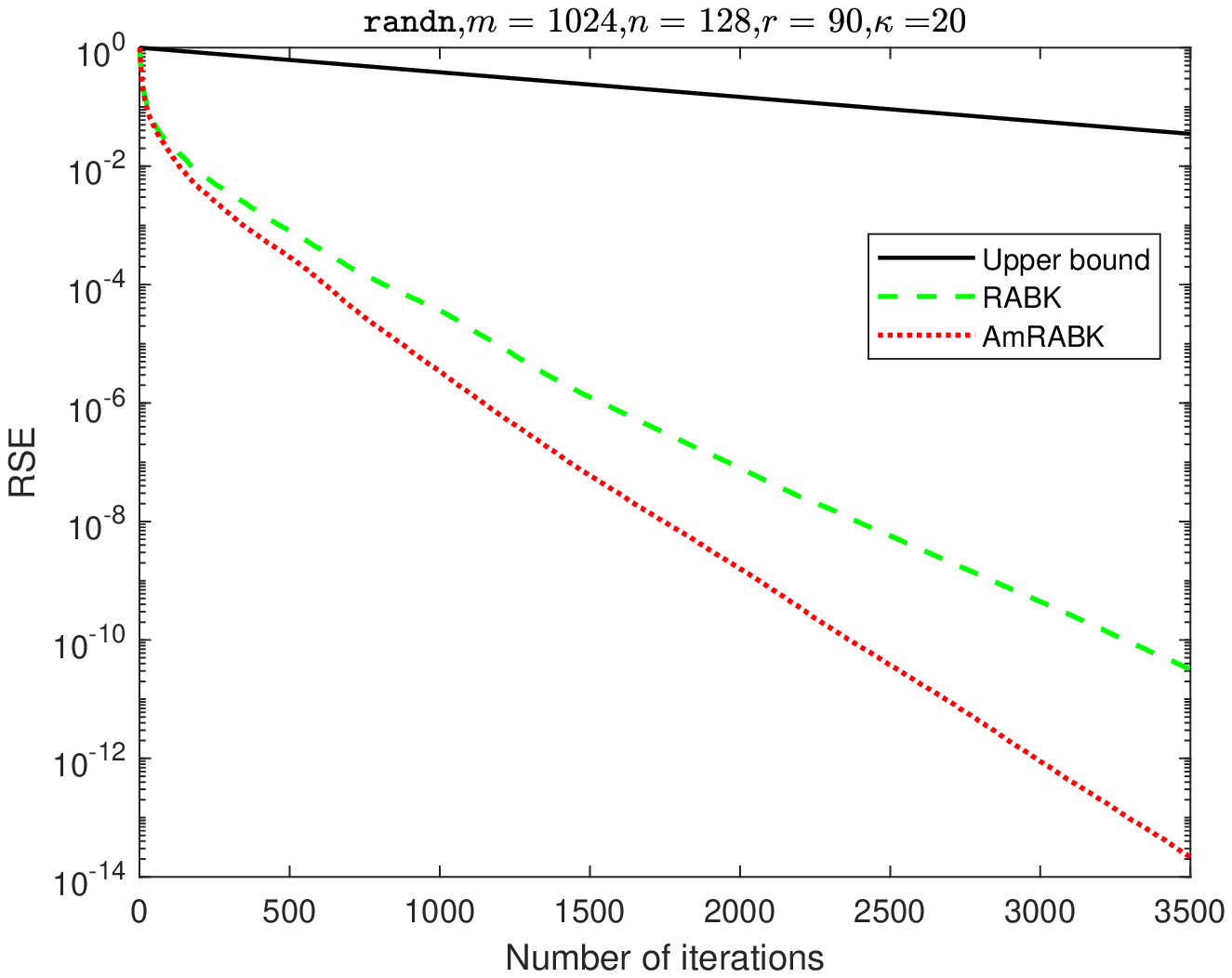}
	\end{tabular}
	\caption{
		The evolution of RSE with respect to the number of iterations.  The title of each plot indicates the values of $m,n,r$, and $\kappa$.  We set  $p=32$ and stop the algorithms if the number of iterations exceeds a certain limit.  The solid curve shows the upper bound \eqref{upper-bound}.}
	\label{figureR3}
\end{figure}

\subsection{Comparison to the method with time-invariant parameters}
In this subsection, we compare the performance of RABK and AmRABK with the method proposed in \cite{loizou2020momentum,han2022pseudoinverse}, which employs time-invariant parameters $\alpha$ and $\beta$.
The method proposed in \cite{loizou2020momentum,han2022pseudoinverse} with partition sampling now
leads to the RABK with momentum (mRABK). According to \cite[Theorem 4.7]{han2022pseudoinverse}, the step-size $\alpha$ for mRABK  can be chosen in $\left(0,\frac{2}{\tau\|A\|^2_F}\right)$, where  
\begin{equation}
	\label{step-size-tau}
	\tau=\left\|\mathbb{E}\left[\frac{I_{:,\mathcal{I}} I_{:,\mathcal{I}}^\top AA^\top I_{:,\mathcal{I}} I_{:,\mathcal{I}}^\top}{\|A_{\mathcal{I},:}\|^4_F}\right]\right\|_2
	=\frac{1}{\|A\|^2_F}\cdot\max_{i\in[t]} \frac{\|A_{\mathcal{I}_i,:}\|^2_2}{\|A_{\mathcal{I}_i,:}\|^2_F}.
\end{equation}
We note that this step-size for  mRABK is aligns with the step-size proposed in \cite[Lemma 2.6]{Du20Ran}. 
In addition, the selection of appropriate momentum parameters $\beta$ is also 
crucial for achieving efficient convergence of  mRABK.
In our test, we will utilize specific choice of $\alpha = \frac{1}{\tau\|A\|^2_F}$ and $\beta$ is obtained through numerical experiments that can achieve faster convergence.

We examine the performance of RABK, mRABK, and AmRABK in terms of number of iterations and CPU time across varying condition numbers of $A$. Figure \ref{figure1} presents the results for the case where $A$ is full rank while Figure \ref{figure2} illustrates the case where it is rank-deficient.
The results demonstrate that AmRABK  outperforms  RABK  in terms of both iteration counts and CPU time, regardless of the condition number of $A$ and whether it is full rank.
It can also be observed that both RABK and  AmRABK  achieve better performance than  mRABK in terms of CPU times, even though mRABK requires fewer iterations than that of RABK. This is because mRABK involves estimation of the step-size ($\tau$ in \eqref{step-size-tau}) at each step.

\begin{figure}[tbhp]
	\centering
	\begin{tabular}{cc}
		\includegraphics[width=0.31\linewidth]{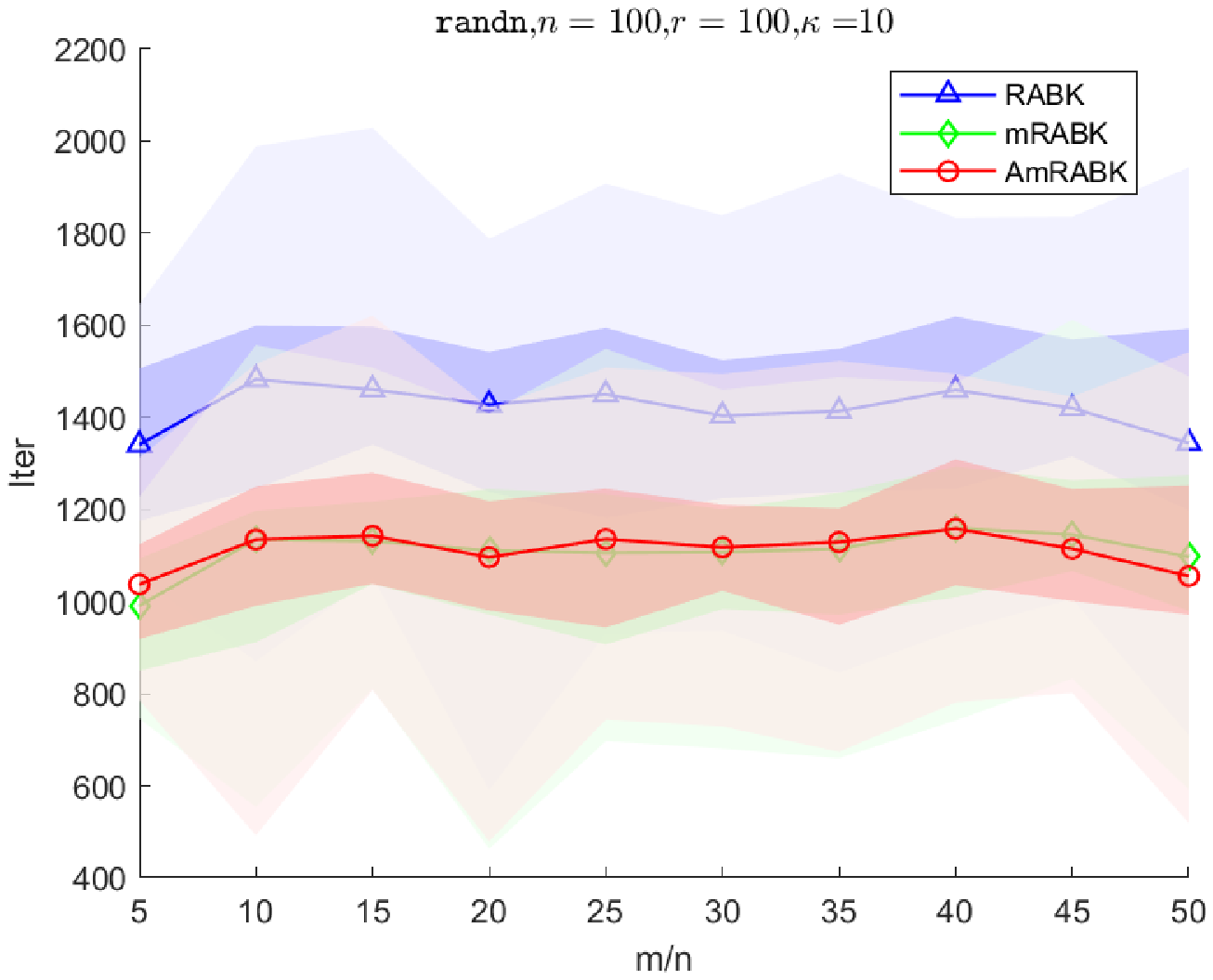}
		\includegraphics[width=0.31\linewidth]{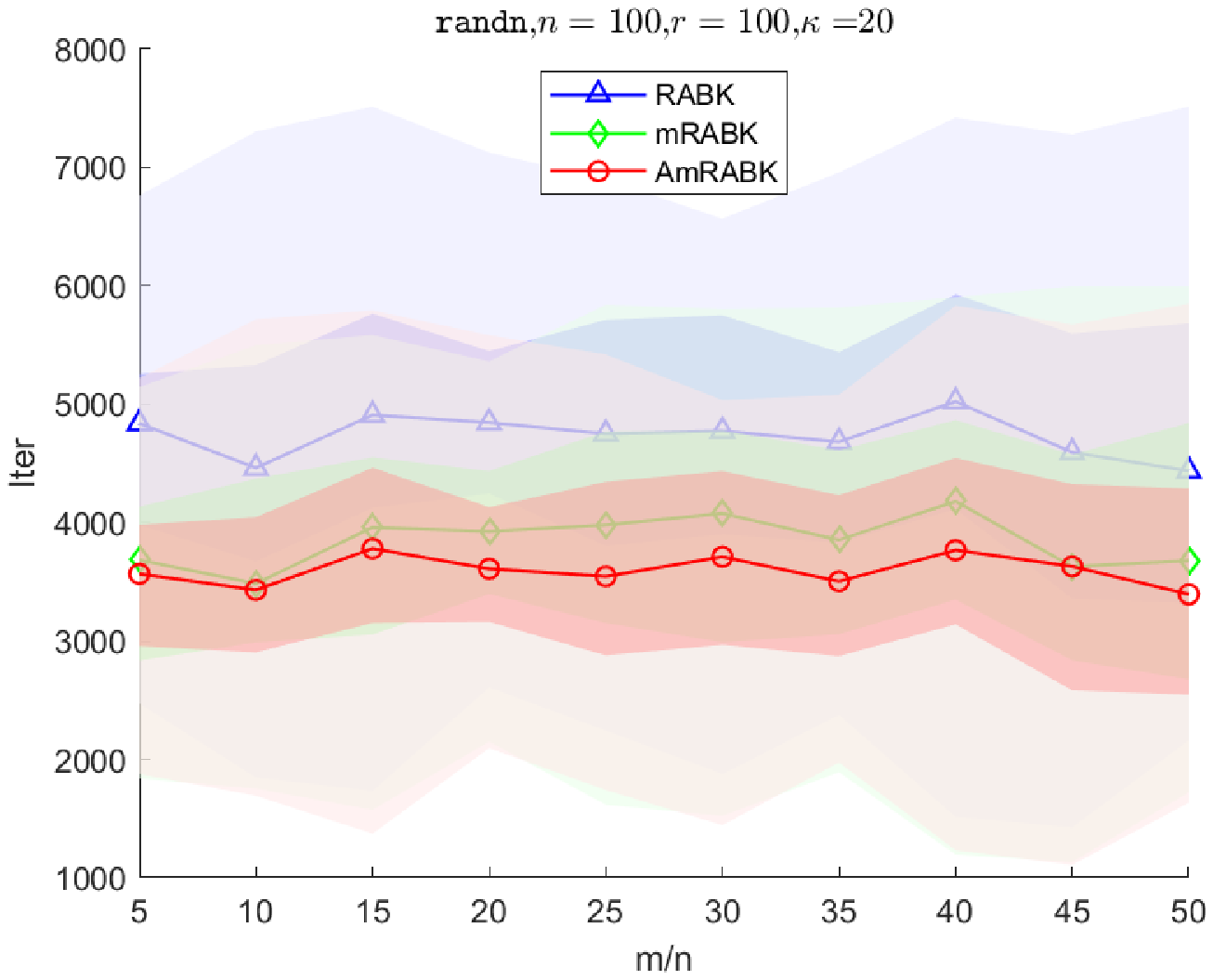}
		\includegraphics[width=0.31\linewidth]{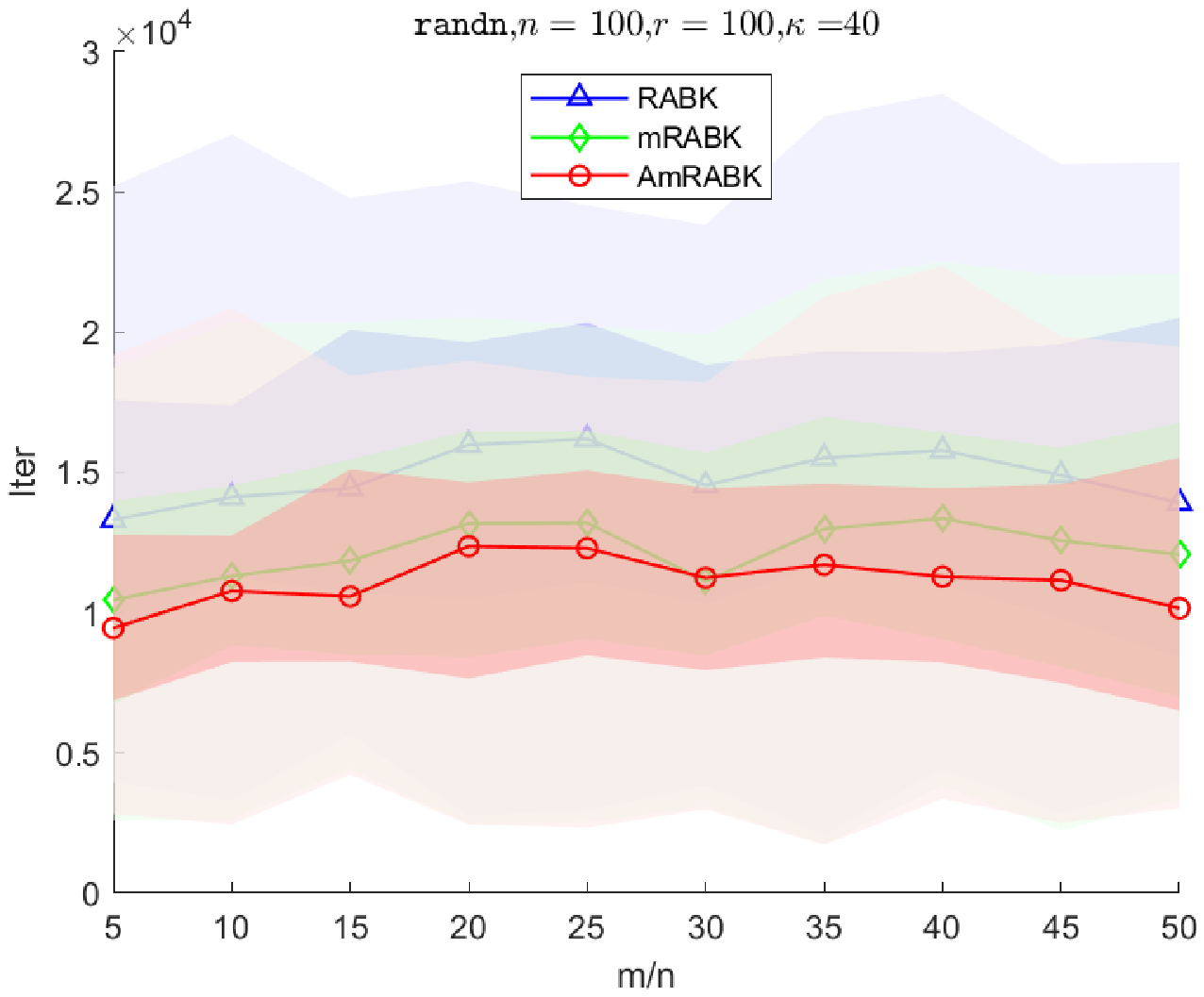}\\
		\includegraphics[width=0.31\linewidth]{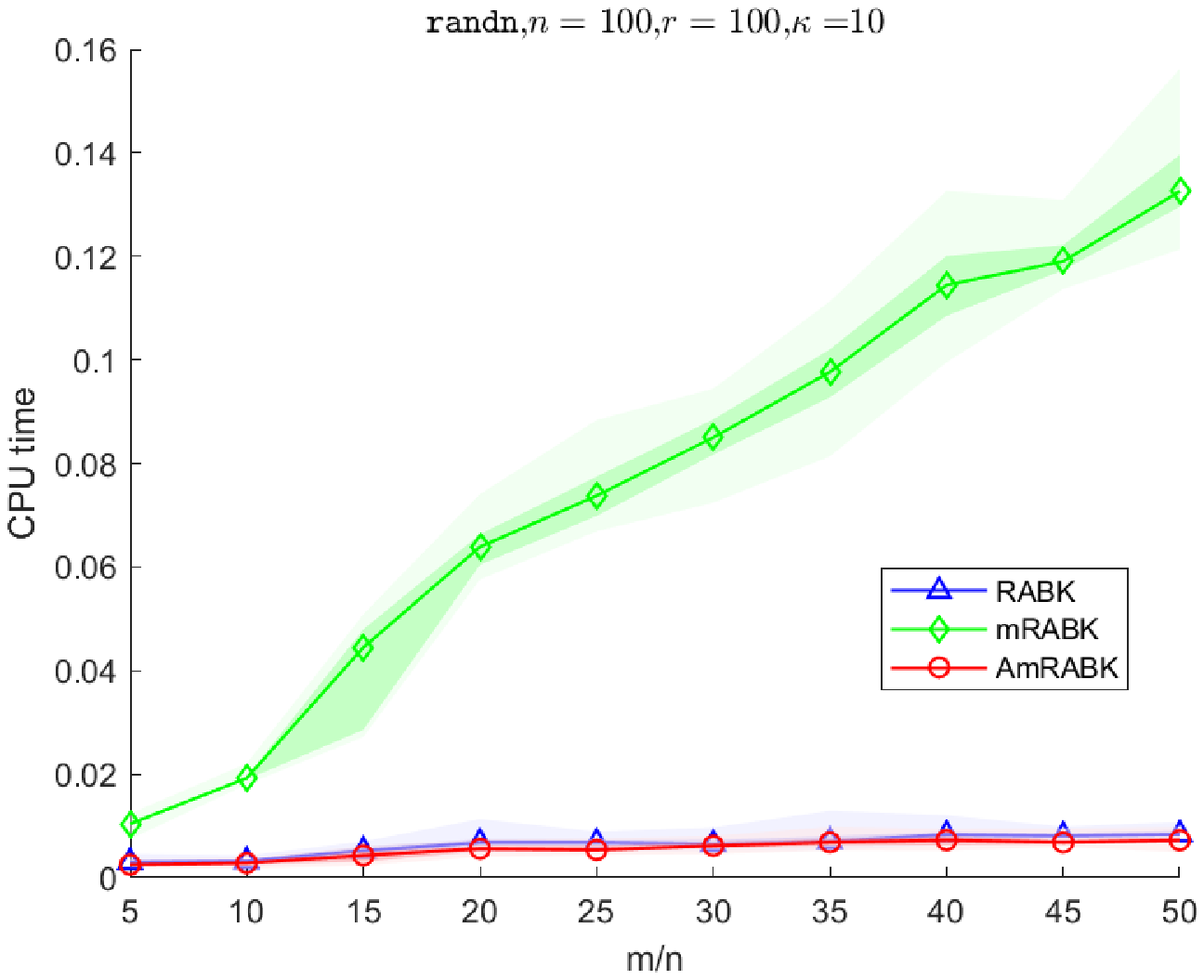}
		\includegraphics[width=0.31\linewidth]{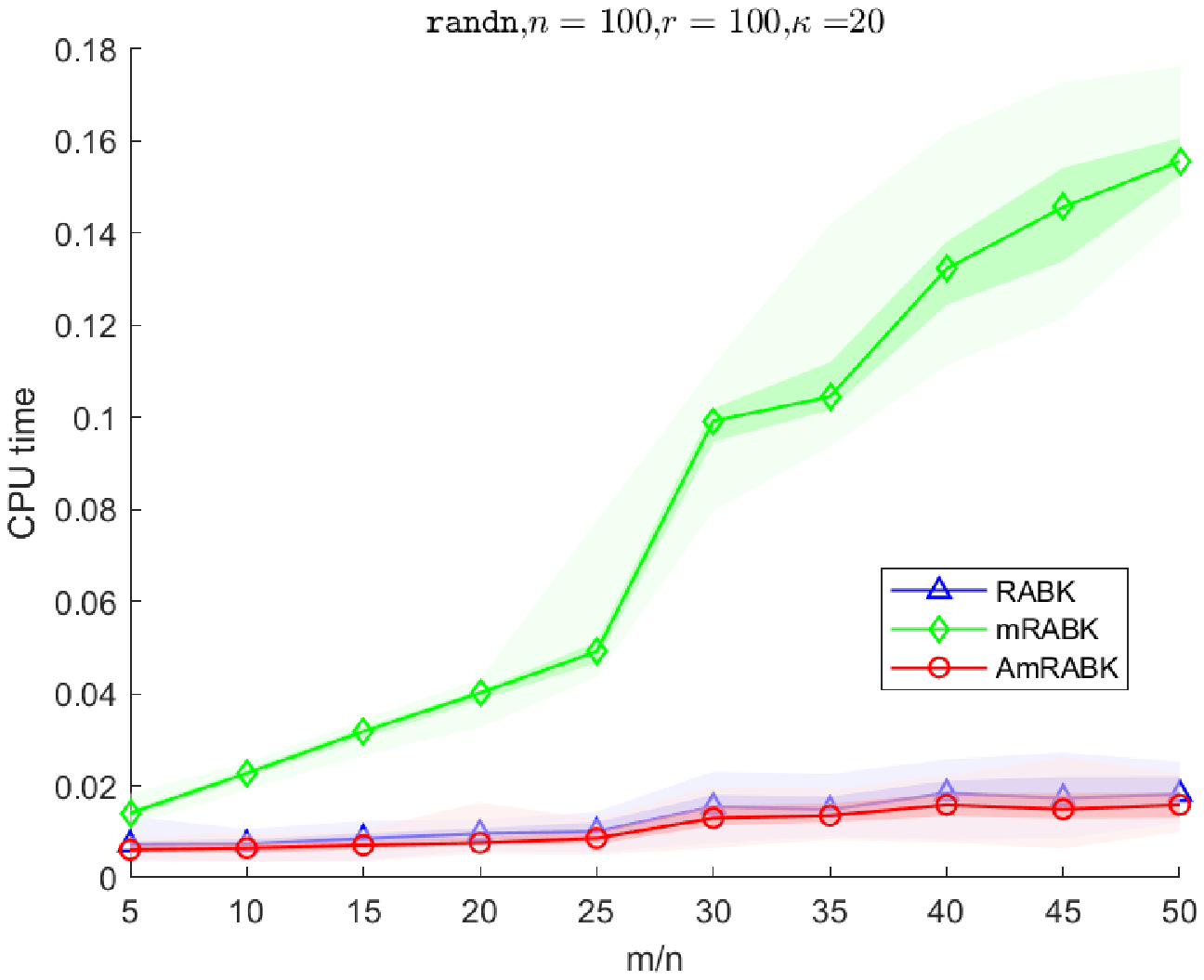}
		\includegraphics[width=0.31\linewidth]{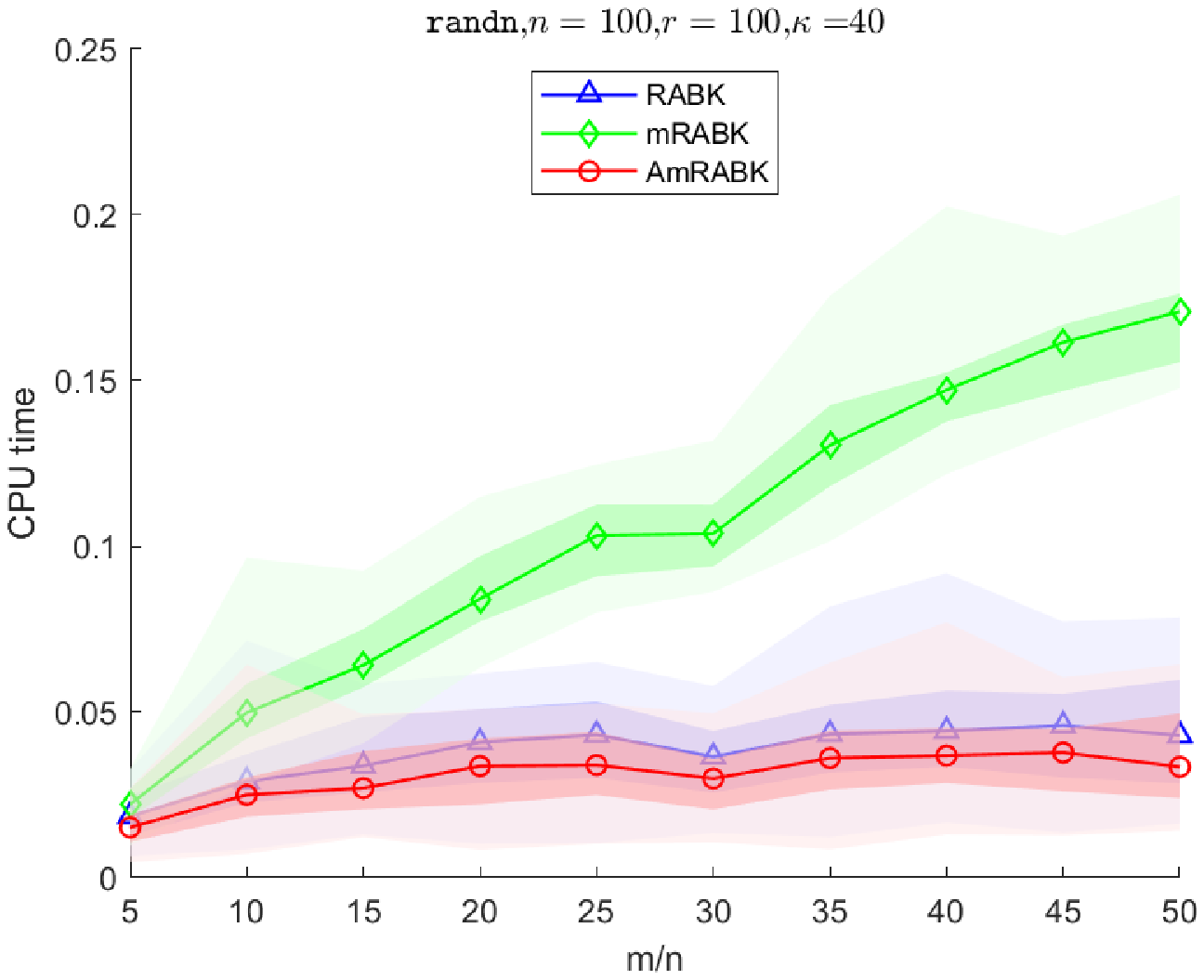}
	\end{tabular}
	\caption{Performance of RABK, mRABK, and AmRABK for linear systems with full rank Gaussian matrix.
		Figures depict the iteration and the CPU time (in seconds) vs increasing number of rows.  The title of each plot indicates the values of $n$, $r$, and $\kappa$. We set $p=30$ and $\beta=0.7$ for mRABK. All computations are terminated once $\text{RSE}<10^{-12}$.}
	\label{figure1}
\end{figure}

\begin{figure}[tbhp]
	\centering
	\begin{tabular}{cc}
		\includegraphics[width=0.31\linewidth]{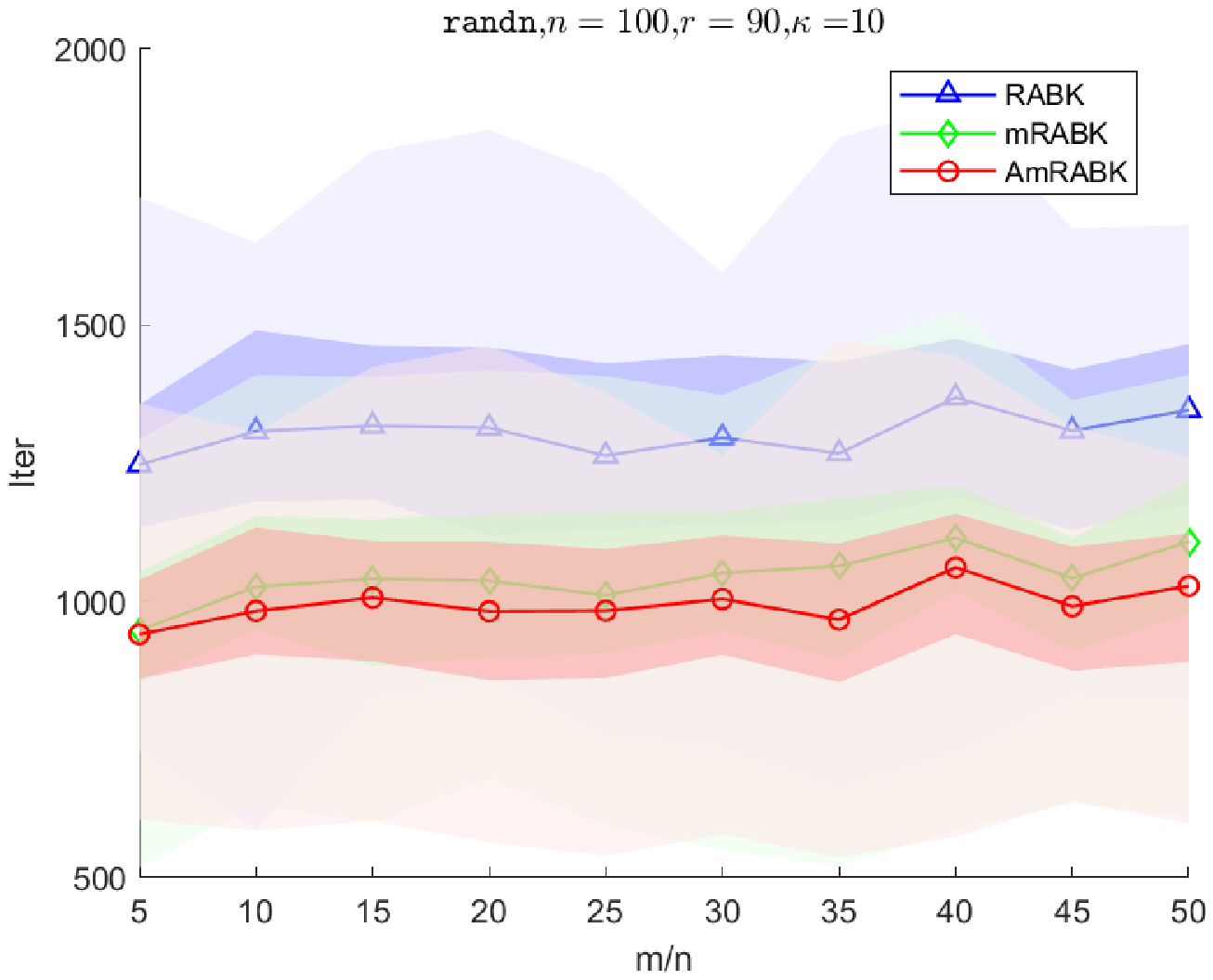}
		\includegraphics[width=0.31\linewidth]{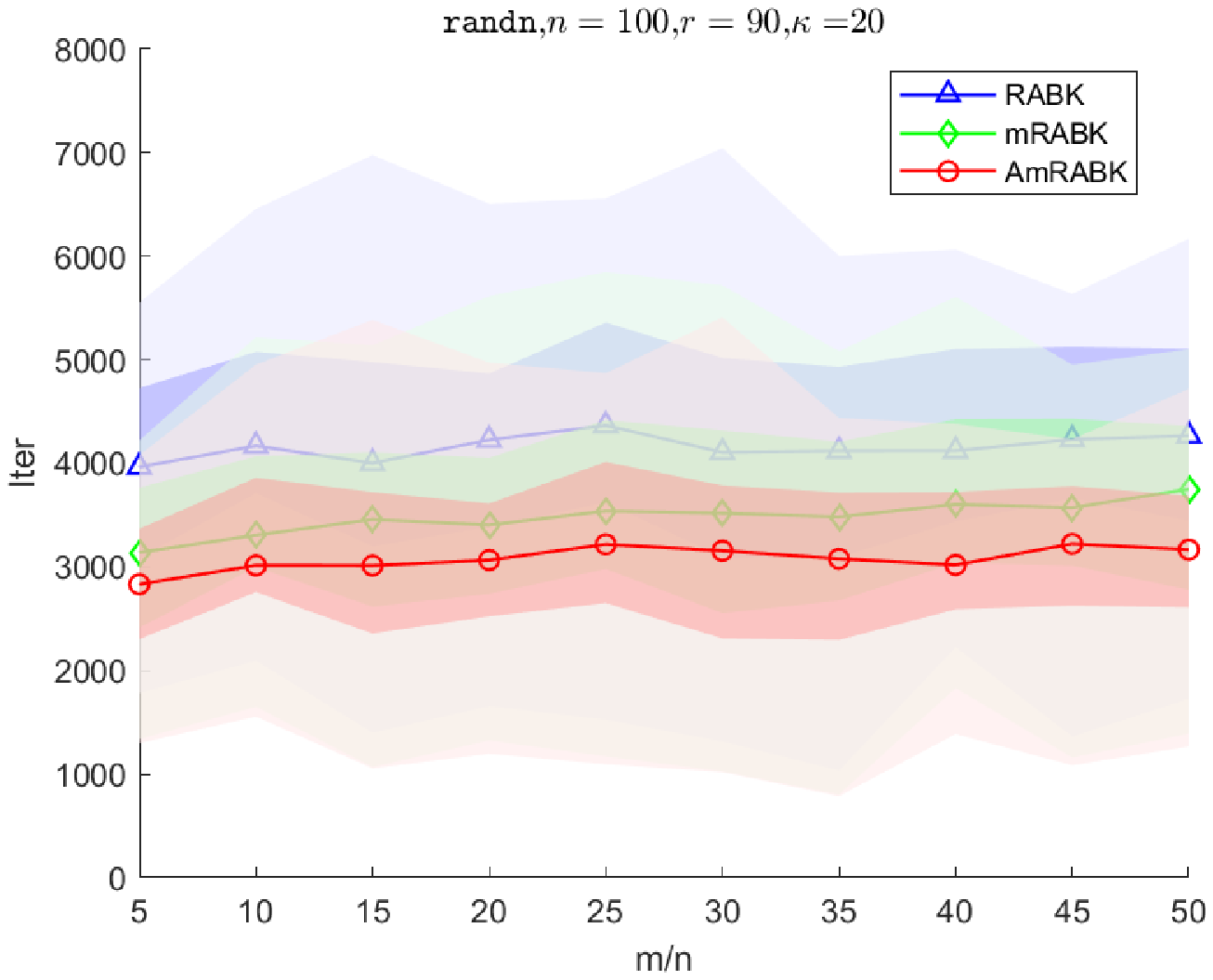}
		\includegraphics[width=0.31\linewidth]{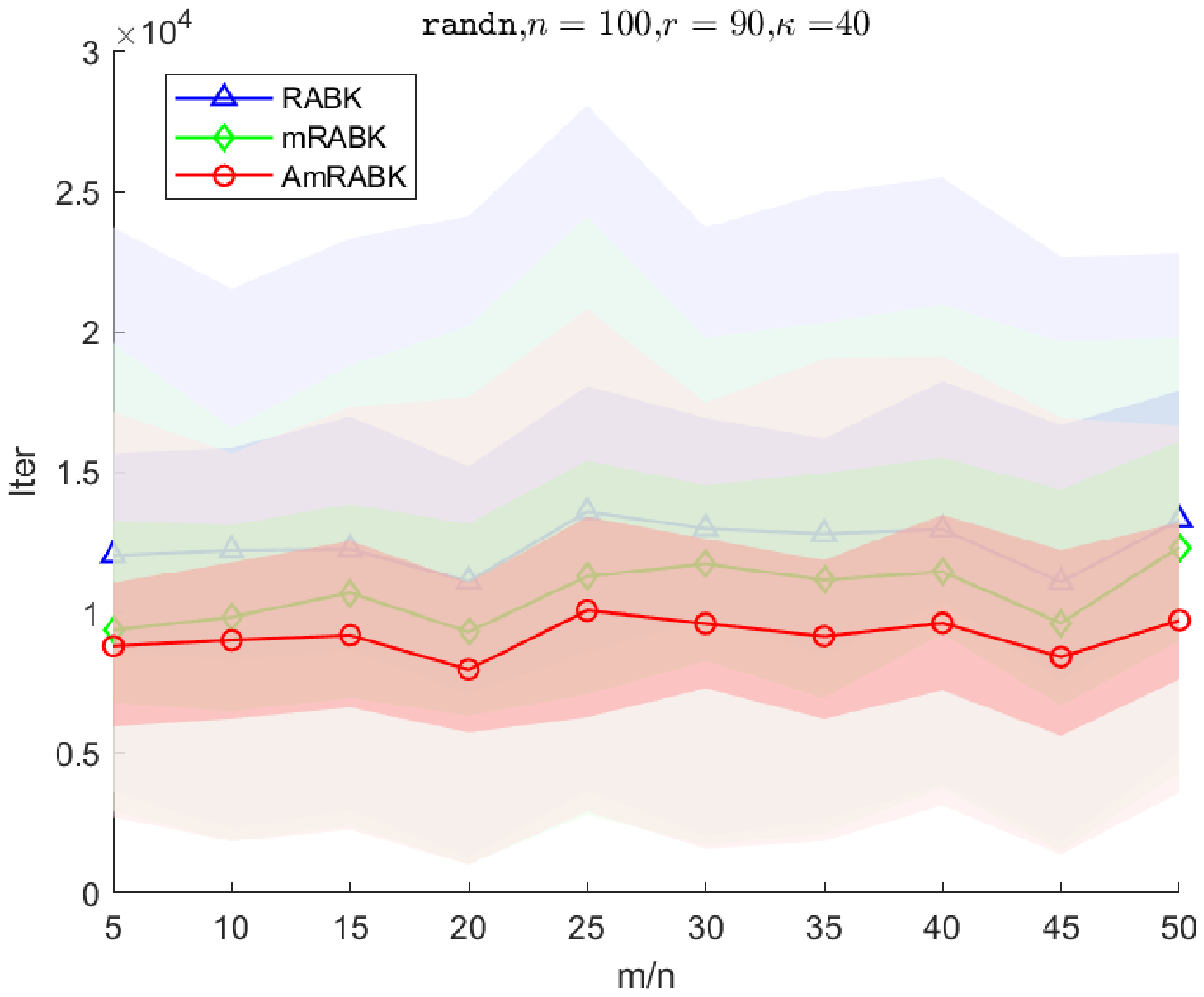}\\
		\includegraphics[width=0.31\linewidth]{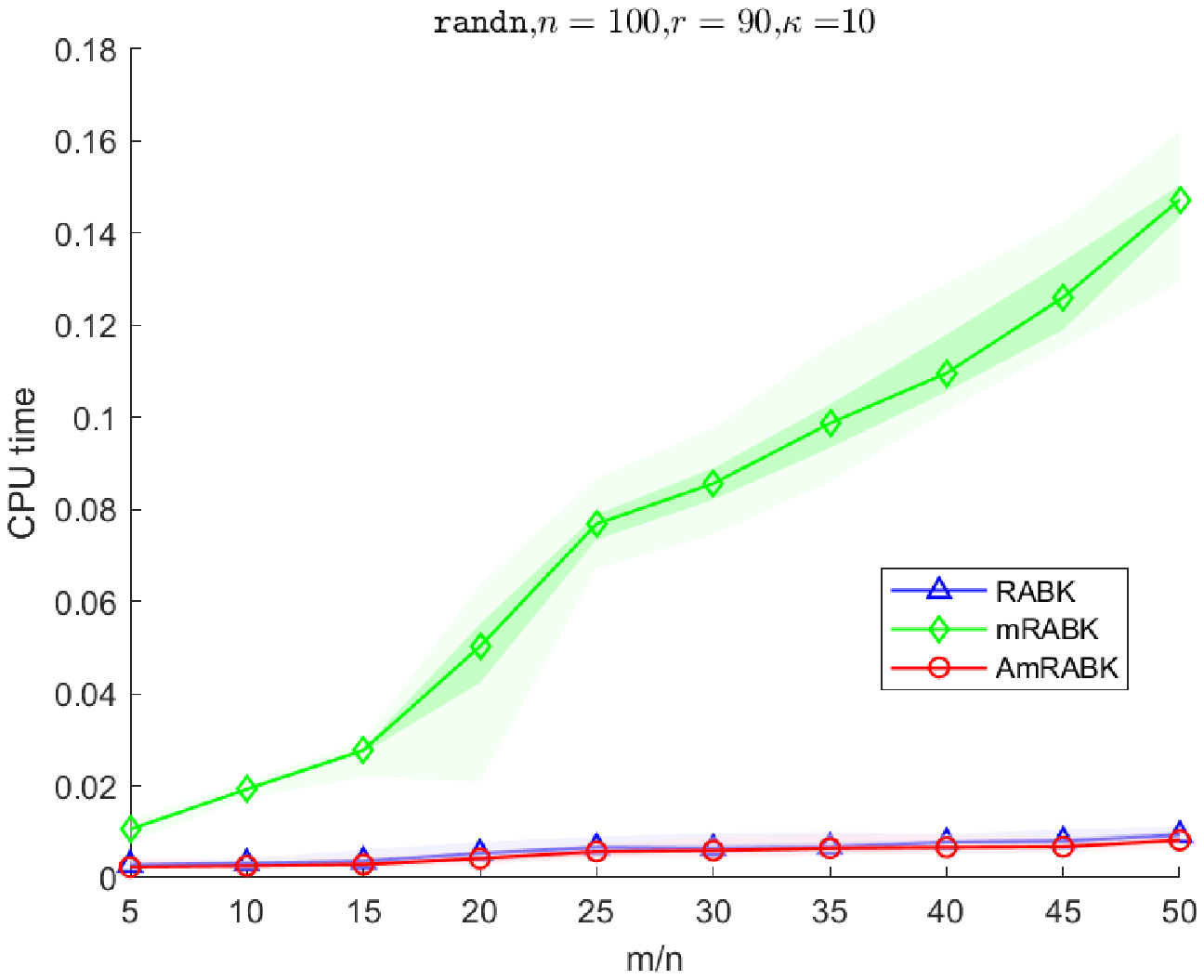}
		\includegraphics[width=0.31\linewidth]{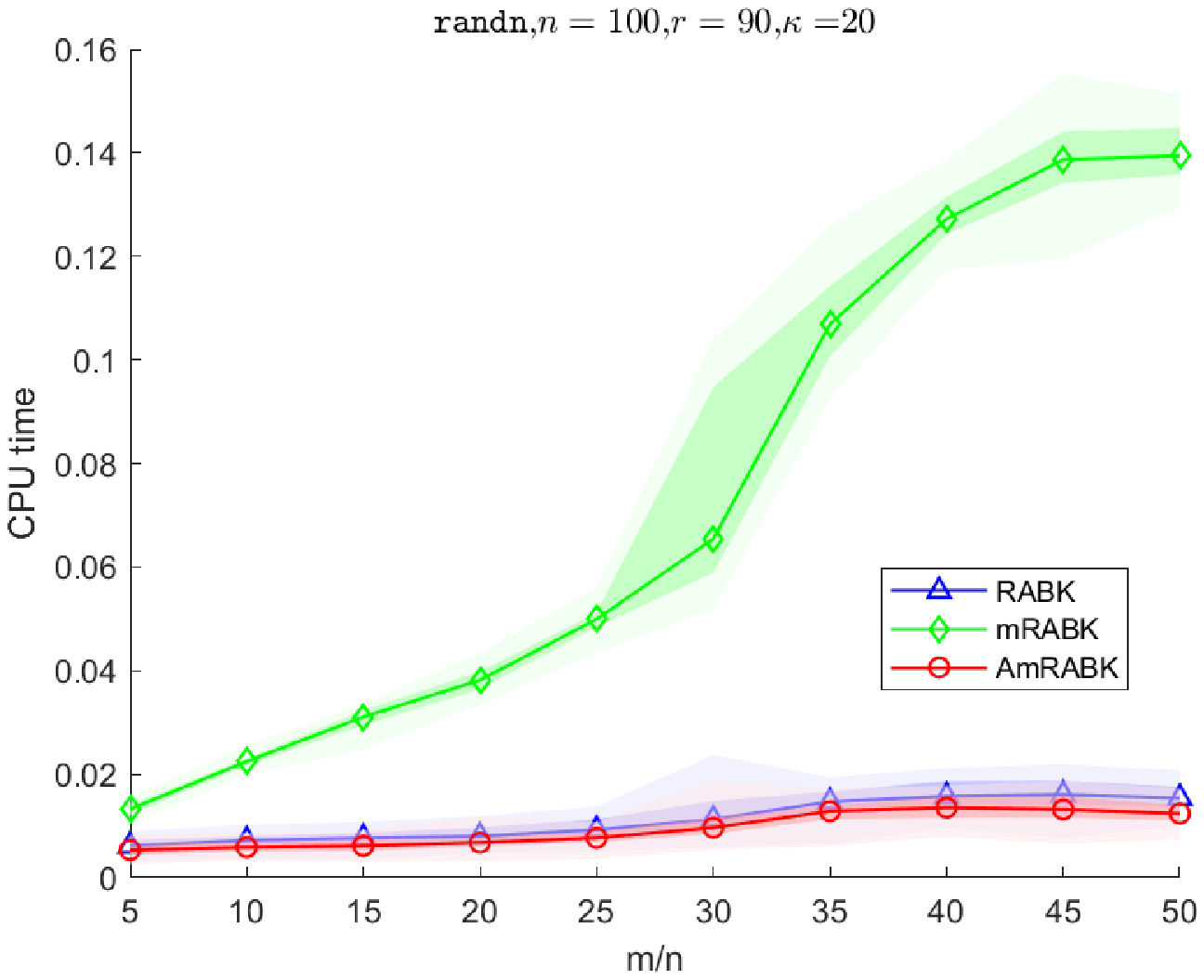}
		\includegraphics[width=0.31\linewidth]{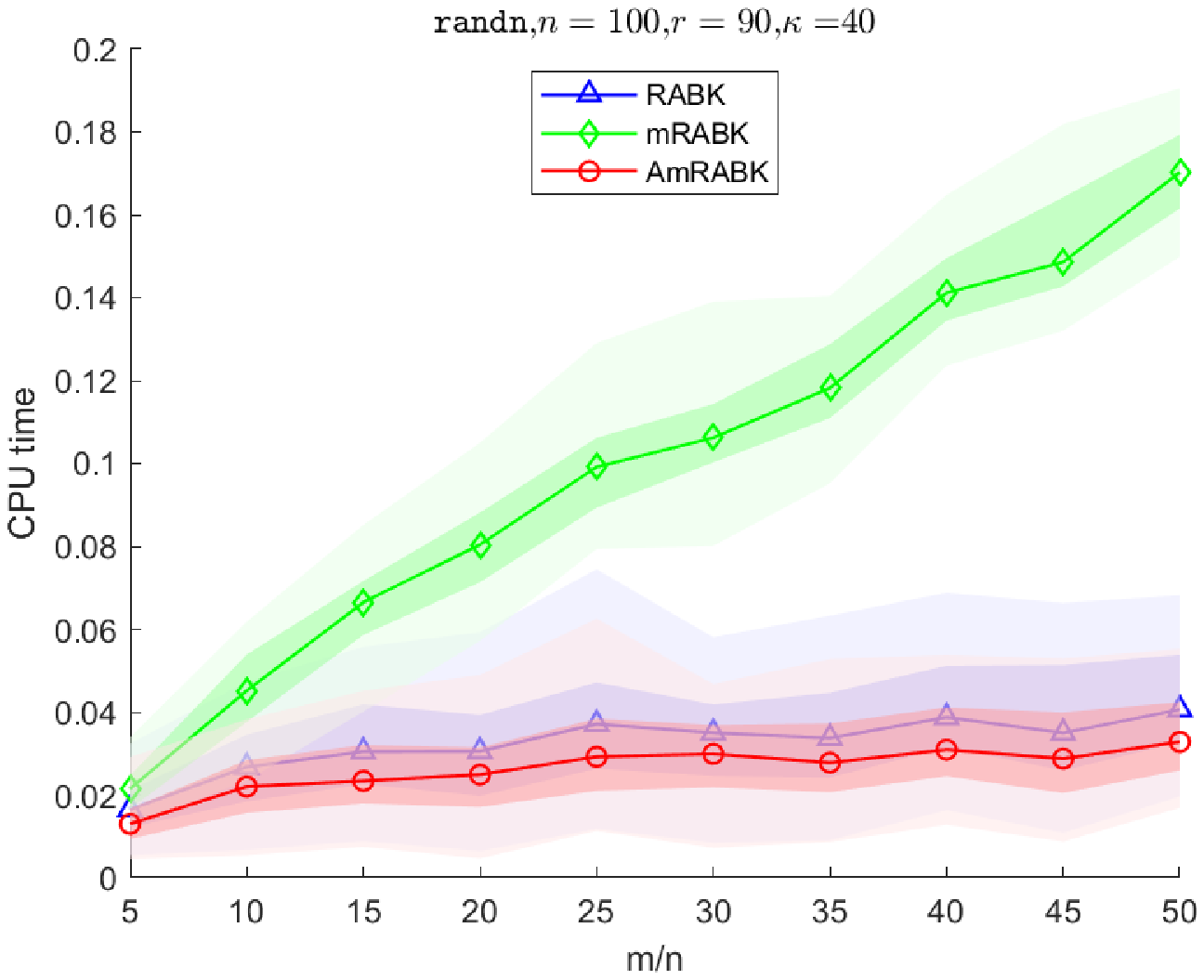}
	\end{tabular}
	\caption{Performance of RABK, mRABK, and AmRABK for linear systems with rank-deficient Gaussian matrix.
		Figures depict the iteration and the CPU time (in seconds) vs increasing number of rows.  The title of each plot indicates the values of $n$, $r$, and $\kappa$. We set $p=30$ and $\beta=0.7$ for mRABK. All computations are terminated once $\text{RSE}<10^{-12}$. }
	\label{figure2}
\end{figure}

Table \ref{table1} and Figure \ref{figure3} present the number of  iterations and the computing time for RABK, mRABK, and AmRABK when applied to sparse matrices from SuiteSparse Matrix Collection \cite{Kol19} and LIBSVM \cite{chang2011libsvm}. The nine matrices from SuiteSparse Matrix Collection are {\tt bibd\_16\_8}, {\tt crew1}, {\tt WorldCities}, {\tt nemsafm}, {\tt model1}, {\tt ash958}, {\tt ch8\_8\_b1}, {\tt Franz1}, and {\tt mk10-b2}, and the four matrices from LIBSVM are {\tt a9a}, {\tt aloi}, {\tt cod-rna}, and {\tt ijcnn1}, some of which are full rank while the others are rank-deficient.

\begin{table}
	\renewcommand\arraystretch{1.5}
	\setlength{\tabcolsep}{2pt}
	\caption{ The average Iter and CPU of RABK, mRABK, and AmRABK for linear systems with coefficient matrices from SuiteSparse Matrix Collection \cite{Kol19}. We set $p=30$ and the appropriate momentum
		parameters $\beta$ for mRABK are also listed. All computations are terminated once $\text{RSE}<10^{-12}$.}
	\label{table1}
	\centering
	{\scriptsize
		\begin{tabular}{  |c| c| c| c| c |c |c |c |c| c|c|  }
			\hline
			\multirow{2}{*}{ Matrix}& \multirow{2}{*}{ $m\times n$ }  &\multirow{2}{*}{rank}& \multirow{2}{*}{$\frac{\sigma_{\max}(A)}{\sigma_{\min}(A)}$}  &\multicolumn{2}{c| }{RABK}  &\multicolumn{3}{c| }{mRABK} &\multicolumn{2}{c| }{AmRABK}
			\\
			\cline{5-11}
			& &   &    & Iter & CPU    & Iter & CPU &$\beta$  & Iter & CPU     \\
			\hline
			{\tt bibd\_16\_8}& $120\times12870$ &  120  & 9.54 &        1052.50&   0.6871&   381.12&  0.2667 &0.70&  252.94 &  {\bf 0.2235}   \\
			\hline
			{\tt crew1} & $135\times6469 $ &  135  &18.20  & 1346.84&   0.1978&  1145.22&  0.2217 &0.80&  718.28 &  {\bf0.1171} \\
			\hline
			{\tt WorldCities} & $315\times100$ &  100  &6.60  & 10990.22&   0.0198&  5605.00&  0.0108 &0.90& 2566.06 &  {\bf 0.0054}  \\
			\hline
			{\tt nemsafm} & $334\times 2348$ &   334  &4.77  &    10974.36&   0.0201& 28218.48&  0.0511 &0.50& 2595.38 &  {\bf 0.0056}  \\
			\hline
			{\tt model1} & $  362\times798 $ &  362  & 17.57 &4111.22&   0.0297& 12643.66&  0.0931 &0.50& 3005.20 &  {\bf 0.0221}\\
			\hline
			{\tt ash958} & $958\times292$ &  292  &3.20  & 423.14&   0.0018&   461.52&  0.0034 &0.60&  409.74 & {\bf 0.0018} \\
			\hline
			{\tt ch8\_8\_b1} & $1568\times64$ &   63  & 3.48e+14  &   65.98&   0.0013&    84.04&  0.0169 &0.30&   65.48 &  {\bf0.0010}   \\
			\hline
			{\tt Franz1} & $ 2240\times768 $ &  755  & 2.74e+15 & 2620.76&   {\bf0.0118}&  1963.64&  0.0373 &0.70& 2571.78 &   0.0123 \\
			\hline
			{\tt mk10-b2} & $  3150\times630 $ &  586 & 2.74e+15 &574.76&   0.0055&   586.60&  0.0405 &0.40&  573.96 &  {\bf0.0054} \\
			\hline
		\end{tabular}
	}
\end{table}

\begin{figure}[tbhp]
	\centering
	\begin{tabular}{cc}
		\includegraphics[width=0.31\linewidth]{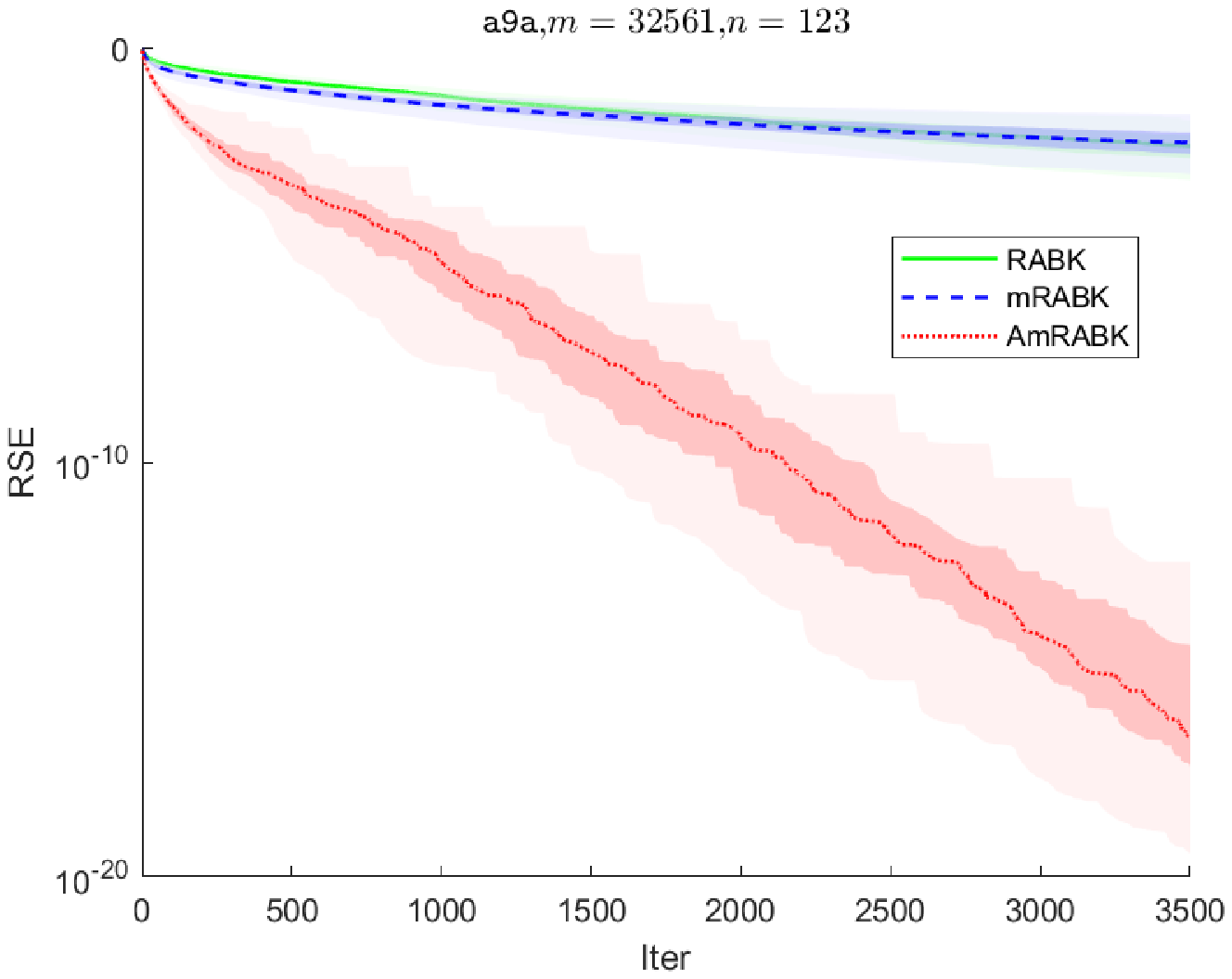}
		\includegraphics[width=0.31\linewidth]{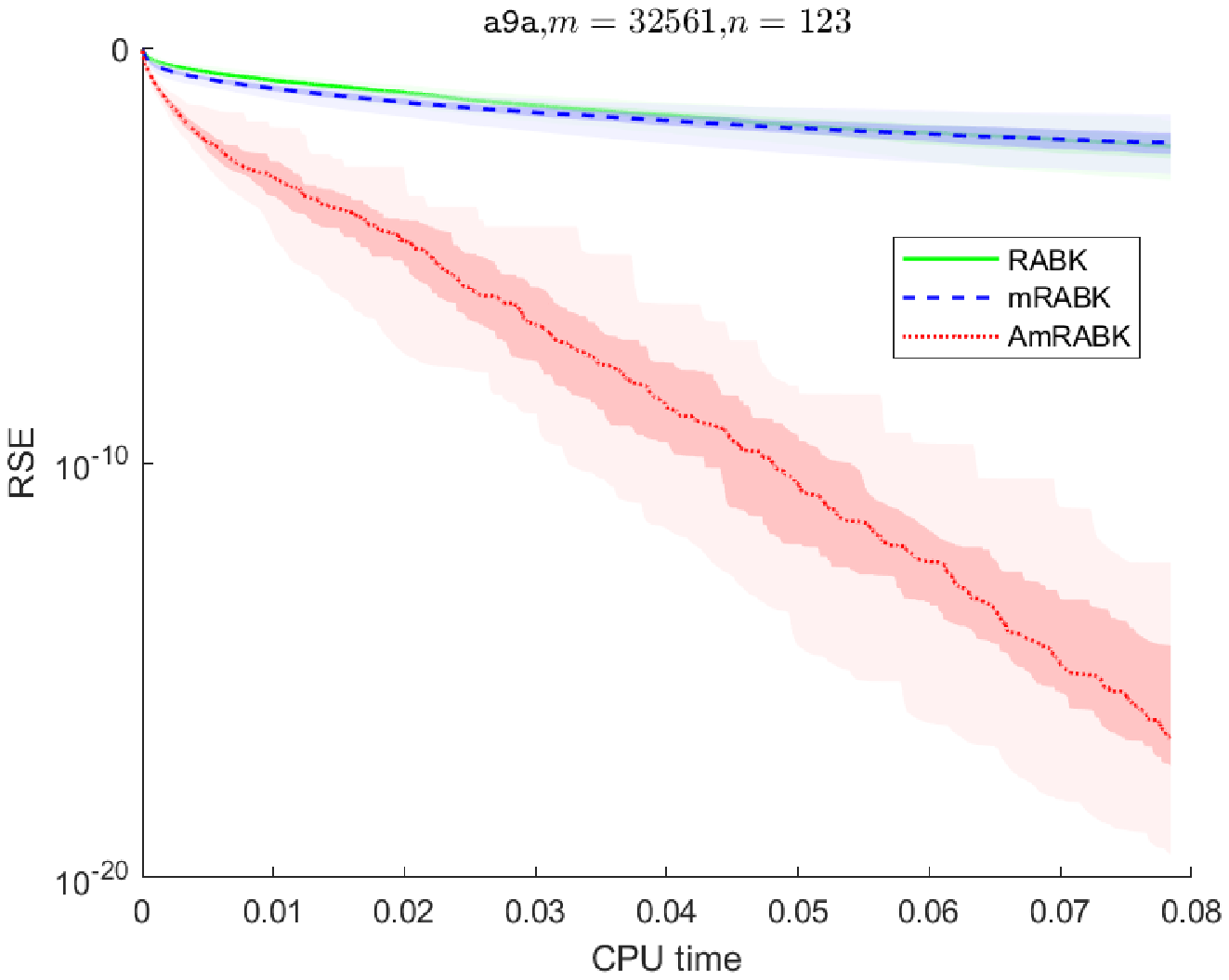}\\
		\includegraphics[width=0.31\linewidth]{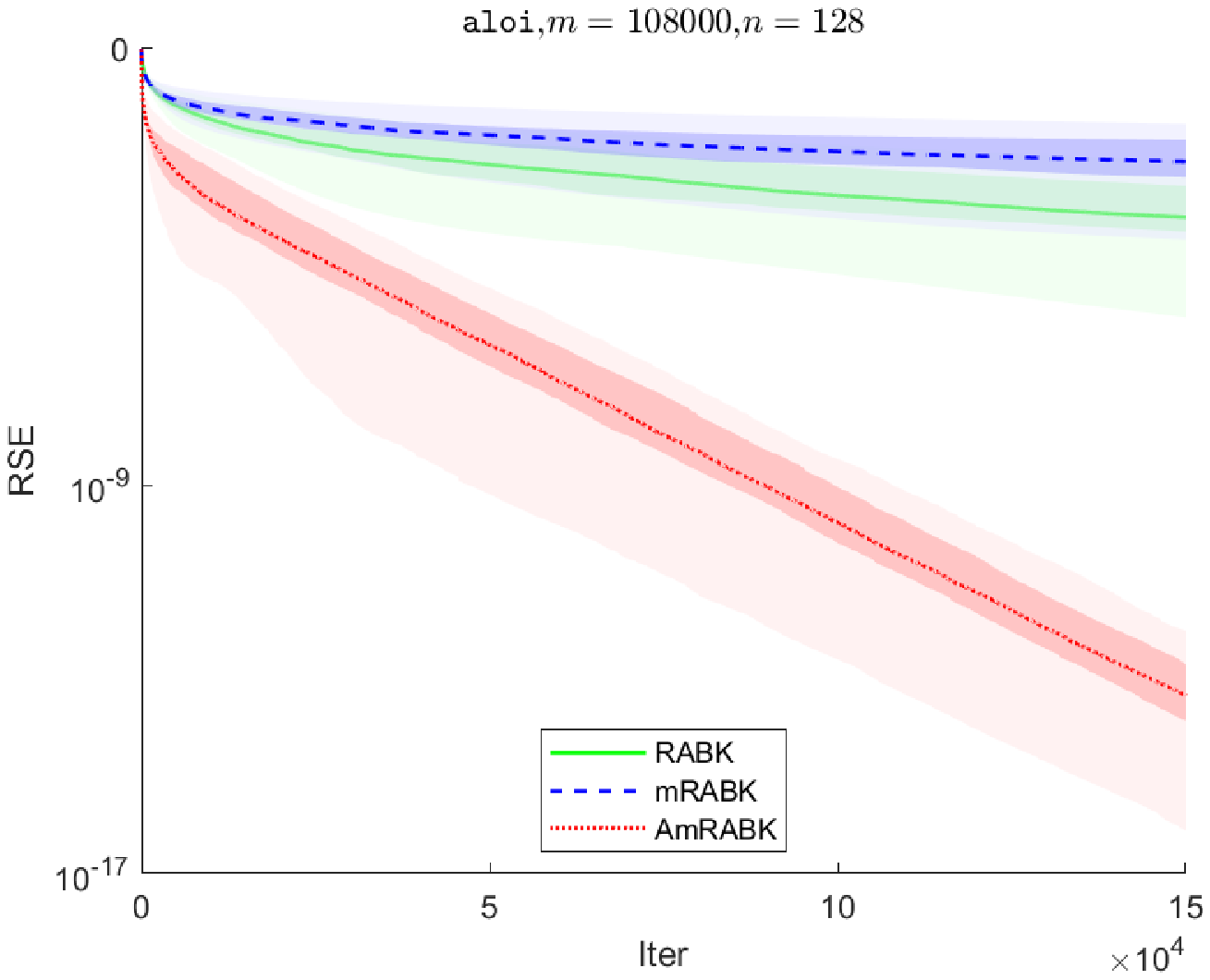}
		\includegraphics[width=0.31\linewidth]{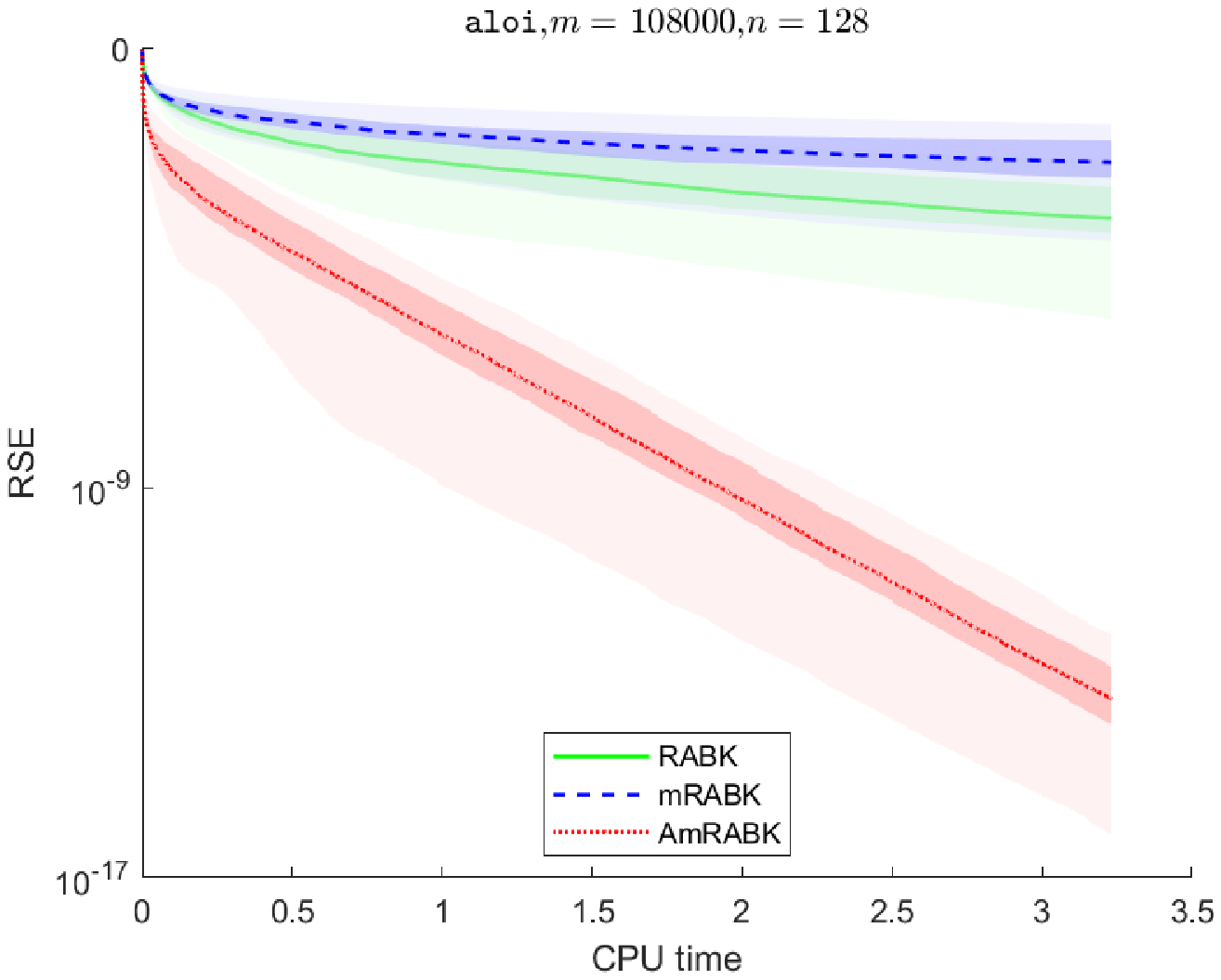}\\
		\includegraphics[width=0.31\linewidth]{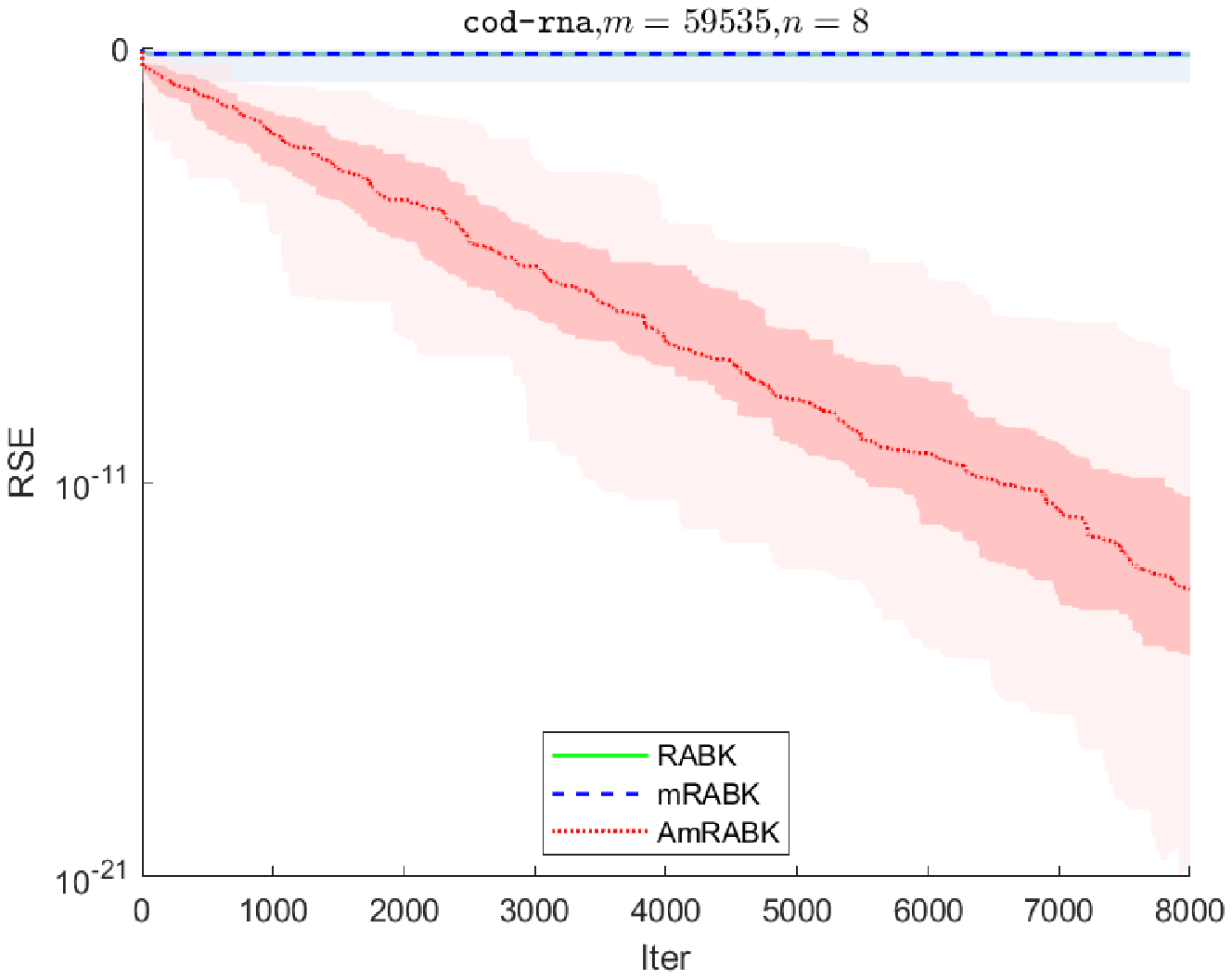}
		\includegraphics[width=0.31\linewidth]{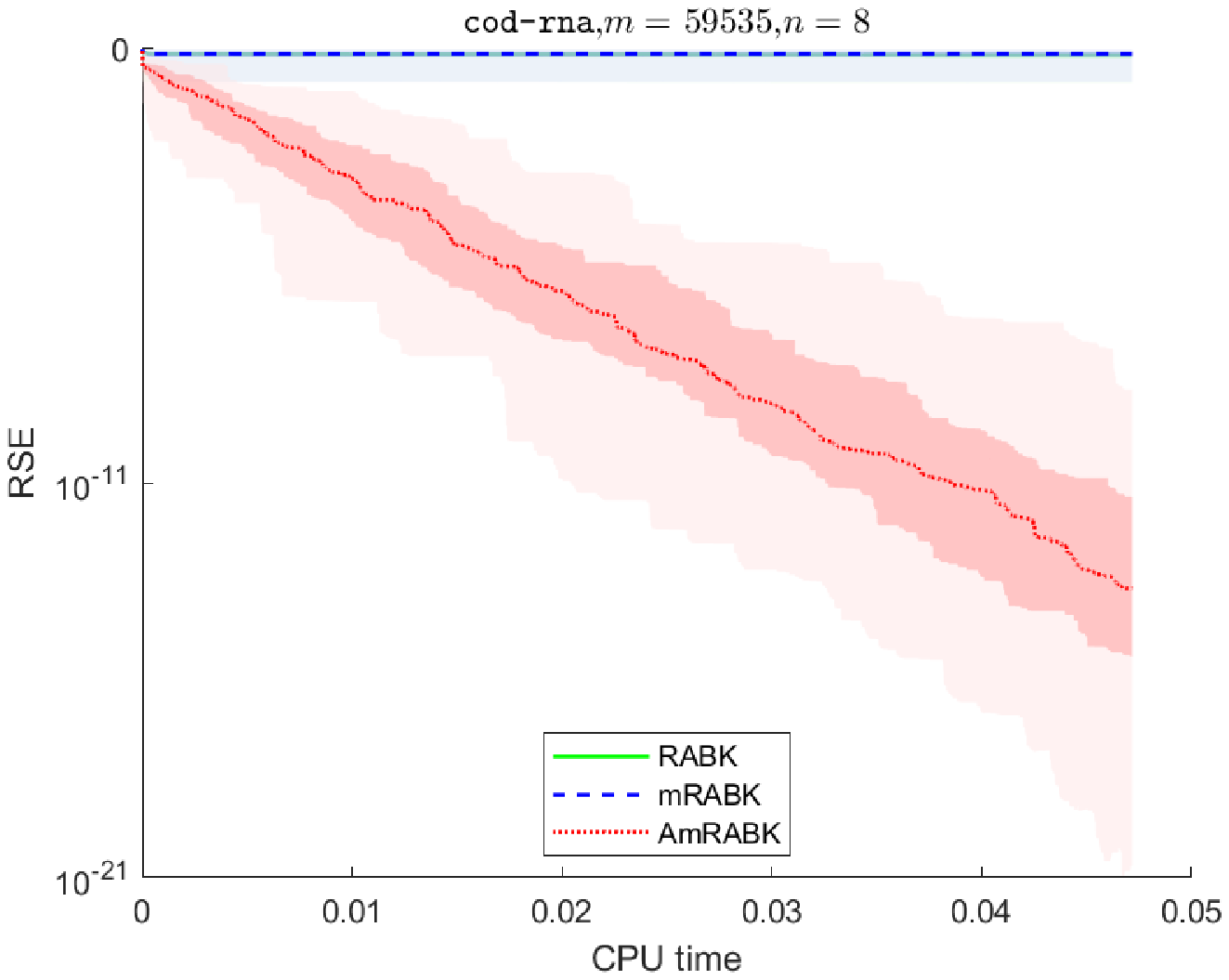}\\
		\includegraphics[width=0.31\linewidth]{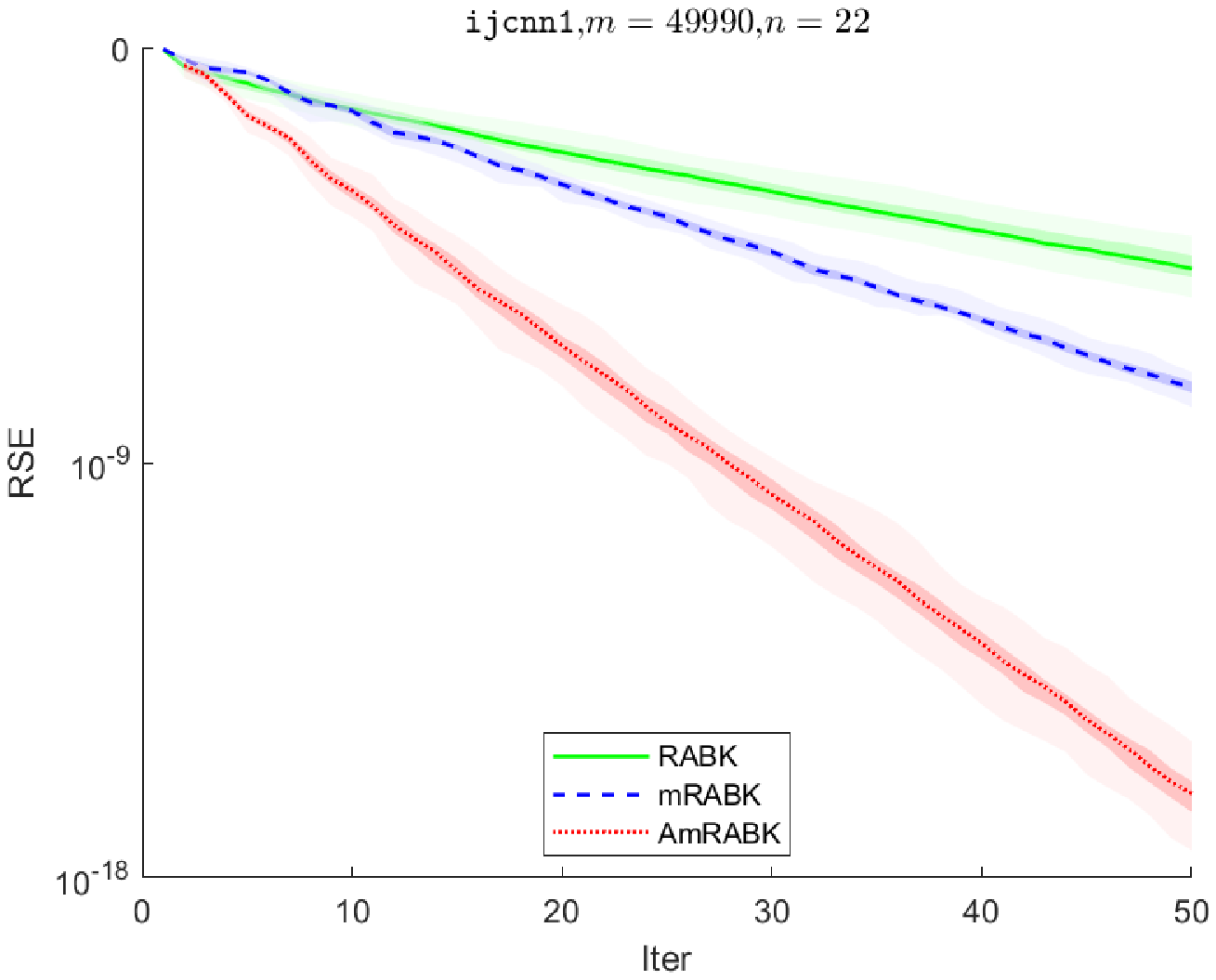}
		\includegraphics[width=0.31\linewidth]{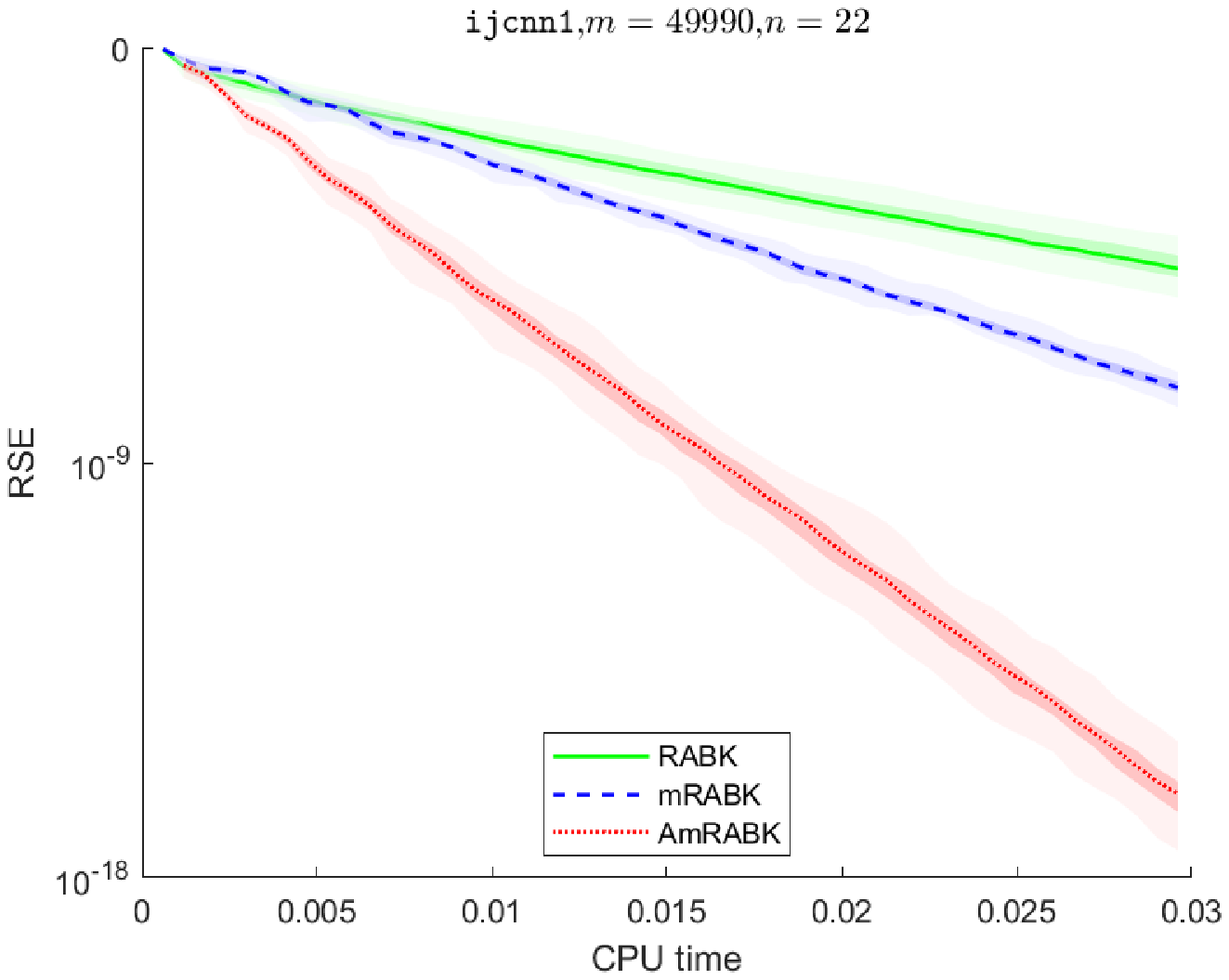}
	\end{tabular}
	\caption{Performance of RABK, mRABK, and AmRABK for linear systems with coefficient matrices from  LIBSVM \cite{chang2011libsvm}.
		Figures depict the evolution of RSE with respect to the number of iterations and the CPU time (in seconds).  
		The title of each plot indicates the names and sizes of the data. The appropriate momentum parameters $\beta$ for mRABK are $\beta=0.9$ for {\tt a9a}, $\beta=0.8$ for {\tt aloi}, $\beta=0.9$ for {\tt cod-rna}, and $\beta=0.7$ for {\tt ijcnn1}. We set the block size $p=300$ and stop the
		algorithms if the number of iterations exceeds a certain limit. 
	}
	\label{figure3}
\end{figure}

Table \ref{table1} presents the numerical results on the matrices from SuiteSparse Matrix Collection \cite{Kol19}.
As depicted in Table \ref{table1}, AmRABK consistently outperforms RABK in terms of iterations across all test cases. However, we note that RABK  may outperform AmRABK in terms of CPU time in specific cases, such as {\tt Franz1}. This is because  AmRABK requires more computations at each step compared to RABK.
Furthermore, Table \ref{table1} also presents that  AmRABK  performs better than mRABK in terms of CPU time across all test cases. The mRABK method and the RABK method may exhibit variable performance, with instances where mRABK performs well and others where RABK outperforms it.

Figure \ref{figure3} presents the numerical results on the matrices from LIBSVM \cite{chang2011libsvm}. 
We note that the CPU time used for obtaining the parameter $\alpha_k$ in  mRABK is excluded from these results. 
Nevertheless, our results demonstrate that AmRABK still outperforms mRABK. In addition, for dataset {\tt cod-rna}, it can be observed that both RABK and mRABK fail, while only AmRABK succeeds in finding the solution.

\subsection{Comparison to {\tt pinv} and {\tt lsqminnorm}}
In this subsection,  we compare the performance of  AmRABK with  {\sc Matlab} functions  {\tt pinv} and {\tt lsqminnorm}.
To obtain the least-norm solution easily,  we use the randomly generated matrices with full column rank, where $m\geq n$ and $r=n$. Then, we generate the solution vector $x^*$ by setting $x^*={\tt randn(n,1)}$, and calculate $b=Ax^*$. It can be guaranteed that $x^*$ is the desired unique solution of the constructed linear system.

We terminate AmRABK if the accuracy of its approximate solution is comparable to that of the approximate solution obtained by {\tt pinv} and {\tt lsqminnorm}. 
Figure \ref{RfigureM1} shows the CPU time against the increasing number of rows with number of columns fixed. It is noteworthy that  AmRABK outperforms {\tt pinv} and {\tt lsqminnorm}, when the number of rows exceeds certain thresholds. Moreover, the performance of  AmRABK  is more sensitive to the increase in the condition number $\kappa$, as suggested by the convergence bound \eqref{upper-bound}.

\begin{figure}[tbhp]
	\centering
	\begin{tabular}{cc}
		\includegraphics[width=0.31\linewidth]{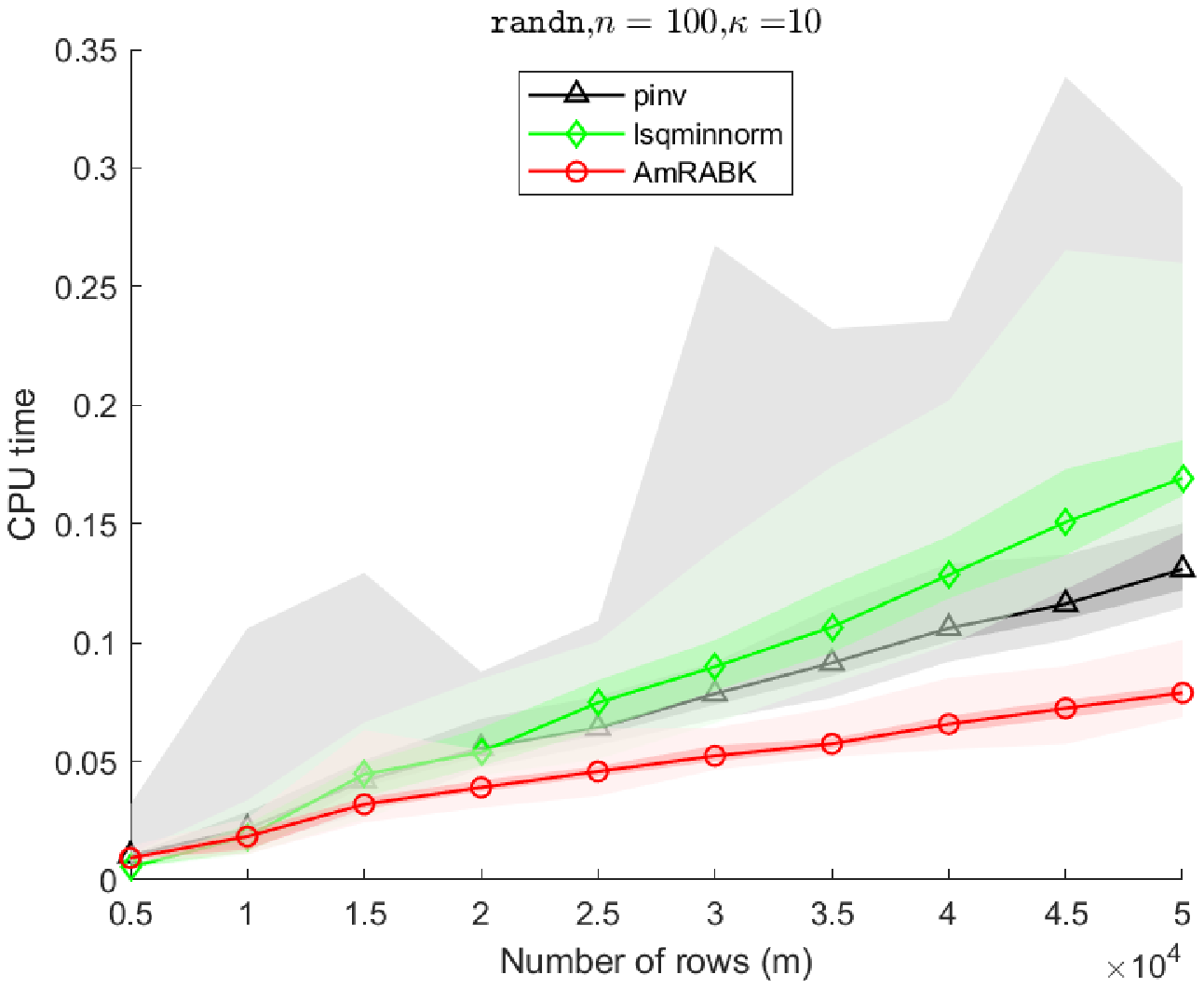}
		\includegraphics[width=0.31\linewidth]{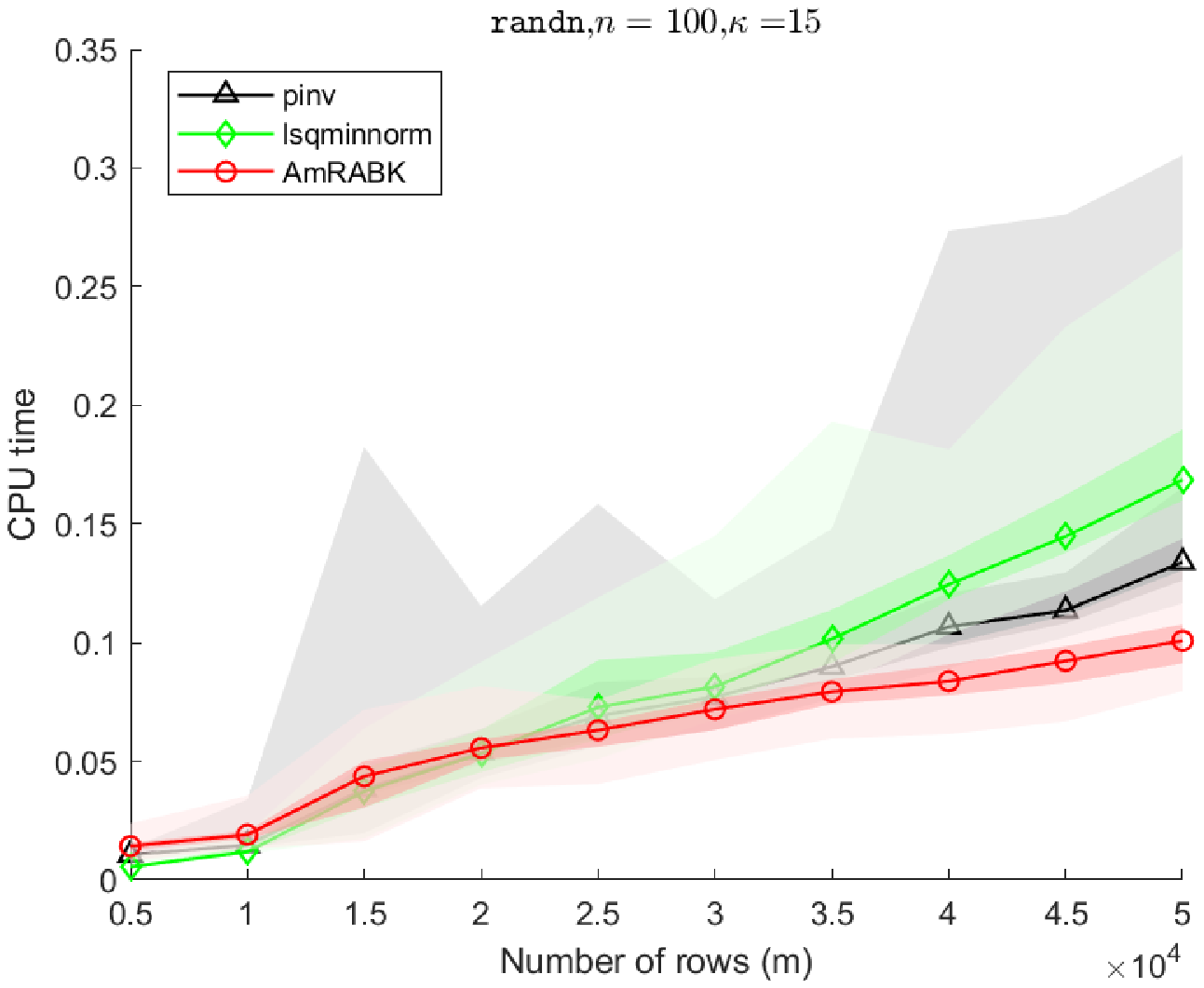}
		\includegraphics[width=0.31\linewidth]{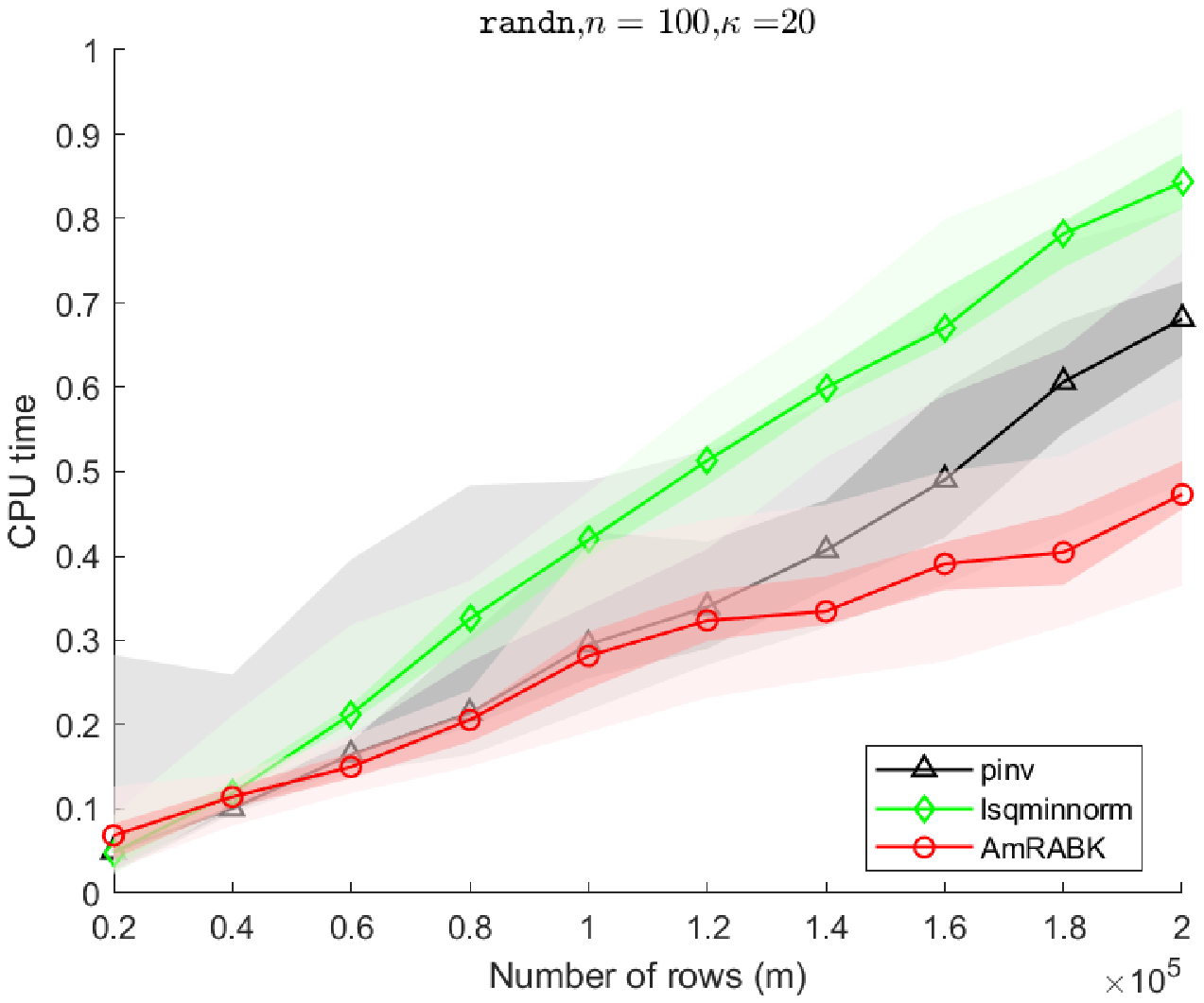}
	\end{tabular}
	\caption{Figures depict the CPU time (in seconds) vs increasing number of rows. The title of each
		plot indicates the values of $n$ and $\kappa$.  We set $p=30$. }
	\label{RfigureM1}
\end{figure}

\section{Concluding remarks}

We have established an adaptive stochastic heavy ball momentum (ASHBM) method, whose parameters do not rely on prior knowledge, for solving  stochastic problems  reformulated from linear systems.  Theoretically, the convergence factor of the ASHBM method is better than that of the basic method (SGD with stochastic Polyak step-size). The deterministic version of our method is serendipitously equivalent to  the conjugate gradient normal equation error (CGNE) method. The ASHBM method was further generalized to obtain a  completely new framework of stochastic conjugate gradient (SCG) method. Numerical results confirmed the efficiency of the ASHBM method.

The linear systems arising in practical problems are very likely to be inconsistent due to noise, which contradicts the basic assumption in this paper.  It should  be a valuable topic to explore the extension of the ASHBM method to handle inconsistent linear systems.
Moreover, we note that the CG method applied to the normal equation $A^\top Ax=A^\top b$ would result in the \emph{conjugate gradient normal equation residual} (CGNR) method \cite[Section 11.3.9] {golub2013matrix}.  It is also a valuable topic to investigate the relationship between the CGNR method and the randomized coordinate descent (RCD) method \cite{Lev10} with adaptive heavy ball momentum.

\bibliographystyle{plain}
\bibliography{zeng2023}

\end{document}